\documentclass[11pt]{article}

\usepackage{latexsym}
\usepackage{epic,eepic,color}
\usepackage{graphicx}
\usepackage{graphics}
\usepackage{amsmath}
\usepackage{amssymb}
\usepackage{multirow}
\usepackage{psfrag}
\usepackage{subfigure}

\def\rcr{\overline{\hbox{\rm cr}}}

\def\alli{{\overline{\ell_i'}}}
\def\elli{\overline{{\overline{\ell_i'}}}}
\def\ellj{\overline{{\overline{\ell_j'}}}}

\def\tfl#1{{\lfloor{#1}\rfloor}}
\def\tcl#1{{\lceil{#1}\rceil}}

\def\ll{{\cal L}}
\def\rr{{\cal R}}

\def\hal{\vrule height 5 pt width 0.05 in  depth 0.8 pt}

\newtheorem{theorem}{Theorem}

\newtheorem{conjecture}{Conjecture}

\newtheorem{lemma}{Lemma}

\newtheorem{proposition}{Proposition}
\newtheorem{remark}{Remark}

\newenvironment{proof}[1][Proof]{\noindent\textbf{#1.} }{\ \rule{0.5em}{0.5em}}

\def\rcr{\overline{\hbox{\rm cr}}}

\def\alli{{\overline{\ell_i'}}}
\def\elli{\overline{{\overline{\ell_i'}}}}
\def\ellj{\overline{{\overline{\ell_j'}}}}

\def\tfl#1{{\lfloor{#1}\rfloor}}
\def\tcl#1{{\lceil{#1}\rceil}}

\begin{document}

\title{$3$-symmetric and $3$-decomposable geometric drawings of
  $K_{n}$ (extended version)\footnote{This extended version contains
    an Appendix not included in the original version.}}
\author{B.M.~\'{A}brego\footnote{Department of Mathematics. California
    State University, Northridge. Northridge, CA 91330.}
\\ M.~Cetina\footnote{Instituto de F\'\i sica, Universidad Aut\'onoma
  de San Luis Potos\'{\i}. San Luis Potos\'{\i}, SLP, Mexico 78000.}
\\ S.~Fern\'{a}ndez-Merchant${}^{\dag}$ \\
J.~Lea\~{n}os\footnote{Universidad Aut\'onoma de Zacatecas, Campus
  Jalpa. Jalpa, Zacatecas, Mexico 99600.} \\  
G.~Salazar\footnote{Instituto de F\'\i sica, Universidad Aut\'onoma
  de San Luis Potos\'{\i}. San Luis Potos\'{\i}, SLP, Mexico
  78000. Supported by CONACYT Grant 45903 and by FAI--UASLP.}}
\maketitle

\begin{abstract}

Even the most superficial glance at the vast majority of crossing-minimal geometric drawings of $K_n$ reveals two hard-to-miss features. First, all such drawings appear to be $3$-fold symmetric (or simply {\em $3$-symmetric}) . And second, they all are {\em $3$-decomposable}, that is, there is a triangle $T$ enclosing the drawing, and a balanced partition $A, B, C$ of the underlying set of points $P$, such that the orthogonal projections of $P$ onto the sides of $T$ show $A$ between $B$ and $C$ on one side, $B$ between $A$ and $C$ on another side, and $C$ between $A$ and $B$ on the third side. In fact, we conjecture that all optimal drawings are $3$-decomposable, and that there are $3$-symmetric optimal constructions for all $n$ multiple of $3$. In this paper, we show that any $3$-decomposable geometric drawing of $K_n$ has at least $0.380029\binom{n}{4}+\Theta(n^3)$ crossings. On the other hand, we produce $3$-symmetric and $3$-decomposable drawings that improve the {\em general} upper bound for the rectilinear crossing number of $K_n$ to $0.380488\binom{n}{4}+\Theta(n^3)$. We also give explicit $3$-symmetric and $3$-decomposable constructions for $n<100$ that are at least as good as those previously known.

\end{abstract}

\section{Introduction}

For a finite set of points $P$ in general position in the plane, let $\overline{\hbox{\rm cr}}\left( P\right) $ denote the number of crossings in the complete \emph{geometric graph} with vertex set $P$, that is, the complete graph whose edges are straight line segments. It is an elementary observation that $\rcr(P)$ equals $\square(P)$, the number of convex quadrilaterals defined by points in $P$. If $P$ has $n$ vertices, the complete geometric graph with vertex set $P$ is also called a \emph{rectilinear drawing} of $K_{n}$. The \emph{rectilinear crossing number} of $K_{n}$, denoted $\rcr(K_{n})$, is the minimum number of crossings in a rectilinear drawing of $K_{n}$. That is, $\rcr(K_{n})=\min_{|P|=n} \rcr(P)$, where the minimum is taken over all $n$-point sets $P$ in general position in the plane. Determining $\rcr(K_{n})$ is a well-known problem in combinatorial geometry posed by Erd\H{o}s and Guy \cite{erdguy}.

Figure \ref{fig:3decomp1}(a) shows the point set of an optimal (crossing minimal) rectilinear drawing of $K_{18}$ (drawing by O.~Aichholzer and H.~Krasser, taken with permission from~\cite{aweb2}).  This drawing exhibits a natural partition of the $18$ vertices into $3$ clusters of $6$ vertices each, with two prominent features: (i) rotating any cluster angles of $2 \pi /3$ and $4 \pi /3$ around a suitable point, one obtains point sets highly resembling the other two clusters; and (ii) the orthogonal projections of these clusters on the sides of an enclosing triangle, have each projected cluster separating the other two. A similar structure is observed in {\em every} known optimal drawing of $K_{n}$, for every $n$ multiple of $3$, perhaps after an order-type preserving transformation (see~\cite{k27,aweb2}). Even the best available examples for $n>27$, i.e., for those values of $n$ for which the exact value of $\rcr(K_n)$ is still unknown, share this property \cite{aweb2}.

\begin{figure}[htbp]
\begin{center}
\includegraphics[height=2truein]{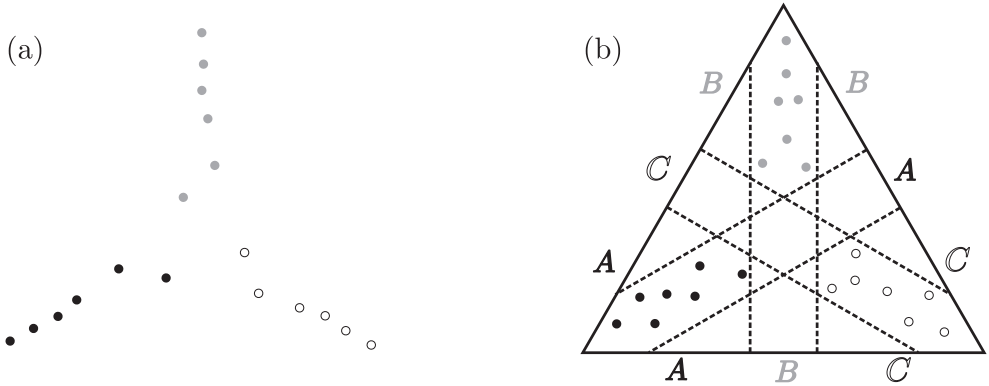}
\caption{(a) An optimal geometric drawing of $K_{18}$. (b) The drawing in (a) is $3$-decomposable.}
\label{fig:3decomp1}
  \end{center}
\end{figure}

To further explore the  distinguishing features of these drawings, we introduce the concepts of $3$-{\em symmetry} and $3$-{\em decomposability}. A geometric drawing of $K_n$ is {\em $3$-symmetric} if its underlying point set  $P$ is partitioned into three {\em wings} of size $n/3$ each, with the property that rotating each wing angles of $2 \pi /3$ and $4 \pi /3$ around a suitable point generates the other two wings.  We also say that $P$ itself is $3$-{\em symmetric}. Now $3$-decomposability is a subtler, yet structurally far more significant, property that has to do with the relative orientation of the points of three $(n/3)$-point subsets of an $n$-point set (the wings, if the point set is also $3$-symmetric).
 A  finite point set $P$ is $3$\emph{-decomposable} if it can be partitioned into three equal-size sets $A$, $B$, and $C$ satisfying the following: there is a triangle $T$ enclosing $P$ such that the orthogonal projections of $P$ onto the the three sides of $T$ show $A$ between $B$ and $C$ on one side, $B$ between $A$ and $C$ on another side, and $C$ between $A$ and $B$ on the third side. We  say that a geometric drawing of $K_n$ is $3${\em -decomposable} if its underlying point set is $3$\emph{-decomposable}.
We note that whenever we speak of a $3$-decomposable or $3$-symmetric drawing of $K_n$, it is implicitly assumed that $n$ is a multiple of $3$.

In this paper, we report our recent research on $3$-decomposable and $3$-symmetric drawings. We have derived a lower bound for the number of crossings in $3$-decomposable geometric drawings.

\begin{theorem}
\label{thm:main} Let $P$ be a $3$-decomposable set of $n$ points. Then
\begin{equation*}
\overline{\hbox{\rm cr}}(P)\geq \frac{2}{27}\left( 15-\pi ^{2}\right) \binom{%
n}{4}+\Theta (n^{3})>0.380029\binom{n}{4}+\Theta (n^{3}).
\end{equation*}
\end{theorem}

Recall that a $(\leq k)$-\emph{set} of a point set $P$ is a subset of $P$ with at most $k$ elements that can be separated from the rest of $P$ by a
straight line. The number ${\chi _{\leq k}(P)}$ of $(\leq k)$-sets of $P$ is a parameter of independent interest in discrete geometry \cite{bmp}. In
Section \ref{sec:proofmain}, we prove Theorem \ref{thm:main} making use of the close relationship between rectilinear crossing numbers and $(\leq k)$-sets, unveiled independently by \'{A}brego and Fern\'{a}ndez-Merchant \cite{af} and by Lov\'{a}sz et al. \cite{lovasz}:

\begin{equation}
\rcr(P)=\sum_{k=1}^{(n-2)/2}(n-2k-1){\chi _{\leq k}(P)}+\Theta (n^{3}).
\label{eq:aflov}
\end{equation}

Besides Equation ~\ref{eq:aflov}, the main ingredient in the proof of Theorem \ref{thm:main} is the following bound for the number of $(\leq k)$-sets in $3$-decomposable point sets, whose proof appears in Section \ref{sec:proofmainksets}.

\begin{theorem}
\label{thm:mainksets} Let $P$ be a $3$-decomposable set of $n$ points, where $n$ is a multiple of $3$, and let $k<n/2$. Then
\begin{equation*}
{\chi _{\leq k}(P)}\geq B\left( k,n\right) ,
\end{equation*}
where%
\begin{equation}
B\left( k,n\right) {:}=3\binom{k+1}{2}+3\binom{k+1-n/3}{2}
+3\sum_{j=2}^{s-1}j(j+1)\binom{k+1-c_{j}n}{2},  \label{eq:seriesbinomials}
\end{equation}
$c_{j}:=\frac{1}{2}-\frac{1}{3j(j+1)}$, and $s:=s(k,n)$ is the unique integer such that
$\binom{s}{2}<\frac{n}{3\left( n-2k-1\right) }\leq \binom{s+1}{2}$.
\end{theorem}

(In case $r$ is not an integer, we use the formal definition $\binom{r}{2}=\frac{r\left( r-1\right) }{2}$. Also, by convention, $\binom{r}{2}=0$
if $r<2 $.)

To improve the \emph{general} upper bound on the number of crossings, we developed a procedure that grows a \emph{base drawing} of a given $K_{m}$ into a so called \emph{augmenting drawing} of $K_{n}$ for some $n>m$. This method is of interest by itself as it preserves certain structural properties that guarantee a relatively small number of crossings in the augmenting drawing. It refines previous constructions by Brodsky et al.~\cite{brodsky}, Aichholzer et al.~\cite{aak}, and \'Abrego and Fern\'andez-Merchant~\cite{af1}. Section~\ref{gencon} is devoted to the description and analysis of our replacing-by-clusters construction. Iterating this procedure, using as initial base drawing any complete geometric graph with an \emph{odd} number of points, yields the following result proved in Section~\ref{doubling}.

\begin{theorem}\label{baseoddite}
If $P$ is an $m$-element point set in general position, with $m$ odd, then
\begin{equation}\label{eq:recursion}
\rcr(K_n) \le \frac{24\rcr(P) + 3m^3 - 7m^2 + (30/7)m}{m^4} \binom{n}{4} + \Theta(n^3).
\end{equation}
\end{theorem}

This inequality was previously known (Theorem 2 in \cite{af1}) only
for drawings with an \emph{even} number of points, and with a base
drawing that satisfies a certain ``halving property''. The existence
of a point set satisfaying such halving property together this theorem
constitute the best tools available to obtain upper bounds for the
rectilinear crossing number constant. In fact, we have produced a
geometric drawing of $K_{315}$ with $152210640$ 
crossings (see Section~\ref{symdra}), that used as the base drawing in Theorem~\ref{baseoddite}, yields the best upper bound currently known for the {\em rectilinear crossing number constant} $q_*:=\lim_{n\to\infty}{\rcr(K_n)}/{\binom{n}{4}}$.

\begin{theorem}\label{bestupper}
The rectilinear crossing number constant $q_*$ satisfies $q_* \leq \frac{83247328}{218791125}  < 0.380488$.
\end{theorem}

The previously best known general bounds for the rectilinear crossing number of $K_{n}$ are $0.379972\binom{n}{4}+\Theta (n^{3}) <
\rcr(K_n) < 0.38054415\binom{n}{4}+\Theta (n^{3})$; see\ \cite{central} for the lower bound, and \cite{af1} with a drawing of $K_{90}$ with $951526$ crossings by Aichholzer for the upper bound. Thus the general upper bound in Theorem~\ref{bestupper}, together with the lower bound given by Theorem~\ref{thm:main}, closes this gap by close to 20\%, under the quite feasible assumption of $3$-decomposability. In fact, we strongly believe that:

\begin{conjecture}\label{conjalldecomp}
For each positive integer $n$ multiple of $3$, all optimal rectilinear drawings of $K_{n}$ are $3$-decomposable.
\end{conjecture}

The reasons for this belief go beyond the evidence of all known optimal drawings: the underlying point sets of all the best crossing-wise known drawings of $K_n$ happen to minimize the number of $(\le k)$-sets for every $k \le n/3$, and a point set with this property is in turn $3$-decomposable (an equivalent form of this statement appears in~\cite{agor2}; see also~\cite{agor3}).

Another strong feeling that we have is about the symmetry.  
We note that {\em none} of the explicit best known constructions, prior to this paper, is $3$-symmetric (except for some very small values of $n$). Yet, they resemble a $3$-symmetric set. This hints to the existence of equally good drawings of $K_n$ that are $3$-symmetric (which seems to be a wide spread belief). In this context we believe that:

\begin{conjecture}\label{conjallsym}
For each positive integer $n$ multiple of $3$, there is an optimal geometric drawing of $K_n$ that is $3$-symmetric.
\end{conjecture}

Our main findings back up Conjectures \ref{conjalldecomp} and
\ref{conjallsym}. Indeed, we have found, for every $n$ multiple of
$3$, a $3$-decomposable and $3$-symmetric geometric drawing of $K_n$
with the fewest number of crossings known to date. Thus, in
particular, for each $n$ multiple of $3$ for which the exact value of
$\rcr(K_n)$ is known (that is, $n \le 27$), we have found an optimal
geometric drawing that is $3$-decomposable and $3$-symmetric. These
drawings are described in Section~\ref{symdra}. Some were obtained
using heuristic methods based of previously known constructions; the
rest were obtained applying our replacing-by-clusters construction
from Section~\ref{gencon}, with base drawings of $K_{30}$ or $K_{51}$.
In fact, this drawing of $K_{315}$ is 
obtained from a base drawing of $K_{51}$, 
and it is the initial base drawing  used to establish Theorem \ref{bestupper}.


\section{Proof of Theorem \protect\ref{thm:main}\label{sec:proofmain}}

Let $P$ be a $3$-decomposable set of $n$ points in general position.
Combining Theorem~\ref{thm:mainksets} and Equation \ref{eq:aflov}, and
noting that the $-1$ in the factor $n-2k-1$ only contributes to smaller
order terms, we obtain

\begin{align*}
{\overline{\hbox{\rm cr}}}(P)& {\geq }\sum_{k=1}^{(n-2)/2}\left( n-2k\right)
B(k,n){+}\Theta (n^{3}) \\
& =36\binom{n}{4}\left( \sum_{k=1}^{(n-2)/2}\frac{1}{n}\left( 1-2\left(
\frac{k}{n}\right) \right) \left( \frac{k}{n}\right)
^{2}+\sum_{k=n/3}^{(n-2)/2}\frac{1}{n}\left( 1-2\left( \frac{k}{n}\right)
\right) \left( \frac{k}{n}-\frac{1}{3}\right) ^{2}\right. \\
& \ \hspace{1in}+\left. \sum_{k=1}^{(n-2)/2}\sum_{j=2}^{s-1}j(j+1)\frac{1}{%
n}\left( 1-2\left( \frac{k}{n}\right) \right) \left( \frac{k}{n}%
-c_{j}\right) ^{2}\right) +\Theta (n^{3})\text{,}
\end{align*}
since $j\leq s(k,n)-1$ if and only if $k>c_{j}n-1/2$, then%
\begin{align*}
{\overline{\hbox{\rm cr}}}(P)& {\geq }36\binom{n}{4}\left(
\sum_{k=1}^{(n-2)/2}\frac{1}{n}\left( 1-2\left( \frac{k}{n}\right) \right)
\left( \frac{k}{n}\right) ^{2}+\sum_{k=n/3}^{(n-2)/2}\frac{1}{n}\left(
1-2\left( \frac{k}{n}\right) \right) \left( \frac{k}{n}-\frac{1}{3}\right)
^{2}\right. \\
& \ \hspace{1in}+\left. \sum_{j=2}^{\infty }j(j+1)\sum_{c_{j}n-1/2<k\leq
(n-2)/2}\frac{1}{n}\left( 1-2\left( \frac{k}{n}\right) \right) \left( \frac{k%
}{n}-c_{j}\right) ^{2}\right) +\Theta (n^{3}).
\end{align*}
Each of the sums is a Riemann Sum which we estimate using
the corresponding integrals. Note that all the error terms are
bounded by $\Theta (n^{3})$.
\begin{align*}\rcr (P) & \geq 36\binom{n}{4}\left(
\int_{0}^{1/2}(1-2x)x^{2}dx+\int_{1/3}^{1/2}(1-2x)\left( x-\frac{1}{3}%
\right) ^{2}dx\right. \\
& \ \hspace{1.5in}+\left. \sum_{j=2}^{\infty
}j(j+1)\int_{c_{j}}^{1/2}(1-2x)(x-c_{j})dx\right) +\Theta (n^{3}) \\
\ & =\binom{n}{4}\left( \frac{3}{8}+\frac{1}{216}+\frac{2}{27}%
\sum_{j=2}^{\infty }\frac{1}{j^{3}(j+1)^{3}}\right) +\Theta (n^{3}).  \hal
\end{align*}

Since
\begin{equation*}
\sum_{j=2}^{\infty }\frac{1}{j^{3}(j+1)^{3}}=\sum_{j=2}^{\infty }\left(
\frac{1}{j^{3}}-\frac{3}{j^{2}}+\frac{6}{j}-\frac{1}{(j+1)^{3}}-\frac{3}{%
(j+1)^{2}}-\frac{6}{j+1}\right) =\frac{79}{8}-\pi ^{2},
\end{equation*}%
then%
\begin{equation*}
\overline{\hbox{\rm cr}}(P)\geq \frac{2}{27}\left( 15-\pi ^{2}\right) \binom{%
n}{4}+\Theta (n^{3}).
\end{equation*}


\section{Proof of Theorem \protect\ref{thm:mainksets}\label{sec:proofmainksets}}

We follow the approach of allowable sequences. An \emph{allowable sequence} $%
\mathbf{\Pi }$ is a doubly infinite sequence $\ldots \pi _{-1},\pi _{0},\pi
_{1},\ldots $ of permutations of $n$ elements, where consecutive
permutations differ by a transposition of neighboring elements, and $\pi
_{i} $ is the reverse permutation of $\pi _{i+{\binom{n}{2}}}$. Then any
subsequence $\Pi $ of $\binom{n}{2}+1$ consecutive permutations in $\mathbf{%
\Pi }$ contains all necessary information to reconstruct the entire
allowable sequence. $\Pi $ is called a \emph{halfperiod} of $\mathbf{\Pi }$.

Our interest in allowable sequences derives from the fact that all the
combinatorial information of an $n$-point set $P$ can be encoded by an
allowable sequence $\mathbf{\Pi }_{P}$ on the set $P$, called the \emph{%
circular sequence} associated to $P$. A halfperiod $\Pi $ of $\mathbf{\Pi }%
_{P}$ is obtained as follows: Start with a circle $C$ containing $P$ in its
interior, and a tangent directed line $\ell $ to $C$. Project $P$
orthogonally onto $\ell $, and record the order of the points in $P$ on $%
\ell $. This will be the initial permutation $\pi _{0}$ of $\Pi $. (In the
remote case that two point-projections overlap, use a small rotation of $%
\ell $ on $C$.) Now, continuously rotate $\ell $ on $C$ (clockwise) and keep
projecting $P$ orthogonally onto $\ell $. Right after two points overlap in
the projection, say $p$ and $q$, the order of $P$ on $\ell $ will change.
This new order of $P$ on $\ell $ will be $\pi _{1}$. Note that $\pi _{1}$ is
obtained from $\pi _{0}$ by the transposition of $pq$. Continue doing this,
rotating $\ell $ on $C$ and recording the corresponding permutations of $P$,
until completing half a turn on $C$. At this time, the order of $P$ on $\ell
$ will be the reverse than the original. Moreover, exactly $\binom{n}{2}$
transpositions have taken place, one per each pair of points. The only thing
that we need to assume from $P$ for this to be well defined, is that any
two lines joining points in $P$ are not parallel. This can be done by
slightly perturbing the points of $P$ without changing its combinatorial
properties.

It is important to note that most allowable sequence are not circular
sequences. In fact, allowable sequence are in one-to-one correspondence with
generalized configurations of points. We refer the reader to the seminal
work by Goodman and Pollack~\cite{goodmanpollack} for further details.

Observe that if $P$ is $3$-decomposable with partition $A$, $B$, and $C$,
then there is a halfperiod $\Pi =(\pi _{0},\pi _{1},\ldots ,\pi _{\tbinom{n}{%
2}})$ of $\mathbf{\Pi }_{P}$ whose points can be labeled $A=\left\{
a_{1},\ldots ,a_{n/3}\right\} $, $B=\left\{ b_{1},\ldots
,b_{n/3}\right\} $, and $C=\left\{ c_{1},\ldots ,c_{n/3}\right\} $,
so that $\pi _{0}=(a_{1},a_{2},...,a_{n/3},b_{1},b_{2},\ldots
,b_{n/3},c_{1},c_{2},\ldots ,c_{n/3})$, and for some indices
$0<s<t\leq \tbinom{n}{2}$, $\pi _{s+1}$ shows all the $b$-elements
followed by all the $a$-elements followed by all the $c$-elements,
and $\pi _{t+1}$ shows all $b$-elements followed by all the
$c$-elements followed by all the $a$-elements. An allowable
sequence
with a halfperiod satisfying these properties is called $3$-\emph{%
decomposable}, generalizing the definition of $3$-decomposability from
point-sets to allowable sequences.

We have the following definitions and notation for allowable sequences. A
transposition that occurs between elements in sites $i$ and $i+1$ is an $i$-%
\emph{transposition}. For $i\leq n/2$, an $i$-\emph{critical} tranposition
is either an $i$-transposition or an $(n-i)$-transposition, and a $(\leq k)$-%
\emph{critical} transposition is a transposition that is $i$-critical for
some $i\leq k$. If $\Pi $ is a halfperiod, then ${N_{\leq k}(\Pi )}$ denotes
the number of $(\leq k)$-critical transpositions in $\Pi $. When $\mathbf{%
\Pi =\Pi }_{P}$ is a circular sequence associated to a point-set $P$, $(\leq
k)$-critical transpositions in $\mathbf{\Pi }$ correspond to $\left( \leq
k\right) $-sets of $P$. More precisely, if a permutation $\pi $ in $\mathbf{%
\Pi }_{P}$ is obtained by a $k$-transposition (similarly, by an $\left(
n-k\right) $-transposition) the first (similarly, last) $k$ elements in $\pi
$ form a $k$-set. Thus ${\chi _{\leq k}(P)}={N_{\leq k}(\Pi )}$ for any
halfperiod $\Pi $ of $\mathbf{\Pi }_{P}$.

The following theorem generalizes Theorem \ref{thm:mainksets}.

\begin{theorem}
\label{prop:main} Let $\Pi $ be a $3$-decomposable halfperiod on $n$ points,
and let $k<n/2$. Then
\begin{equation*}
{N_{\leq k}(\Pi )}\geq B\left( k,n\right) .
\end{equation*}
\end{theorem}

We devote the rest of this section to the proof of Theorem~\ref{prop:main}.

\subsection{Proof of Theorem~\protect\ref{prop:main}}

Throughout this section, $\Pi =(\pi _{0},\pi _{1},\ldots ,\pi _{\binom{n}{2}%
})$ is a $3$-decomposable halfperiod on $n$ points, with initial permutation
$\pi _{0}=(a_{1},\ldots ,a_{n/3},\ldots ,a_{1},b_{1},\ldots
,b_{n/3},c_{1},\ldots ,c_{n/3})$ and $A=\left\{ a_{1},\ldots
,a_{n/3}\right\} $, $B=\left\{ b_{1},\ldots ,b_{n/3}\right\} $, and $%
C=\left\{ c_{1},\ldots ,c_{n/3}\right\} $.

In order to lower bound the number of $(\leq k)$-critical transpositions in $%
\Pi $, we distinguish two types of transpositions. A transposition is \emph{%
monochromatic }if it occurs between two $a$-elements, between two $b$%
-elements, or between two $c$-elements; otherwise it is called \emph{%
bichromatic}. We let ${N_{\leq k}^{{}mono}(\Pi )}$ (respectively, ${N_{\leq
k}^{{}bi}(\Pi )}$) denote the number of monochromatic (respectively,
bichromatic) $(\leq k)$-critical transpositions in $\Pi $, so that ${N_{\leq
k}(\Pi )}={N_{\leq k}^{{}mono}(\Pi )}+{N_{\leq k}^{{}bi}(\Pi )}.$ We now
bound ${N_{\leq k}^{{}mono}(\Pi )}$ and ${N_{\leq k}^{{}bi}(\Pi )}$
separately.

\subsubsection{Calculating ${N_{\leq k}^{{}bi}(\Pi )}$}

\begin{proposition}
\label{pro:heterogeneous} Let $\Pi $ be a $3$-decomposable halfperiod on $n$
points, and let $k<n/2$. Then
\begin{equation*}
{N_{\leq k}^{{}bi}(\Pi )}=%
\begin{cases}
3\binom{k+1}{2}\text{\quad if $k\leq n/3$,} \\[0.4cm]
3\binom{n/3+1}{2}+(k-n/3)n\text{\quad if $n/3<k<n/2$.}%
\end{cases}%
\end{equation*}
\end{proposition}

\begin{proof} Each bichromatic transposition is either an $ab$- or an $ac$- or a $bc$-transposition. Since $\Pi $ is $3$-decomposable, $A$ and $B\cup C$ are
separated in $\pi _{0}$. Using only this fact, we compute the number of $i$%
-critical bichromatic transpositions involving $A$, that is, the $ab$- and $%
ac$-transpositions together. This number multiplied by $3/2$ is the total
number of bichromatic $i$-critical transpositions of $\Pi $. This is
because, by definition of $3$-decomposable, there is a permutation $\pi _{s}$%
\ of $\Pi $ where $B$ is separated from $A\cup C$, as well as a permutation $%
\pi _{t}$ where $C$ is separated from $A\cup B$. Thus, multiplying by $3$
counts each $i$-critical bichromatic transposition twice.

For $x\in\{b, c\}$ each $ax$-transposition in $\Pi$ moves the involved $a$ to the right and
the involved $b$ or $c$ to the left. Since $A$ occupies the first $n/3$
positions in $\pi _{0}$, then $A$ must occupy the last $n/3$ positions in $%
\pi _{\binom{n}{2}}$. For each $i\leq n/3$, a bichromatic $i$-transposition
involving $A$, replaces one $a$-element occupying one of the first $i$-positions by a $b$%
- or a $c$-element. This must happen exactly $i$ times in order for $A$ to
leave the first $i$ positions. That is, there are exactly $i$ bichromatic $i$%
-transpositions involving $A$. Similarly, for each $i\geq 2n/3$, there are
exactly $i$ bichromatic $i$-transpositions involving $A$ (each of these
transpositions replaces one $b$- or $c$-element in the last $i$ positions by
an $a$-element). Finally, for $n/3<i<2n/3$, there are exactly $n/3$
bichromatic $i$-transpositions involving $A$, since all elements of $A$ must
leave the region formed by the first $i$ positions. Therefore, the number of
$(\leq k)$-critical bichromatic transpositions is exactly $\sum_{i=1}^{k}3i=3%
\binom{k+1}{2}$ if $k\leq n/3$, and $\sum_{i=1}^{n/3}3i+\sum_{i=n/3}^{k}n=3%
\binom{n/3+1}{2}+(k-n/3)n$ if $n/3<k<n/2$.
\end{proof}

\subsubsection{Bounding ${N_{\leq k}^{{}mono}(\Pi )}$}

A transposition between elements in positions $i$ and $i+1$ with $k<i<n-k$
is called a $\left( >k\right) $\emph{-transposition}. All these
transpositions are said to occur in the $k$\emph{-center} (of $\Pi $). Our
goal is to give a lower bound (Proposition~\ref{pro:boundingnoedges}) for ${%
N_{\leq k}^{{}mono}(\Pi )}$. Each monochromatic transposition is an $aa$- or
$bb$-, or $cc$-transposition. Our approach is to find an upper bound for the
number of $\left( >k\right) $-critical $aa$-, $bb$-, and $cc$%
-transpositions, denoted by ${N_{>k}^{{aa}}(\Pi )}$, ${N_{>k}^{{bb}}(\Pi )}$%
, and ${N_{>k}^{{cc}}(\Pi )}$, respectively. The lower bound for ${N_{\leq
k}^{{}mono}(\Pi )}$ follows from the observation that the number of $(\leq
k) $-critical $aa$-transpositions is exactly $\binom{n/3}{2}-{N_{>k}^{{aa}%
}(\Pi )}$, and similarly for $bb$- and $cc$-transpositions. Thus%
\begin{equation*}
{N_{\leq k}^{{}mono}(\Pi )}=3\binom{n/3}{2}-{N_{>k}^{{aa}}(\Pi )}-{N_{>k}^{{%
bb}}(\Pi )}-{N_{>k}^{{cc}}(\Pi )}.
\end{equation*}%
Again, we bound ${N_{>k}^{{aa}}(\Pi )}$ using only the fact that there is a
permutation where $A$ is separated from $B\cup C$, and thus this bound is
the same for ${N_{>k}^{{bb}}(\Pi )}$ and ${N_{>k}^{{cc}}(\Pi )}$.

It is known that for $k\leq n/3$, the bound ${N_{\leq k}(\Pi )}\geq 3\binom{%
k+1}{2}$ is tight. Since we have shown that there are $3\binom{k+1}{2}$
bichromatic $\left( \leq k\right) $-transposition, we focus on the case $%
n/3<k<n/2$. In this case, let $D_{k}$ be the digraph with vertex set $%
1,2,\ldots ,n/3$, and such that there is a directed edge from $i$ to $j$ if
and only if $i<j$ and the transposition $a_{i}a_{j}$ occurs in the $k$%
-center. Then the number of edges of $D_{k}$ is exactly ${N_{>k}^{{aa}}(\Pi )%
}$.

We now bound the number of edges in $D_{k}$ using the following essential
observation. We denote the outdegree and the indegree of a vertex $v$ in a
digraph by ${[v]^{+}}$ and ${[v]^{-}}$, respectively.

\begin{lemma}
For the graph $D_{k}$,%
\begin{equation}
\left[ i\right] ^{+}\leq \min \{n-2k-1+\left[ i\right]
{^{-}},n/3-i\}. \label{eq:degrees}
\end{equation}
\end{lemma}

\begin{proof}
Clearly, $\left[ i\right] ^{+}\leq n/3-i$ because there are only
$n/3-i$ indices $j>i$. To show that $\left[ i\right] ^{+}\leq
n-2k-1+\left[ i\right] ^{-}$, note that $n-2k-1+\left[ i\right]
^{-}$ is the number of $\left(
>k\right) $-transpositions in which $a_{i}$ moves right, and only $\left[ i%
\right] ^{+}$ of these transpositions involve two $a$-elements. Indeed, $%
\left[ i\right] {^{-}}$ is the number of $\left( >k\right) $-transpositions
involving two $a$-elements in which $a_{i}$ moves backward. There are $%
n-2k-1 $ forced $\left( >k\right) $-transpositions of $a_{i}$: since $a_{i}$
moves from position $i$ to position $n-i+1$, for each $k<j<n-k$ there is at
least one $j$-transposition in which $a_{i}$ moves right. Also, each of the $%
\left[ i\right] {^{-}}$ transpositions in which $a_{i}$ moves left in the $k$%
-center allows an extra transposition in the $k$-center in which $a_{i}$
moves right.
\end{proof}

\begin{proposition}
\label{pro:boundingnoedges} If $\Pi $ is a $3$-decomposable halfperiod on $n$
points, and $n/3<k<n/2$, then%
\begin{equation*}
{N_{\leq k}^{{}mono}(\Pi )}\geq B\left( k,n\right) -3\binom{n/3+1}{2}%
-(k-n/3)n.
\end{equation*}
\end{proposition}

\begin{proof}
We just need to show that $D_{k}$ has at most $\binom{n/3}{2}-\frac{1}{3}%
\left( B\left( k,n\right) -3\binom{n/3+1}{2}-(k-n/3)n\right) =\frac{1}{3}%
\left( kn-B\left( k,n\right) \right) $ edges. We start by giving two
definitions. Let $\mathcal{D}_{v,m}$ be the class of all digraphs on
$v$
vertices $1,2,\ldots ,v$ satisfying that $\left[ i\right] ^{+}\leq m+\left[ i%
\right] ^{-}$ for all $1\leq i\leq v$, and $i<j$ whenever $i\rightarrow j$.
Let $D_{0}\left( v,m\right) $ be the graph in $\mathcal{D}_{v,m}$ with
vertices $1,2,\ldots ,v$ recursively defined by

\begin{itemize}
\item $\left[ 1\right] ^{-}=0$,

\item $\left[ i\right] ^{+}=\min \left\{ \left[ i\right] ^{-}+m,v-i\right\} $
for each $i\geq 1$, and

\item for all $1\leq i<j\leq v$, $i\rightarrow j$ if and only if $i+1\leq
j\leq i+\left[ i\right] ^{+}$.
\end{itemize}

These definitions are equivalent to those in \cite{baloghsalazar} (pages 677
and 683). There, Balogh and Salazar show that the maximum of the function $%
2\sum_{i=1}^{v}\left[ i\right] ^{-}+\sum_{i=1}^{v}\min \left\{ \left[ i%
\right] ^{-}-\left[ i\right] ^{+}+m,m+1\right\} $ over all digraphs in $%
\mathcal{D}_{v,m}$ is attained by $D_{0}\left( v,m\right) $. Their original
statement imposes some dependency between $v$ and $m$, but this is only used
to bound the given function applied to $D_{0}\left( v,m\right) $. And their
proof, actually maximizes separately each of the two sums above. In other
words, they implicitly show that the maximum number of edges of a graph in $%
\mathcal{D}_{v,m}$ is attained by $D_{0}\left( v,m\right) $.

Note that $D_{k}$ is in $\mathcal{D}_{n/3,n-2k-1}$, and thus its number of
edges is bounded above by the number of edges of $D_{0}\left(
n/3,n-2k-1\right) $. Thus, it suffices to bound above the number of edges of
$D_{0}\left( n/3,n-2k-1\right)$.

\begin{lemma}
\label{cla:theclaim}$D_{0}\left( n/3,n-2k-1\right) {\ }$ has at most $\frac{1%
}{3}\left( kn-B\left( k,n\right) \right) $ edges.
\end{lemma}

The next section is devoted to the proof of this claim.
\end{proof}

\bigskip

The proof of Theorem~\ref{prop:main} follows immediately from Propositions~%
\ref{pro:heterogeneous} and \ref{pro:boundingnoedges}.

\section{Proof of Lemma~\ref{cla:theclaim}}

We prove Lemma~\ref{cla:theclaim} in two steps. We first
obtain an expression for the exact number of edges in $D_{0}(n/3,n-2k)$, and
then we show that this value is upper bounded by the expression in Lemma~\ref%
{cla:theclaim}. For brevity, in the rest of the section, we use $%
D_{0}:=D_{0}(n/3,n-2k-1)$, $v:=n/3$ and $m:=n-2k-1$.

\subsection{The exact number of edges in $D_{0}$}

For positive integers $j\leq i$ define (c.f., Definition 16 in \cite%
{baloghsalazar}) $S_{j}(i)$ as the unique nonnegative integer such
that
\begin{equation*}
\binom{S_{j}(i)}{2}<\frac{i}{j}\leq \binom{S_{j}(i)+1}{2}\text{; and}
\end{equation*}
$T_{j}(i)$ and $U_{j}(i)$ as the unique integers satisfying $0\leq
T_{j}(i)\leq j-1$, $0\leq U_{j}(i)\leq S_{j}(i)-1$, and
\begin{equation}
i=1+j\binom{S_{j}(i)}{2}+S_{j}(i)T_{j}(i)+U_{j}(i)\text{.}  \label{eq:a1}
\end{equation}

The key observation is that we know the indegree of each vertex in $D_{0}$.

\begin{proposition}[Proposition 17 in~\protect\cite{baloghsalazar}]
\label{ingrade} For each vertex $1\leq i\leq v$ of $D_{0}$,%
\begin{equation*}
\left[ i\right] ^{-}=m\left( S_{m}(i)-1\right) +T_{m}(i).
\end{equation*}
\end{proposition}

We now find a close expression for $\sum_{i=1}^{v}\left[ i\right] ^{-}$, the
number of edges in $D_{0}$.


\begin{proposition}
The exact number of edges in $D_{0}$ is%

\begin{equation}
E(k,n):= 2m^{2}\binom{S_{m}(v)}{3}+\binom{m}{2}\binom{S_{m}(v)}{2}+2m\cdot
T_{m}(v)\binom{S_{m}(v)}{2}+ \label{eq:exactedges} 
\end{equation}
\begin{equation*}
\binom{T_{m}(v)}{2}S_{m}(v)+\left( U_{m}(v)+1\right) \left(
m(S_{m}(v)-1)+T_{m}(v)\right) 
\end{equation*} 

\end{proposition}

\noindent \emph{Proof.} We break $\sum_{i=1}^{v}\left[ i\right] {^{-}}$ into
three parts. Let $v_{1}:=m\tbinom{S_{m}(v)}{2}$, $v_{2}:=S_{m}(v)T_{m}(v)$,
and set%
\begin{equation*}
V_{1}=\sum_{i=1}^{v_{1}}\left[ i\right] {^{-}}\text{, }V_{2}=%
\sum_{i=v_{1}+1}^{v_{1}+v_{2}}\left[ i\right] {^{-}}\text{, and }%
V_{3}=\sum_{i=v_{1}+v_{2}+1}^{v}\left[ i\right] {^{-}}
\end{equation*}%
so that
\begin{equation}
\sum_{i=1}^{v}\left[ i\right] {^{-}}=V_{1}+V_{2}+V_{3}.  \label{eq:decomp}
\end{equation}

We calculate $V_{1}$, $V_{2}$, and $V_{3}$ separately.

If $\ell ,j$ are integers such that $1\leq j\leq S_{m}(v)-1$ and $0\leq \ell
\leq m$, we define $P_{j}:=\{i:S_{m}(i)=j\}$ and $Q_{j,\ell }:=\{i\in
P_{j}:T_{m}(i)=\ell \}.$

We first calculate $V_{1}$. Note that $P_{1},P_{2},\ldots ,P_{S_{m}(v)-1}\ $%
is a partition of $\{1,2,...,v_{1}\}$ and $Q_{j,0},Q_{j,1},\ldots ,Q_{j,m}$
is a partition of $P_{j}$, for each $1\leq j\leq S_{m}(v)-1.$ Also, $%
S_{m}(v_{1}+1)=S_{m}(v)$ and $S_{m}(i)\leq S_{m}(v)-1$ for $1\leq i\leq
v_{1} $. Thus $V_{1}$ can be rewritten as $\sum_{j=1}^{S_{m}(v)-1}\sum_{i\in
P_{j}}\left[ i\right] {^{-}}.$ By Proposition \ref{ingrade}, this equals
\begin{eqnarray*}
V_{1} &=&\sum_{j=1}^{S_{m}(v)-1}\left( m\sum_{i\in P_{j}}\left(
S_{m}(i)-1\right) +\sum_{i\in P_{j}}T_{m}(i)\right) \\
&=&\sum_{j=1}^{S_{m}(v)-1}\left( m\sum_{i\in P_{j}}\left( j-1\right)
+\sum_{\ell =0}^{m}\sum_{i\in Q_{j,\ell }}\ell \right) .
\end{eqnarray*}%
On other hand, by definition $\left\vert Q_{j,\ell }\right\vert $ $=j$ for $%
0\leq \ell \leq m-1$, which implies that $\left\vert P_{j}\right\vert =mj$.
Therefore%
\begin{align}
V_{1}\ & =\sum_{j=1}^{S_{m}(v)-1}\left( m^{2}j\left( j-1\right) +\sum_{\ell
=0}^{m}\ell \left\vert Q_{j,\ell }\right\vert \right)
=\sum_{j=1}^{S_{m}(v)-1}\left( m^{2}j\left( j-1\right) +j\sum_{\ell
=1}^{m-1}\ell \right)  \notag \\
& =\sum_{j=1}^{S_{m}(v)-1}\left( 2m^{2}\binom{j}{2}+\binom{m}{2}j\right)
=2m^{2}\binom{S_{m}(v)}{3}+\binom{m}{2}\binom{S_{m}(v)}{2}.  \label{eq:fora}
\end{align}%
$\allowbreak $

Now, we calculate $V_{2}$. Since $S_{m}(i)=S_{m}(v)$\textrm{\ }for each $%
v_{1}+1\leq i\leq v$, and $\left[ i\right] {^{-}}=$ $m\left(
S_{m}(i)-1\right) +T_{m}(i)$, then $V_{2}=\sum_{i=v_{1}+1}^{v_{1}+v_{2}}%
\left[ i\right] ^{-}=\sum_{i=v_{1}+1}^{v_{1}+v_{2}}m\left( S_{m}(v)-1\right)
+T_{m}(i).$ Therefore
\begin{equation*}
V_{2} =\sum_{i=v_{1}+1}^{v_{1}+v_{2}}m\left( S_{m}(v)-1\right)
+\sum_{i=v_{1}+1}^{v_{1}+v_{2}}T_{m}(i)=m\left( S_{m}(v)-1\right)
S_{m}(v)T_{m}(v)+\sum_{i=v_{1}+1}^{v_{1}+v_{2}}T_{m}(i)\text{.}
\end{equation*}%
Again, we have that $\left\vert Q_{S_{m}(v),\ell }\right\vert
=S_{m}(v)$ for every $0\leq \ell \leq m-1$. Because $0\leq
T_{m}(i)\leq T_{m}(v)-1$ for every $v_{1}+1\leq i\leq v_{1}+v_{2}$,
and $T_{m}(v_{1}+v_{2}+1)=T_{m}(v)$, it follows that
$Q_{S_{m}(v),0},Q_{S_{m}(v),1},\ldots ,Q_{S_{m}(v),T_{m}(v)-1}$ is a
partition of $\{v_{1}+1,v_{1}+2,...,v_{1}+v_{2}\}.$ Thus%
\begin{eqnarray*}
\sum_{i=v_{1}+1}^{v_{1}+v_{2}}T_{m}(i) =\sum_{\ell
=0}^{T_{m}(v)-1}\sum_{i\in Q_{S_{m}(v),\ell }}T_{m}(i)=\sum_{\ell
=0}^{T_{m}(v)-1}\ell \left\vert Q_{S_{m}(v),\ell }\right\vert
=\sum_{\ell =1}^{T_{m}(v)-1}\ell \cdot
S_{m}(v)=S_{m}(v)\binom{T_{m}(v)}{2}.
\end{eqnarray*}%
Then%
\begin{eqnarray}
V_{2} = 2m\binom{S_{m}(v)}{2}T_{m}(v)+S_{m}(v)\binom{T_{m}(v)}{2}.
\label{eq:forb}
\end{eqnarray}

Finally, we calculate $V_{3}$. Since $S_{m}(i)=S_{m}(v)$ and $%
T_{m}(i)=T_{m}(v)$ for every $v_{1}+v_{2}+1\leq i\leq v$ and $\left[ i\right]
{^{-}}=$ $m\left( S_{m}(i)-1\right) +T_{m}(i)$, it follows that
\begin{align}
V_{3}=\sum_{i=v_{1}+v_{2}+1}^{v}\left[ i\right] {^{-}}&
=\sum_{i=v_{1}+v_{2}+1}^{v}m\left( S_{m}(i)-1\right)
+T_{m}(i)=\sum_{i=v_{1}+v_{2}+1}^{v}m\left( S_{m}(v)-1\right) +T_{m}(v)
\notag \\
& =(v-v_{1}-v_{2})\left( m\left( S_{m}(v)-1\right) +T_{m}(v)\right)  \notag
\end{align}%
From (\ref{eq:a1}) it follows that $U_{m}(v)+1=v-v_{1}-v_{2}$, and so
\begin{equation}
V_{3}=\left( U_{m}(v)+1\right) (m\left( S_{m}(v)-1\right) +T_{m}(v)).
\label{eq:forc}
\end{equation}%
Now from (\ref{eq:fora}), (\ref{eq:forb}), and (\ref{eq:forc}), it follows
that $E(k,n)=V_{1}+V_{2}+V_{3}$, and so Proposition~\ref{pro:boundingnoedges}
follows from (\ref{eq:decomp}).

\subsection{Upper bound for number of edges in $D_{0}$}

\begin{proof}[\textbf{Proof of Lemma~\protect\ref{cla:theclaim}}]
Recall that $v:=n/3$ and $m:=n-2k-1$. If $k>n/3$, then $v\geq m$. From (\ref%
{eq:a1}), it follows that
\begin{equation*}
T_{m}(v)=\frac{v-1-m\binom{S_{m}(v)}{2}-U_{m}(v)}{S_{m}(v)}\text{.}
\end{equation*}%
Note that $s=s(k,n)$ in the definition of $B(k,n)$ is equal to $S_{m}(v)$.
We use this fact, together with the previous identity substituted in the
expression of $E(k,n)$ in (\ref{eq:exactedges}), to obtain the following
expression for $E(k,n)+(B(k,n)-kn)/3$. The next identity follows from a
long, yet elementary, simplification (which can be efficiently performed in
a CAS like Maxima, Mathematica or Maple).%
\begin{eqnarray*}
E(k,n)+\frac{1}{3}\left( B(k,n)-kn\right) &=&\frac{4s^{2}-s^{4}-3\left(
2+2U_{m}(v)-s\right) ^{2}}{24s} \\
&\leq &\frac{s^2\left( 4-s^{2}\right) }{24}\leq 0\text{.}
\end{eqnarray*}%
The last inequality follows from the fact that $s=S_{m}(v)\geq 2$ whenever $%
k>n/3$.
\end{proof}


\section{Constructing geometric drawings from smaller ones}\label{gencon}

In this section, we describe a refinement of a method used in \cite{af1, aak, brodsky} to grow a geometric drawing $D_m$ of $K_m$ (the {\em base} drawing) into a geometric drawing of $K_n$ (the {\em augmented} drawing) for some $n > m$. The goal is to produce geometric drawings of complete graphs with as few crossings as possible. The method substitutes each point $p_i$ in the underlying point set of $D_m$ by a {\em cluster} of points $C_i$. The cluster $C_i$ is an affine copy of a preset {\em cluster model} $S_i$ (so that the order types of $C_i$ and $S_i$ are the same) carefully placed near $p_i$ and almost aligned along a line $\ell_i$ through $p_i$. More precisely, if $C=\bigcup_{j=1}^{m}{C_j}$, then $\ell_i$ divides the set $C \setminus C_i$ into two sets of sizes as equal as possible, and any line spanned by two points in $C_i$ has the same  ``halving'' property as $\ell_i$ on $C \setminus C_i$. Such a placement helps to minimize the number of convex quadrilaterals that involve two points in $C_i$ and, as a consequence, the total number of crossings in the augmented drawing.

In a nutshell, the difference between our approach and that in~\cite{aak} is that, for each $i$, we allow one cluster $C_{\sigma(i)}$ with $\sigma(i) \neq i$ to be splitted by $\ell_i$, and ask that no two clusters split each other. Whereas in~\cite{aak}, each cluster $C_j$ other than $C_i$ is completely contained in a semiplane of $\ell_i$. While this step further is more general and powerful, it brings new technical complications that are analyzed and sorted out throughout this section.

\subsection{Input and  output}

The primary ingredients of our construction are a base point-set $P$, sets $S_i$ that serve as models for our clusters, and what we call a {\em pre-halving} set of lines (Condition 3 below), which is a generalization of the corresponding ``halving properties'' required in~\cite{af1, aak}.

\vglue 0.4 cm

\noindent{\em The input }

\vglue 0.4 cm

\begin{enumerate}
\item The {\em base set}: a point set $P=\{p_1,p_2,\ldots,p_m\}$ in general position. This is the underlying set of the {\em base} geometric drawing of $K_m$.
\item The {\em cluster models}: for each $i=1,2,\ldots,m$, a nonempty point set $S_i$ in general position. We ask that no two points in a cluster $S_i$ have the same $x$-coordinate. Let $s_i=|S_i|$ and $I=\{i:s_i>1\}$.

\begin{figure}[htbp]
\begin{center}
\includegraphics[width=6.5in]{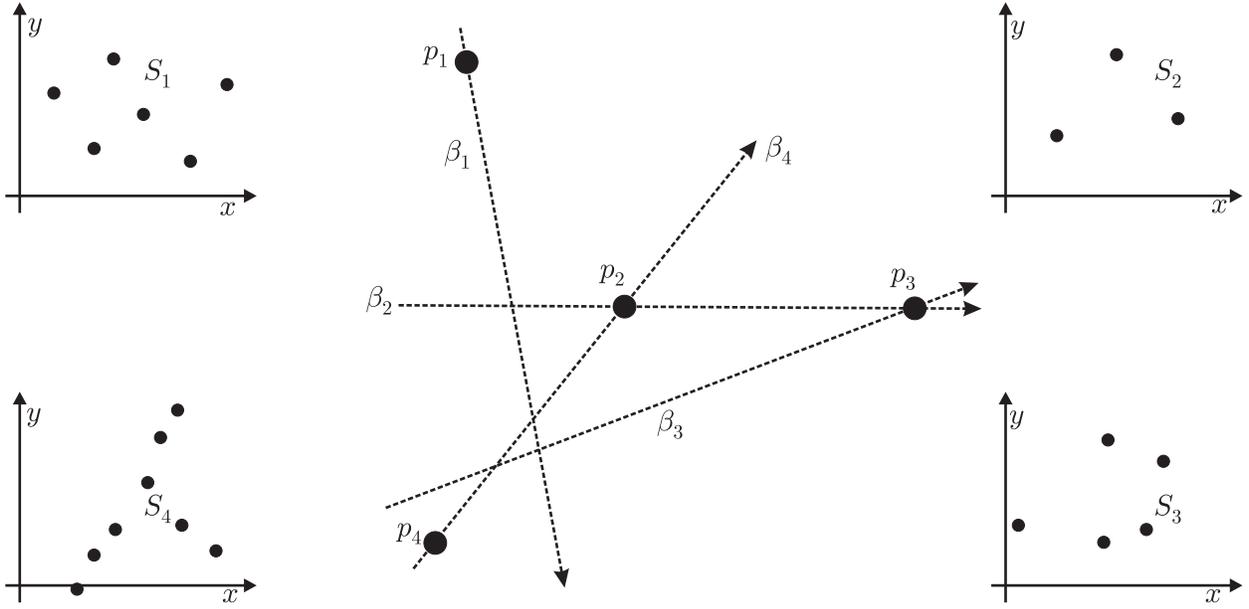}
\caption{The sets $S_1, S_2, S_3$, and $S_4$ are cluster models. We show a pre-halving set of lines $\{\beta_1,\beta_2,\beta_3,\beta_4\}$ for the base point-set $P=\{p_1,p_2,p_3,p_4\}$ and the integers $s_1=|S_1|=6, s_2=|S_2|=3, s_3=|S_3|=5$, and $s_4=|S_4|=8$.}
\label{fig:f01}
\end{center}
\end{figure}

\item The pre-halving set of lines: for each $i \in I$, a directed line $\beta_i$ containing $p_i$. For each $\beta_i$, we let $\ll(i)$ (respectively, $\rr(i)$) denote the set of those $k$ such that $p_k$ is on the left (respectively, right) semiplane of $\beta_i$. If $\beta_i$ goes through a $p_j$ other than $p_i$, we say that $p_i$ and $\beta_i$ are {\em splitting}. In this case, we say that $\beta_i$ {\em splits} $p_j$, and write $j = \sigma(i)$. Otherwise, $p_i$ and $\beta_i$ are called {\em simple}. (Note that $\sigma(i)$ is defined if and only if $p_i$ and $\beta_i$ are splitting.) The collection of these lines must satisfy the following properties.
\begin{description}
\item{(a)} If $i\neq j$, then $\beta_i \neq \beta_j$ and $\beta_i \neq -\beta_j$, the reverse line of $\beta_j$.
\item{(b)} If $\beta_i$ is simple, then $0 \leq \sum_{k \in \ll(i)} s_k - \sum_{k \in \rr(i)} s_k \leq 1$.
\item{(c)} If $\beta_i$ is splitting, then $\beta_i$ is directed from $p_i$ to $p_{\sigma(i)}$ and $|\sum_{k \in \ll(i)} s_k - \sum_{k \in \rr(i)} s_k| \le s_{\sigma (i)} - 1$.
\end{description}
\end{enumerate}

Note that properties (a) to (c) relate only to the point set $P$ and to the integers $s_i$, and are independent of the order types of the sets $S_i$.

\vglue 0.4 cm

\noindent{\em The output }

\vglue 0.4 cm

The construction consists of substituting each $p_i$, with $i \in I$, by a {\em cluster}\ $C_i$. $C_i$ is a suitable affine copy of $S_i$ whose points are aligned along a line $\ell_i$. If $s_i=1$, then $C_i=\{p_i\}$. The result is a set $C:=\bigcup_{i=1}^m C_i$ of $n:=|C|$ points in general position, the {\em augmented} point set. To describe in detail the properties of $C_i$ and $\ell_i$, we need a couple of definitions.

A directed line $\ell$ {\em halves} a set of points $T$ if the left semiplane of $\ell$ contains $\tcl{|T|/2}$ points of $T$, and the right semiplane contains the remaining $\tfl{|T|/2}$ points. It follows from the definition that $\ell$ and $T$ are disjoint.
If $\ell$ is a line that halves a set $T$, and $S$ is a set of points disjoint from $T$, then $S$ {\em halves $T$ as $\ell$}, if every line $\ell'$ spanned by two points in $S$ can be directed so that it halves $T$ in exactly the same way as $\ell$. That is, the left (respectively, right) semiplane of $\ell'$ contains the same subset of $T$ as the left (respectively, right) semiplane of $\ell$.

With this terminology, the key properties of the sets $C_i$ and of the lines $\ell_i$ are the following.

\begin{description}
\item{(1)} {\bf Inherited order type property.} For any three pairwise distinct $i,j,k$, and $q_i \in C_i, q_j\in C_j, q_k\in C_k$, the order type of the triple $q_i q_j q_k$ is the same as the order type of $p_i p_j p_k$.
\item{(2)} {\bf Halving property.} For each $i \in I$, $\ell_i$ halves $C \setminus C_i$ and $C_i$ halves $C \setminus C_i$ as $\ell_i$.
\end{description}

\subsection{The construction}

\bigskip
\noindent{\sc Step 1 } {\sl Enlarging each point $p_i$ to a very small disc $D^i$ that will contain the cluster $C_i$.}
\bigskip

For each $i=1,\ldots,m$, let $D^i$ be a disc of radius $r_i$ centered at $p_i$, such that the collection ${D^i}$ satisfies the following. If $q_i\in D^i, q_j\in D^j, q_k\in D^k$ (with $i,j,k$ pairwise distinct), then the order type of the triple $q_i q_j q_k$ is the same as the order type of $p_i p_j p_k$. It is clear that this can be achieved by making the radius of each $D^i$ sufficiently small.

\bigskip
\noindent{\sc Step 2 } {\sl Replacing each $p_i$ with a set $U_i$ contained on $D^i \cap \beta_i$.}
\bigskip

We now construct a first approximation $U_i$ to each cluster $C_i$. The first simplification is that the each set $U_i$ is collinear, as opposed to $C_i$, which is in general position. Although, we might certainly describe the construction without using intermediate collinear sets, it is a convenient device that greatly simplifies our work.

For each $i \in I$, consider a similarity transformation that takes the origin to $p_i$ and the $x$-axis to $\beta_i$, such that the image $C_i$ of $S_i$ is contained in the interior of the disc centered at the origin with radius $r_i/2$. Let $U_i$ be the projection of $C_i$ onto $\beta_i$, thus $U_i$ lies on $\beta_i$. If $s_i=1$, we make $U_i=C_i=\{p_i\}$. Then $U_i$ is completely contained in $D^i$ for every $i$. Let $U=\bigcup_{j=1}^m {U_j}$. See Figure~\ref{fig:f03}.

\begin{figure}[htbp]
\begin{center}
\includegraphics[width=6.5in]{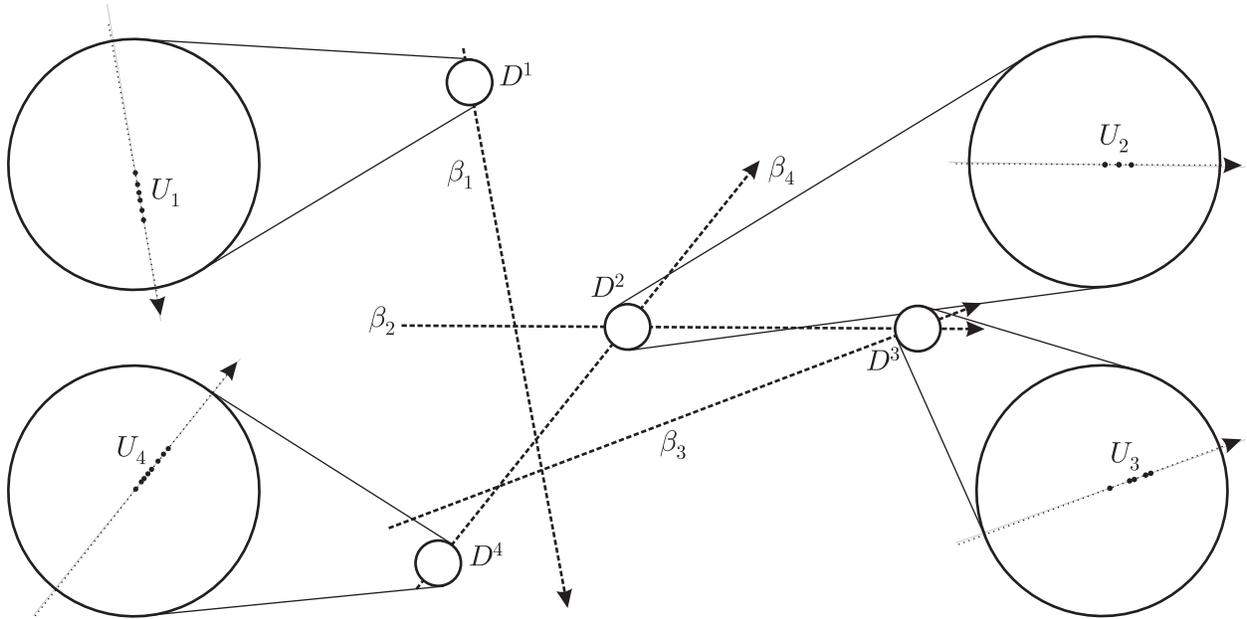}
\caption{Enlarging each point $p_i$ to a small disc $D^i$ of radius $r_i$ (example from Figure~\ref{fig:f01}) and the sets $U_1, U_2, U_3$, and $U_4$ from Step 2. Eeach set $U_i$ lies on $\beta_i$ and is contained in the disc $D^i$.}
\label{fig:f03}
\end{center}
\end{figure}

Before moving on to the next step, we observe that each set $\beta_i$ has a good halving potential. In fact, if $\beta_i$ is simple, it already halves $U \setminus U_i$. And if $\beta_i$ is splitting, then the difference between the number of points in $U \setminus U_i$ on each side of $\beta_i$ is at most $s_{\sigma(i)} - 1$ . In this case, $\beta_i$ does not necessarily halve $U \setminus U_i$, but it intersects $D^{\sigma(i)}$, which contains exactly $s_{\sigma(i)}$ points of $U \setminus U_i$. Thus, a very small rotation of $U_i$ (and $\beta_i$) may balance this difference. A preview of Figure~\ref{fig:f08} may be of help here.
Unfortunately, there is a significant gap to be filled: we may certainly perform this rotation to adjust any particular $\beta_i$, but whenever the turn comes for $\beta_{\sigma(i)}$ to be adjusted, if we rotate this line we may break the halving property previously achieved by $\beta_i$. Taking care of this possible scenario transforms an otherwise intuitive, straightforward procedure into a somewhat technical one. This is the task for the next step.

\bigskip
\noindent{\sc Step 3 } {\sl Moving the sets $U_i$, so that each $U_i$ lies on a line $\ell_i$ that halves $U \setminus U_i$.}
\bigskip

Our goal in this step is to slightly move (rotate or translate) each set $U_i$ with $i \in I$, so that the line containing $U_i$ passes through $p_i$ and halves $U \setminus U_i$. In what follows, $\ell_i$ denotes the line containing $U_i$. We describe a dynamic process that moves $U_i$, and accordingly $\ell_i$ and $C_i$. Even when we are actually transforming the $U_i$, $\ell_i$, and $C_i$, we keep their names all the way through. If $s_i=1$, $U_i=C_i=\{p_i\}$ remains unchanged throughout this process. The central feature of the whole process is the following

\bigskip
\noindent{\bf Key property }
{\sl The set $U_i$ is contained in the interior of $D^i$ and lies on $\ell_i$ (whenever $s_i>1$) during the entire process. In their final position, $\ell_i$ goes through $p_i$ and halves $U \setminus U_i$.}

\begin{figure}[htbp]
\begin{center}
\includegraphics[width=6.5in]{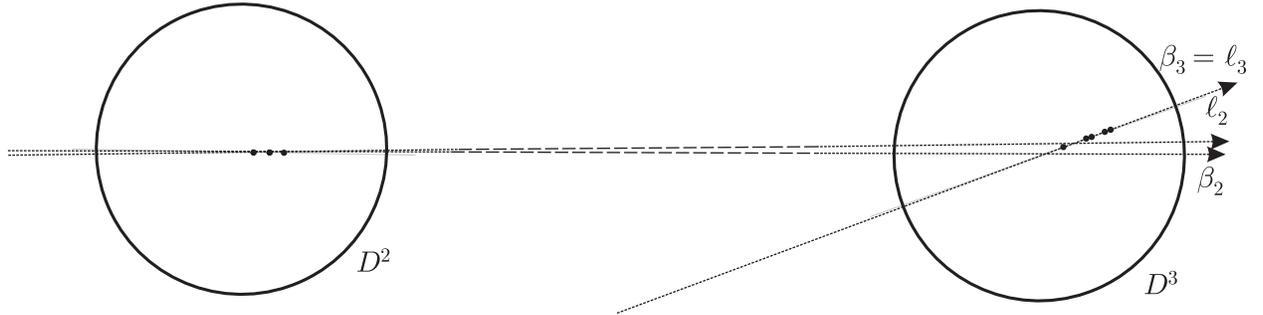}
\caption{We consider $D^2$ and $D^3$ from Figure~\ref{fig:f03}. The  $\ell_2$ halves $U \setminus U_2$ and $\ell_3$ halves $U \setminus U_3$. Since $p_3$ is simple, $U_3$ remains unchanged and $\ell_3=\beta_3$. $p_2$ is splitting, with $\beta_2$ through $p_3$. There are $19$ points in $U \setminus U_2$, $10$ of which must be on the left of $\ell_2$. $U_1$ ($6$ points) is on the left of $\beta_2$ and $U_4$ ($8$ points) is on its right. We use $U_3$ to balance: rotate $U_2$ around $p_2$, so that $\ell_2$ leaves $4$ points of $U_3$ on its left. }
\label{fig:f08}
\end{center}
\end{figure}

To describe the process, we consider the digraph $G$ with vertex set $P^{\prime}=\{p_i\in P:i\in I\}$, induced by the set of splitting pre-halving lines, that is, there is an arc from $p_i$ to $p_j$ if and only if $\sigma(i) =j$, see Figure~\ref{fig:graphs}. Thus, if $p_i$ is simple, then its outdegree is zero, and if it is splitting, then its outdegree is one. These properties guarantee that each strong component of $G$ is either acyclic, or contains at most one directed cycle. In any case, each strong component must have a vertex, called {\em root}, that can be {\em reached} from all other vertices in the component. (That is, for each vertex $p$ in the component, there is a directed path from $p$ to the root.)

\begin{figure}[h]
\begin{center}
\includegraphics[width=6.5in]{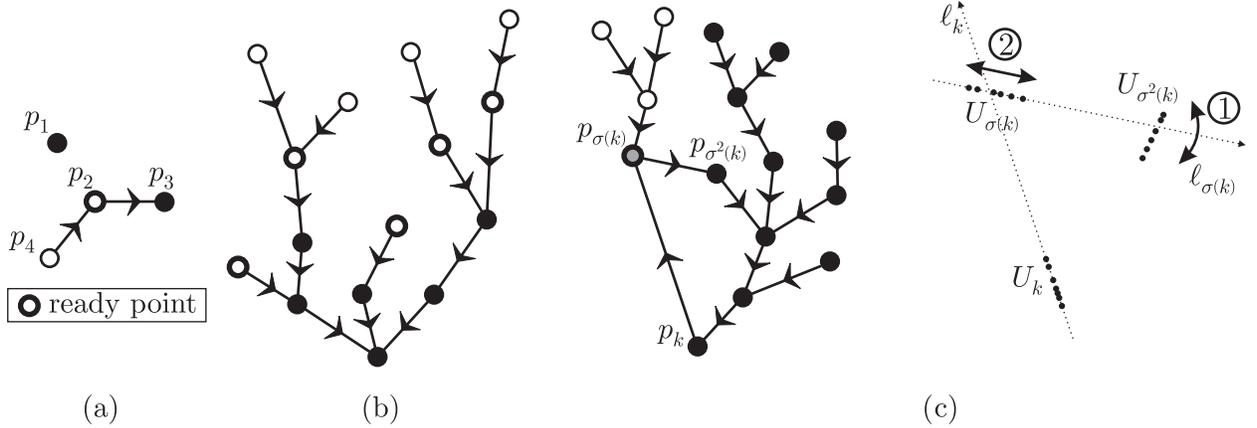}
\caption{(a) The graph $G$ corresponding to the example in Figure~\ref{fig:f01}. (b) An acyclic component. (c) A component with a cycle at the time (2) is applied in Step 3.}
\label{fig:graphs}
\end{center}
\end{figure}

We work on one component at a time. Let $P_c\subseteq P^{\prime}$ be a strong component of $G$ and $p_k$ its root. Start by coloring all vertices of $G$ white. Coloring a point $p_i$ black means that $\ell_i$ and $U_i$ have reached their final position. Color $p_k$ black, and if $p_k$ is splitting, then color $p_{\sigma (k)}$ grey. A white or a grey point is said to be {\em ready} if $p_{\sigma (k)}$ is black. As long as there are ready points, we apply (1) or (2) below.

\begin{description}
\item{(1)} If possible, arbitrarily choose a white ready point $p_i$. Slightly rotate $U_i$ around $p_i$ until $\ell_i$ halves $U \setminus U_i$. This is always possible asking that $\ell_i$ intersects $D^{\sigma (i)}$ at all times, because $\beta_i \neq \pm \beta_j$, $\ell_i$ intersects $D^{\sigma (i)}$, $D^{\sigma (i)}$ has $s_{\sigma (i)}$ points, and before rotating $U_i$, we have an unbalance of at most $s_{\sigma (i)}-1$. Color $p_i$ black.

\item{(2)} If (1) cannot be applied, then work with the grey point $p_{\sigma (k)}$. First, proceed as in (1), that is, rotate $\ell_{\sigma (k)}$ until it halves $U \setminus U_{\sigma (k)}$. Then translate $U_{\sigma (k)}$ along $\ell_{\sigma (k)}$ until $\ell_k$ (which stays still) halves $U \setminus U_{\sigma (k)}$. Since $U_{\sigma (k)}$ was originally contained on a disc of radius $r_{\sigma (k)}/2$ centered at $p_{\sigma (k)}$, then $U_{\sigma (k)}$ is still contained in $D^{\sigma (k)}$ during the translation. Color $p_{\sigma (k)}$ black. See Figure~\ref{fig:graphs}(c).
\end{description}

Note that (2) is applied at most once, and if we cannot apply (1) or (2), then all points are already black. Since the key property is maintained at all times during the process, then at the end we have achieved our goal: Each $U_i$ lies on $\ell_i$ and is contained in the interior of $D^i$. Also, $\ell_i$ goes through $p_i$ and halves $U \setminus U_i$.

\bigskip
\noindent{\sc Step 4 } {\sl Flattening $C_i$ towards $U_i$.}
\bigskip

Finally, for $i \in I$, we affinely flatten each $C_i$ towards $U_i$ to obtain its final position. Again, if $s_i=1$, then $C_i=\{p_i\}$. For each $0 \leq \epsilon \leq 1$ and each $i\in I$, let $C_i(\epsilon)$ be the set obtained from $C_i$ by orthogonally moving its points towards $U_i$ reducing their distance to $\ell_i$ by a factor of $\epsilon$. (If $s_i=1$, then $C_i(\epsilon)=\{p_i\}$. For each $i$, measure the distances from all points in $\bigcup_{j\neq i}{S_j(\epsilon)}$ to $\ell_i$, making it negative if the point and its corresponding point in $U$ are on different sides of $\ell_i$. Let $f(\epsilon)$ be the minimum of these distances for fixed $\epsilon$ and over all $i\in I$. Note that the function $f$ is continuous and $f(0)>0$ as $\bigcup_{i=1}^m{C_i(0)}=U$. Then there must be an $\epsilon^\prime > 0$ such that $f(\epsilon^\prime)>0$. The final position of $C_i$ is $C_i(\epsilon^\prime)$. Let $C:=\bigcup_{i=1}^m{C_i}$. Since each $C_i$ is contained in $D^i$, then $C$ satisfies the inherited order type property and the halving property. And because $U$ satisfies the halving property then $C$ also satisfies it. The fact that each $C_i$ is an affine copy of $S_i$, preserving this ways order types,  will allow us to count the number of crossings in $C$.

\subsection{Keeping $3$-symmetry and $3$-decomposability}

Let $\theta$ be the counterclockwise rotation of $2 \pi /3$ around the origin. We say that the input set $(P, \{ \beta_i \}_{i\in I}, \{S_i \}_{i=1}^m)$ is \emph{$3$-symmetric} if: the base point-set $P$ is $3$-symmetric, say via the function $\theta$, the pre-halving set of lines $\{ \beta_i \}_{i\in I}$ is $3$-symmetric under the same function $\theta$, and the collection of cluster models $\{ S_i \}_{i=1}^m$ is partitioned into orbits of equal clusters according to the function $\theta$. That is, if $p_i=\theta (p_j)=\theta^2 (p_k)$, then $\beta_i=\theta (\beta_j)=\theta^2 (\beta_k)$ and $S_i=S_j=S_k$.

Similarly, we say that the input set is \emph{$3$-decomposable}, if the base point-set $P$ is $3$-decomposable, with partition $A$, $B$, and $C$, and if the collection of cluster models satisfies that
\[
\sum_{i:p_i \in A}{s_i}=\sum_{j:p_j \in B}{s_j}=\sum_{k:p_k \in C}{s_k}.
\]
Note that no assumption is made on the pre-halving set of lines.

The following observations are worth highlighting.

\begin{remark}\label{keep3s}
If the input set is $3$-symmetric, then the construction can be
performed so that the resulting augmented point set $C$ is
$3$-symmetric. Similarly, if the input set is $3$-decomposable, then
the construction can be performed so that the resulting augmented
point set $C$ is $3$-decomposable.
\end{remark}

\subsection{Counting the crossings in the augmented drawing}

Now we count the number of crossings in the resulting point set $C=\bigcup_{i=1}^m C_i$, equivalently, the number of convex quadrilaterals $\square(C)$. The most important aspect of the calculation is that it only depends on the input set, that is, on the base point set $P$, the cluster models $S_i$, and the collection of pre-halving lines. Thus the number of crossings in the augmented drawing can be calculated (perhaps using a computer) without explicitly doing the construction. This is particularly useful in Section \ref{doubling}, where we iterate this construction and, as a consequence, we obtain
the currently best general drawings of $K_n$.

\subsubsection{A closer look into how clusters get splitted}

Before going into the calculation, we introduce some terminology. If
$p_i$ is simple (respectively, splitting), then we say that $C_i$
itself is {\em simple} (respectively, {\em splitting}). If $C_i$ is
simple, then each $C_j$ with $ i\neq j$ is completely contained in a
semiplane of $\ell_i$. If $C_i$ is splitting, then the same holds
except for the cluster $C_{\sigma(i)}$: a nonempty subset $L_i$ of
$C_{\sigma(i)}$ is on the left semiplane of $\ell_i$, while the also
nonempty subset $R_i = C_{\sigma(i)} \setminus L_i$ is on  the right
semiplane.  We remark that $L_i$ and $R_i$ are {\em not} subsets of
$C_i$, but of $C_{\sigma(i)}$. By convention, if $C_i$ is simple, so
that $\sigma(i)$ is not defined, then we let $L_i = R_i =
\emptyset$.

Note that the previously defined set $\ll(i)$ (respectively, $\rr(i)$) coincides with the set of those $j$ such that $C_j$ is completely contained in the left (respectively, right) semiplane of $\ell_i$. Thus, if $C_i$ is simple, then $\ll(i) \cup \rr(i) = \{1,2,\ldots,m\}$, and if $C_i$ is splitting, then $\ll(i) \cup \rr(i) = \{1,2,\ldots,m\} \setminus\{\sigma (i)\}$. We also remark that the sizes of $L_i$ and $R_i$ are fully determined by $\sum_{j\in\ll(i)} s_j$ and $s_i$. Indeed, the left semiplane of $\ell_i$ contains $\tcl{(n-s_i)/2}$ points of $C\setminus C_i$, $\sum_{j\in\ll(i)} s_j$ of which belong to a $C_j$ other than $C_{\sigma(i)}$.  Therefore, $|L_i|= \tcl{(n-s_i)/2} - \sum_{j\in\ll(i)} s_j$. The size of $R_i$ is analogously calculated.

\subsubsection{The calculation of crossings}

We now count the number of crossings in $D_n$, that is, the number $\square(C)$ of convex quadrilaterals defined by points in $C$. We count separately five different types of convex quadrilaterals contributing to $\square(C)$. Adding the five contributions gives the exact value of $\square(C)$.

\vglue 0.3 cm
\noindent{\bf Type I } {\sl Convex quadrilaterals whose points all
  belong to different clusters.}
\vglue 0.3 cm

It follows from the inherited order type property that the number of quadrilaterals of Type I is:
\begin{equation*}
\sum_{\substack{ i<j<k<\ell \\ p_{i},p_{j},p_{k,}p_{\ell}\text{ is a
convex quadrilateral} }}s_{i}s_{j}s_{k}s_{\ell}.
\end{equation*}

\vglue 0.3 cm
\noindent{\bf Type II } {\sl Convex quadrilaterals whose points belong to three distinct clusters.}
\vglue 0.3 cm

Every convex quadrilateral of Type II has two points in a cluster $C_i$ and the other two points in clusters $C_j, C_k$, with $i,j,k$ pairwise distinct.  Now any four such points define a convex quadrilateral if and only if the points in $C_j$ and $C_k$ are on the same semiplane determined by $\ell_i$.  Recalling that the set of points in $C\backslash C_i$ on the left (respectively, right)  halfplane of $\ell_i$ is $(\bigcup_{j\in\ll(i)}
  C_j) \cup L_i$ (respectively, $(\bigcup_{j\in\rr(i)} C_j) \cup R_i$), it follows that the total number of convex quadrilaterals of Type II equals:
\begin{equation*}
\sum_{i=1}^m \  \binom{s_i}{2} \biggl( \ \sum_{\substack{
j,k\in\ll(i) \\ j < k}} s_j s_k + \sum_{j\in\ll(i)} s_j |L_i| +
\sum_{\substack{j,k\in\rr(i) \\ j < k}} s_j s_k+ \sum_{j\in\rr(i)}
s_j |R_i| \ \biggr).
\end{equation*}

\vglue 0.3 cm
\noindent{\bf Type III } {\sl Convex quadrilaterals whose points belong to two distinct clusters, with two points in each cluster.}
\vglue 0.3 cm

For each fixed $C_i$, and points $p,q$ in $C_i$, $p$ and $q$ define a convex quadrilateral of Type III with those pairs of points that are on the same $C_j$ and on the same halfspace of $\ell_i$, except when $i = \sigma(j)$ and one of $p$ and $q$ belongs to $L_j$ and the other to $R_j$. Thus the number of convex quadrilaterals of Type III that involve two points in $C_i$ is $\binom{s_i}{2} \left(\sum_{j\notin\{i,\sigma(i)\}} \binom{s_j}{2} + \binom{|L_i|}{2} + {|R_i| \choose 2}\right) - \sum_{ j:i=\sigma(j)} {s_j \choose 2} |L_j||R_j| $.  When summing over all $i$, each convex quadrilateral of Type III gets counted exactly twice. Thus the total number of convex quadrilaterals of Type III is:
\begin{equation*}
\frac{1}{2} \sum_{i=1}^m \left( \binom{s_i}{2} \left(
\sum_{j\notin\{i,\sigma(i)\}} \binom{s_j}{2} + \binom{|L_i|}{2} +
{|R_i|
  \choose 2}\right) -
\sum_{\substack{j: \\ i=\sigma(j) }} {s_j \choose 2} |L_j||R_j|
\right).
\end{equation*}

\vglue 0.3 cm
\noindent{\bf Type IV } {\sl Convex quadrilaterals with three points in the same cluster and the other point in a distinct cluster.}
\vglue 0.3 cm

To count these crossings we need to introduce a bit of terminology. If $S$ is a point set in general position in the plane, and $p=(p_x,p_y),q=(q_x,q_y),r=(r_x,r_y)\in S$, with $p_x < q_x < r_x$, then the concatenation of the segments $\overline{pq}$ and $\overline{qr}$ is either concave up or concave down.  In the former case, we say that $\{p,q,r\}$ is itself {\em concave up}, and in the latter case, we say it is {\em concave down}. We let $\sqcup(S)$ (respectively, $\sqcap(S)$) denote the number of $3$-subsets of $S$ that are concave up (respectively, concave down).  If no two points in $S$ have the same $x$-coordinate, then each $3$-subset of $S$ is
either concave up or concave down, and so in this case $\sqcup(S) + \sqcap(S) = \binom{|S|}{3}$.

Now it follows from the construction of the clusters $C_i$, that given any $3$ points $p,q,r \in C_i$, then a fourth point $s$ in
another cluster forms a convex quadrilateral with $p,q$, and $r$ if and only if either (i) $s$ is in the left semiplane of $\ell_i$ and
$\{p,q,r\}$ is concave up in $S_i$; or (ii) $s$ is in the right semiplane of $\ell_i$ and $\{p,q,r\}$ is concave down in $S_i$.

Since there are $\tcl{(n-s_i)/2}$ points $s$ in $C\setminus C_i$ in the left halfspace of $\ell_i$, and $\tfl{(n-s_i)/2}$ points $s$ of
$C\setminus C_i$ in the right halfspace of $\ell_i$, it follows that the total number of quadrilaterals of Type IV equals:

\begin{equation*}
\sum_{i=1}^m  \left( \ \sqcup(S_i) \cdot \left\lceil
\frac{n-s_{i}}{2}\right\rceil
 + \sqcap(S_i) \cdot
\left\lfloor \frac{n-s_{i}}{2}\right\rfloor \right).
\end{equation*}

\vglue 0.3 cm
\noindent{\bf Type V } {\sl Convex quadrilaterals with all four points in the same cluster.}
\vglue 0.3 cm

This is simply the sum of the number of convex quadrilaterals in each $C_i$, or equivalently, in each $S_i$:
\begin{equation*}
\sum_{i=1}^m \square(C_i)  = \sum_{i=1}^m \square(S_i).
\end{equation*}

\section{Doubling all points of a set with an odd number of points}\label{doubling}

There is a case in which the construction from Section~\ref{gencon} is particularly useful: when the cluster models are all equal to each other. This is the approach followed by Aichholzer et al.~\cite{aak} and by \'Abrego and Fern\'andez-Merchant~\cite{af1}.

In ~\cite{aak}, the equivalent of our $\ell_i$s do not split any cluster, and the cluster models are sets in convex position called {\em lens arrangements}. This is the best possible choice (under the   no-splitting assumption) to minimize the number of crossings of the augmented point set.

In~\cite{af1}, clusters of size $2$ are used in an iterative process, starting from a base point set with $m$ points, and producing augmented point sets with $2^k m$ points for $k = 0, 1, \ldots$.  This has been used to obtain the best upper bounds known for the rectilinear crossing number prior to the present work. The only limitations of the process in~\cite{af1} are that (i) the base configuration $P$ is assumed to have an even number of points; and (ii) the base configuration $P$ is assumed to have a {\em halving matching}, that is, an injection from $P$ to the set of halving lines of $P$, such that each $p\in P$ gets mapped to a line incident with $p$. The base for this iterative process is the following result.

\bigskip
\noindent{\bf Lemma 3 in~\cite{af1}} If $P$ is an $m$-element set, $m$ even, and $P$ has a halving-line matching, then there is a point set $Q=Q(P)$ in general position, $|Q|=2m$, $Q$ also has a halving-line matching, and $\square(Q) = 16\square(P) + (m/2)(2m^2 - 7m + 5)$.
\bigskip

As in~\cite{af1}, we now use clusters of size $2$, but within the more general framework described in the previous section, we can use a base
configuration with an {\em odd} number of points. This also has the advantage that the existence of a pre-halving set of lines is trivially satisfied. Moreover, after one iteration, we get a set with an even number of points and a halving matching, allowing us to use the iterative construction in~\cite{af1}.

 \begin{proposition}\label{baseodd}
Starting from any point set $P$ with $m:=|P|$ odd, and duplicating each point (that is, substituting each point by a $2$-point cluster), our construction yields a $2m$-point set $C$ in general position with $\square(C) = 16\square(P) + (m/2)(2m^2-7m+5)$.  Moreover, $C$ has a halving matching.
\end{proposition}

\noindent{\em Proof.}  To apply our construction, we first need to check the existence of a pre-halving set of lines. This is trivial because $s_i = 2$ for every $i=1,\ldots,m$. That is, it suffices to choose, for each $p_i$, a line $\beta_i$ through $p_i$ that leaves $(m-1)/2$ points of $P$ on
each side. Moreover, such a line is simple, and thus $L_i=R_i=\emptyset$. Knowing the existence of a pre-halving set of lines, we may proceed to calculate the number of convex quadrilaterals in the augmented $2m$-set $C$.

\begin{itemize}
\item {\bf Type I.} Since $s_i=2$ for each $i$, then $C$ has $16\square(P)$ convex quadrilaterals of Type I.
\item {\bf Type II.} For each $i$, the line $\ell_i$ has exactly $(m-1)/2$ clusters $C_j$ on each side.  Thus $C$ has
  $\sum_{i=1}^m \binom{2}{2} \left( \binom{(m-1)/2}{2} \cdot 4 + \binom{(m-1)/2}{2} \cdot 4 \right) = m(m-1)(m-3)$ convex quadrilaterals of Type II.
\item {\bf Type III.} For each $i$, $\sigma(i)$ is undefined and $L_i=R_i=\emptyset$. Thus $C$ has
  $\frac{1}{2} \sum_{i=1}^m \binom{2}{2}\left(\sum_{j\neq i} \binom{2}{2}\right)$ $ = \frac{1}{2} m(m-1)$ convex quadrilaterals in $C$ of Type III.
\item {\bf Types IV and V.} Since there are no clusters of size $3$ or larger, then $C$ has no convex quadrilaterals of Types IV or V.
\end{itemize}

Summing up the contributions of Types I, II, and III, it follows that $\square(C)=16\square(P) + (m/2)(2m^2-7m+5)$, as claimed.

Finally, we show that $C$ has a halving matching. If $\ell$ is a directed line that spans points $p$ and $q$, then $p$ is {\em before} $q$ {\em in} $\ell$ if as we traverse $\ell$, first we find $p$ and then $q$. Recall that in the last step in the construction we start with all points in each cluster $C_i$ lying on line $\ell_i$, and perturb them so that the order type of $C_i$ coincides with that of $S_i$.  Since here all clusters have size $2$, there is no need to perturb them: their final position may as well be on $\ell_i$. For each $p_i\in P$, we let $p_i', p_i''$ denote the two
points in $C$ into which $p_i$ get splitted, labelled so that $p_i'$ is before $p_i''$ in $\ell_i$. We assume without any loss of generality that all lines $\ell_i$ are directed so that their angles with the $x$-axis are between $0$ and $\pi$.

Now $\ell_i$ is clearly a halving line for every $i$. Thus we may associate $\ell_i$ to one of $p_i'$ and $p_i''$, and only need to seek a halving line to associate to the other point. We rotate $\ell_i$ counterclockwise around $p_i'$ until we hit another point in $C$ (say $q$), and let $\overline{\ell_i'}$ denote the line through $p_i'$ and $q$, with the direction it naturally inherits from $\ell_i$.
 If $q$ is before $p_i'$ in $\alli$, then let $\elli:=\alli$. Otherwise, let $\elli$ denote the line spanning $p_i''$ and $q$ with the orientation it naturally inherits  from $\ell_i$, that is, so that $q$ is before $p_i''$ in $\elli$. In either case, $\elli$ is a halving line that goes through one of $p_i'$ or $p_i''$.  We associate this halving line to the point in $\{p_i', p_i''\}$ belonging to it, and to the other point we associate $\ell_i$.  It is easily checked that if $i\neq j$, then $\elli \neq \ellj$ (and trivially $\ell_i \neq \ell_j$). Therefore this defines an injection from $C$ to the set of its halving lines. Thus $C$ has a halving matching, as claimed.  \hal

\vglue 0.4 cm

Now, we use Proposition~\ref{baseodd} to prove Theorem~\ref{baseoddite}, which together with Theorem 2 in~\cite{af1} gives
\begin{equation}\label{qstar}
q_* \le \frac{24\rcr(P) + 3m^3 - 7m^2 + (30/7)m}{m^4},
\end{equation}
for any $m$-set $P$ in general position with either $m$ odd, or $m$ even and $P$ with a halving matching.
\vglue 0.4 cm

\noindent{\textbf{Proof of Theorem~\ref{baseoddite}.}  We closely follow the proof of Theorem 2 in~\cite{af1}. (Note that Lemma 3 in~\cite{af1}, the equivalent to our Proposition~\ref{baseodd}, may also be derived from the construction in Section~\ref{gencon}).

Applying Proposition~\ref{baseodd} to $P_{-1}:=P$, we obtain an even cardinality point set $P_0$ with a halving matching. Thus, we can apply iteratively Lemma 3 in~\cite{af1} with $P_0$ as the base configuration. Then, for all $k>0$, if $P_k$ denotes the set obtained from $P_{k-1}$ using Lemma 3 in~\cite{af1}, we have
\begin{equation}\nonumber
\rcr(P_k) = 16 \rcr(P) + m^3 8^{k-1} (2^k -1) -\frac{7}{6}m^2 4^{k-1}
(4^k - 1) + \frac{5}{14} m 2^{k-1} (8^k - 1).
\end{equation}
Now by letting $n:=|P_k| = 2^k m$, we get
\begin{equation}\nonumber
\rcr(P_k ) = \left(\frac{24\rcr(P) + 3m^3 - 7m^2 +
(30/7)m}{24m^4}\right) n^4 - \frac{1}{8} n^3 + \frac{7}{24} n^2 -
\frac{5}{28} n. \ \hal
\end{equation}

We cannot overemphasize the importance of Theorem~\ref{baseoddite}
and Theorem 2 in~\cite{af1}: they constitute the best tools
available to obtain upper bounds for the rectilinear crossing number
constant $q_*$.  As of the time of writing, the best bound known for
$q_*$, namely
\begin{equation}\nonumber
q_* \leq \frac{83247328}{218791125}  < 0.380488,
\end{equation}
is obtained by applying
Theorem~\ref{baseoddite} to a particular drawing of $K_{315}$. See
Section~\ref{symdra}.


\section{Symmetric geometric drawings}\label{symdra}

The most fruitful and comprehensive effort to produce good geometric
drawings of $K_n$ is the Rectilinear Crossing Number Project, led by
Oswin Aichholzer~\cite{aweb2}.  Prior to the present work, the
drawings in~\cite{aweb2} constitute the state-of-the-art in the
subject: for every $n \le 100$, the previously best crossing-wise
geometric drawing of $K_n$ can be found in~\cite{aweb2}. A detailed
look at the information in~\cite{aweb2} shows that the {\em vast majority} of
drawing seems close to being $3$-symmetric.

We have successfully produced $3$-symmetric and $3$-decomposable
drawings that match or improve the best drawings reported
in~\cite{aweb2}. Our results are summarized as follows.

\begin{description}
\item{(1)}
For every positive integer $n < 100$, $n$ a multiple of  $3$, we
produced a $3$-symmetric and $3$-decomposable geometric drawing of
$K_n$ whose number of crossings is less than or equal to that
in~\cite{aweb2}. Some of these drawings were obtained using
heuristic methods based on previous drawings, and the rest using our
replacing-by-clusters construction in Section~\ref{gencon}. For a
brief summary of our results, see Table~\ref{summary}.
\item{(2)}
The best upper bound for the rectilinear crossing number constant
$q_*=\lim_{n\to\infty}\rcr(K_n)/\binom{n}{4}$ is now achieved by
$3$-symmetric and $3$-decomposable drawings. For this we apply
Theorem~\ref{baseoddite} to a $3$-symmetric and $3$-decomposable
drawing of $K_{315}$ with $152210640$  crossings, and recall
Remark~\ref{keep3s}.
\end{description}

Trying to produce $3$-symmetric geometric drawings of $K_{n}$ that
improve those of Aichholzer is a formidable task, specially for
large values of $n$. Prior to our work, no good crossing-wise
$3$-symmetric drawings had been reported, other than those for very
small values of $n$. For each positive integer $n$ multiple of 3, we
produced 3-symmetric drawings of $K_{n}$ whose number of crossings
is less than or equal to the previous best drawing. Our drawings are
optimal for $n\leq 27$~\cite{k27}, and we conjecture they are
optimal for $n=36$, $39$, and $45$. The drawings for $n\leq 57$,
with the exception of $n=33$, were obtained independently. A good
sample of these drawings is our $3$-symmetric drawing of $K_{24}$,
sketched in Figure~\ref{fig:k24}. The precise coordinates of the
eight points in one wing $W$ are: $ p_1 = (-51, 113); p_2 = (6,
834); p_3 = (16, 989); p_4 = (18, 644); p_5 =(18, 1068); p_6 = (22,
211); p_7 = (-26, 313); p_8 = (17, 1036)$. If $\theta$ denotes the
counterclockwise rotation of $2 \pi /3$ around the origin, then the
whole $24$-point set is $P=W\cup \theta(W) \cup \theta^2 (W)$.

\begin{figure}[ht]
\begin{center}
\includegraphics[width=6.5in]{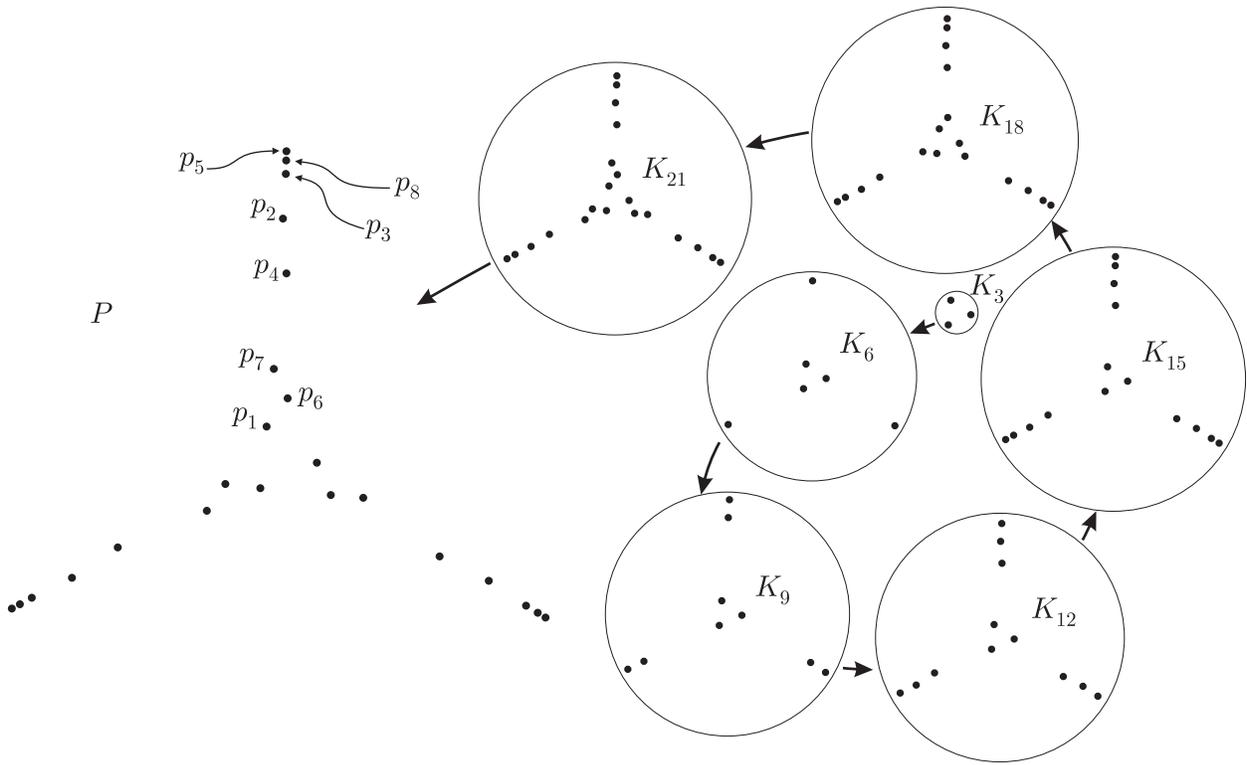}
\end{center}
\caption{The underlying vertex set of an optimal $3$-symmetric
geometric drawing of $K_{24}$. This point set contains optimal
nested $3$-symmetric drawings of $K_{21}, K_{18}, K_{15}, K_{12},
K_9,K_6$, and $K_3$.} \label{fig:k24}
\end{figure}

The geometric drawing induced by this point-set has $3699$
crossings, and is thus optimal~\cite{k27}. A remarkable property of
this drawing is that it contains a chain of optimal $3$-symmetric
subdrawings of $K_{21}, K_{18}, K_{15}, K_{12}, K_9,K_6$, and $K_3$.
Indeed, if $W_i=\{p_1,p_2,\ldots,p_i\}$ then the point-set $W_i\cup
\theta(W_i) \cup \theta^2 (W_i)$ is an optimal drawing of $K_{3i}$
for $1\leq i\leq 8$, that is, its number of crossings matches the
one known to be optimal (see~\cite{k27} and~\cite{agor}).

We also include $3$--symmetric drawings of 
$K_{27}$ and $K_{30}$ (Figure~\ref{fig:s27}), 
$K_{36}$ and $K_{39}$ (Figure~\ref{fig:s36}), and
$K_{45}$ (Figure~\ref{fig:s45}).  
The drawing of $K_{27}$ is known to be optimal~\cite{k27}. 
For reasons that are beyond the scope of
this work, we firmly believe that the given drawings of $K_{30}$,
$K_{36}$, $K_{39}$, and $K_{45}$  are also optimal. 

\begin{figure}[htb!]
\begin{center}
\includegraphics[width=6.5in]{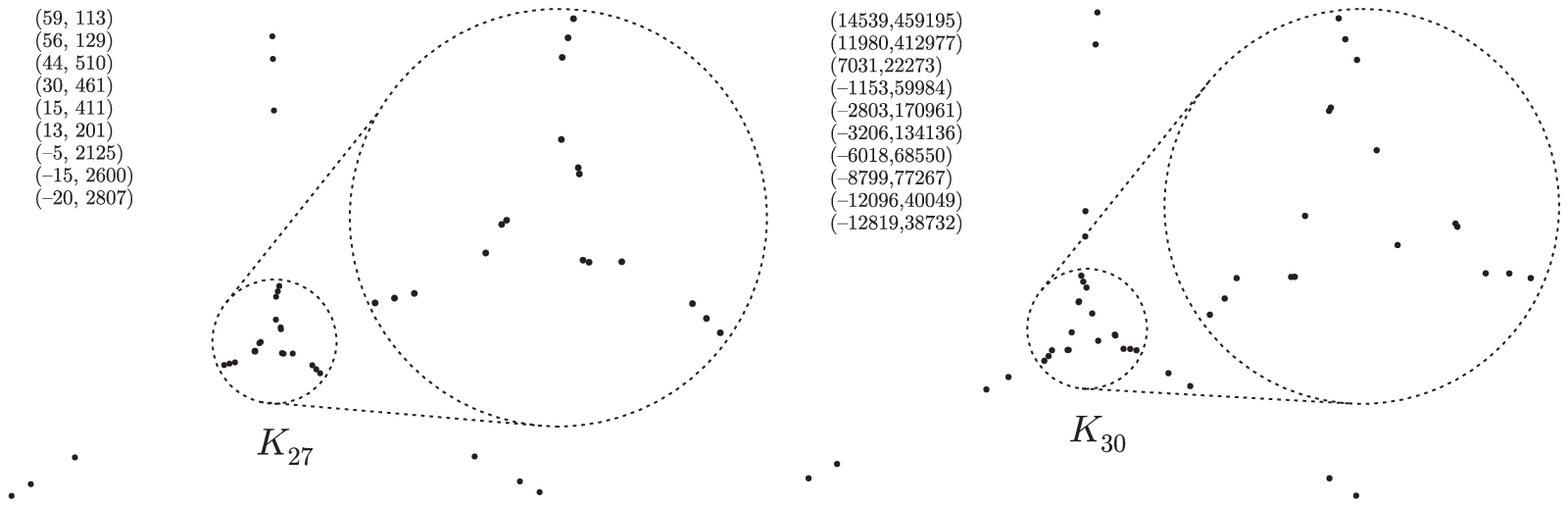}
\end{center}
\caption{The underlying vertex sets of $3$-symmetric
geometric drawings of $K_{27}$ (left) and $K_{30}$ (right). In each
case, the
coordinates given correspond to one third of the points; the other two
thirds are obtained by rotating the given set $120$ and $240$ degrees.
The
induced drawing of $K_{27}$ is known to be optimal, and we conjecture that the
induced drawing of $K_{30}$ is also optimal.} 
\label{fig:s27}
\end{figure}

\vglue 1 cm 

\begin{figure}[htb!]
\begin{center}
\includegraphics[width=6.5in]{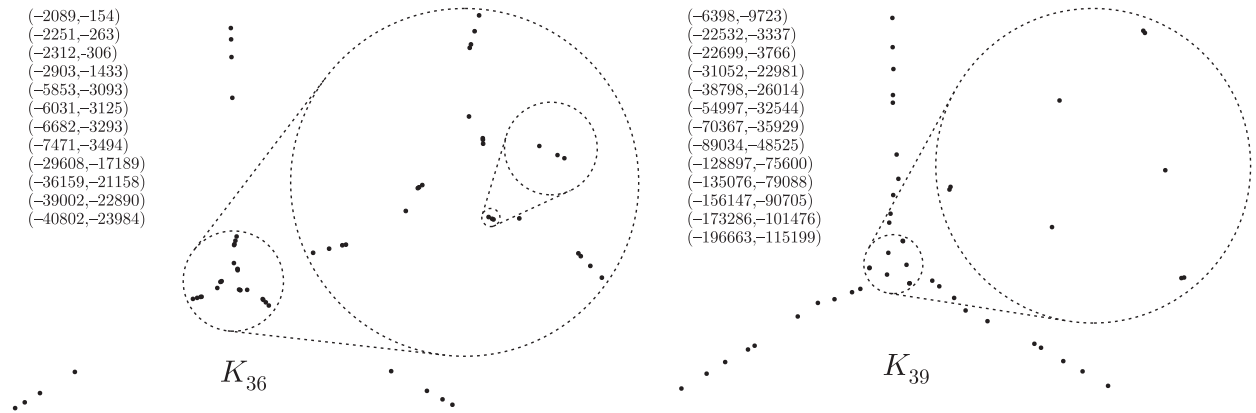}
\end{center}
\caption{The underlying vertex sets of $3$-symmetric
geometric drawings of $K_{36}$ (left) and $K_{39}$ (right), both of
which we conjecture are optimal. In each
case, the
coordinates given correspond to one third of the points; the other two
thirds are obtained by rotating the given set $120$ and $240$ degrees.}
\label{fig:s36}
\end{figure}

\vglue 1 cm

\begin{figure}[htb!]
\begin{center}
\includegraphics[height=2in]{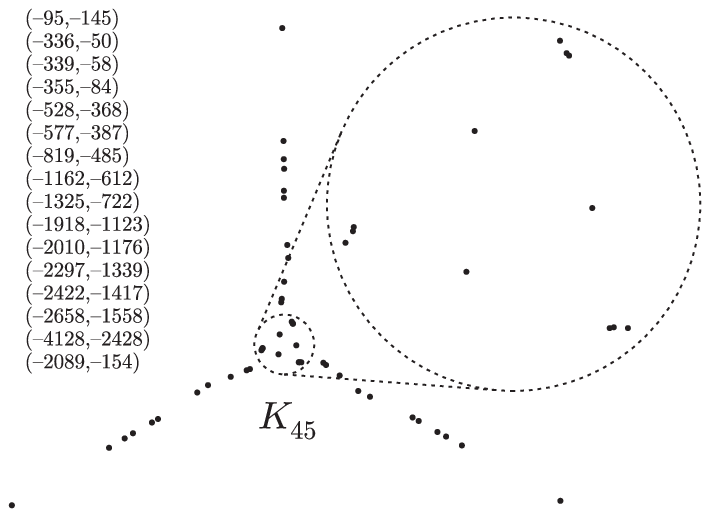}
\end{center}
\caption{The underlying vertex set of a $3$-symmetric
geometric drawing of $K_{45}$, which we conjecture is optimal. The
coordinates given correspond to one third of the points; the other two
thirds are obtained by rotating the given set $120$ and $240$ degrees.}
\label{fig:s45}
\end{figure}


In the Appendix we give $3$--symmetric drawings of
$K_{42}$, $K_{48}$, $K_{51}$, $K_{54}$, and $K_{57}$. 


To obtain the drawings for $n\geq 60$, and for the special case
$n=33$, we use the construction in Section~\ref{gencon}. For each
such $K_{n}$, it suffices to give the base drawing $D_{m}$ for some
suitable $m<n$, the cluster models $S_{i}$, and a pre-halving set of
lines $\{\beta_i\}_{i\in I}$ for $D_{m}$. This determines the
information relevant to calculate the number of crossings of the
resulting drawing of $K_{n}$: the sizes of the clusters that lie to
the left of each line $\ell _{i}$, and the sizes of the sets $L_{i}$
and $R_{i}$ of the cluster (if any) that is splitted by $\ell _{i}$.
We use a base drawing of $K_{30}$ to obtain drawings for $K_{33}$
and $K_{60}$, and a base drawing of $K_{51}$ to obtain drawings of
$K_{n}$ with $60<n<100$ and for $n=315$.  

The details are given in the Appendix.

\renewcommand{\arraystretch}{0.85}

\begin{table}[h]
\centering
\begin{tabular}{|c|c|c|c|}\hline
 &{\small  Number of crossings}  & {\small Number of crossings}   &
  {\small How we obtained the}
  \\
{\small $n$} & {\small in previous}  & {\small in currently best} &
{\small drawing reported} \\
& {\small best drawing~\cite{aweb2}} & {\small $3$-symmetric drawing}
&
{\small in the third column} \\ \hline
\raisebox{-1ex}{\small $n \le 27$,} &   & &  \\
{\small $n$ divisible by $3$} &  
\raisebox{1ex}{\small Optimal for
  each $n$} & \raisebox{1ex}{\small Optimal for each $n$}  & \raisebox{1ex}{\small Independently}  \\ \hline
 {\small 30} & {\small 9726} & {\small 9726} & {\small Independently} \\ \hline 
{\small 33} & {\small 14634} & {\small 14634} & {\small From $K_{30}$} \\ \hline
 {\small 36} &  {\small 21175} &  {\small 21174}  &  {\small Independently} \\ \hline
 {\small 39} &  {\small 29715} &  {\small 29715} &  {\small Independently} \\ \hline
 {\small 42} &  {\small 40595}&  {\small 40593}  &  {\small Independently} \\ \hline
 {\small 45} &  {\small 54213}&  {\small 54213} &  {\small Independently} \\ \hline
 {\small 48} &  {\small 71025}&  {\small 71022} &  {\small Independently} \\ \hline
 {\small 51} &  {\small 91452} & {\small 91452} &  {\small Independently} \\ \hline
 {\small 54} &  {\small 115994} &  {\small 115977} &  {\small Independently} \\ \hline
 {\small 57} &  {\small 145178}&  {\small 145176} &  {\small Independently} \\ \hline
 {\small 60 } & {\small  179541} & {\small  179541 } &  {\small From $K_{30}$} \\ \hline
 {\small 63 } & {\small  219683} & {\small  219681 } & {\small From $K_{51}$} \\ \hline
 {\small 66 } & {\small  266188} & {\small  266181  } & {\small From $K_{51}$}  \\ \hline
 {\small 69 } & {\small  319737} & {\small  319731 } & {\small From $K_{51}$}\\ \hline
 {\small 72 } & {\small  380978} & {\small  380964 } & {\small From $K_{51}$}\\ \hline
 {\small 75 } & {\small  450550} & {\small  450540 } & {\small From $K_{51}$}\\ \hline
 {\small 78 } & {\small  529350} & {\small  529332 } & {\small From $K_{51}$}\\ \hline
 {\small 81 } & {\small  618048} & {\small  618018 } & {\small From $K_{51}$}\\ \hline
 {\small 84 } & {\small  717384} & {\small  717360 } & {\small From $K_{51}$}\\ \hline
 {\small 87 } & {\small  828233} & {\small  828225 } & {\small From $K_{51}$}\\ \hline
 {\small 90 } & {\small  951526} & {\small  951459 } & {\small From $K_{51}$}\\ \hline
 {\small 93 } & {\small  1088217} & {\small     1088055 } & {\small From $K_{51}$}\\ \hline
 {\small 96 } & {\small  1239003} & {\small     1238646 } & {\small From $K_{51}$}\\ \hline
 {\small 99 } & {\small 1405132} & {\small  1404552 } & {\small From $K_{51}$}\\ \hline
 {\small 315 } & {\small --   } & {\small 152210640}& {\small From $K_{51}$}\\ \hline
\end{tabular}
\caption{\small For each $n < 100$, $n$ a multiple of $3$, we have
found a $3$-symmetric and $3$-decomposable drawing whose number of
crossings is less than or equals to the number of crossings in the previously best geometric drawing of $K_n$.
We also include our current record for $K_{315}$, the drawing that
gives, in combination with Theorem~\ref{baseoddite}, $q_* <
0.380488$.} \label{summary}
\end{table}

\clearpage

\section{Appendix}

\subsection{A $3$--symmetric drawing of $K_{42}$ with $40593$ crossings}

\begin{figure}[htb!]
\begin{center}
\includegraphics[width=5in]{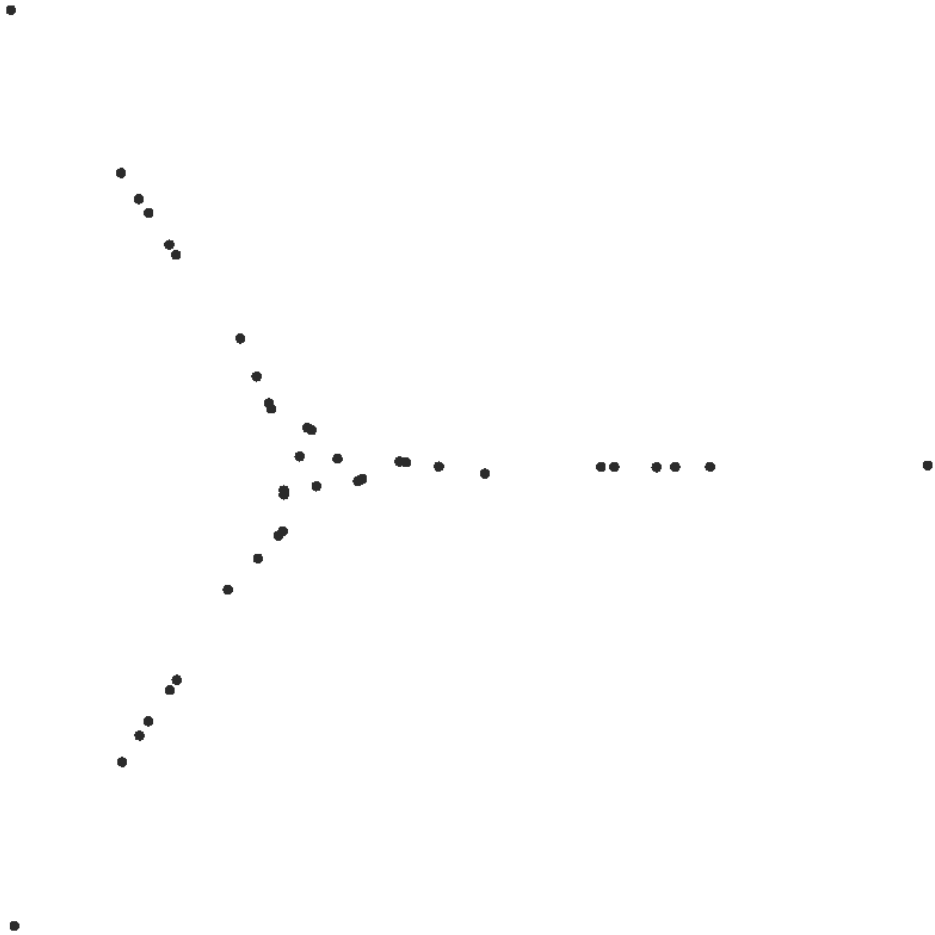}
\end{center}
\caption{The underlying vertex set of a $3$--symmetric
geometric drawing of $K_{42}$ with $40593$ crossings.} 
\label{fig:k42}
\end{figure}

Consider the $42$--point set obtained from the points
$p_{	1	}=(	620	,	308	)$, 
$p_{	2	}=(	1260	,	-504	)$, 
$p_{	3	}=(	1288	,	-482	)$, 
$p_{	4	}=(	1396	,	-427	)$, 
$p_{	5	}=(	2564	,	206	)$, 
$p_{	6	}=(	2775	,	173	)$, 
$p_{	7	}=(	3806	,	25	)$,
$p_{	8	}=(	5250	,	-229	)$, 
$p_{	9	}=(	8891	,	12	)$, 
$p_{	10	}=(	9315	,	10	)$, 
$p_{	11	}=(	10634	,	-6	)$, 
$p_{	12	}=(	11224	,	13	)$, 
$p_{	13	}=(	12322	,	21	)$, and 
$p_{	14	}=(	19157	,	64	)$,
plus the points obtained by rotating each of
these points $120$ and $240$ degrees around the origin.  See
Figure~\ref{fig:k42}. The induced geometric drawing of $K_{42}$ has
$40593$ crossings.

\clearpage

\subsection{A $3$--symmetric drawing of $K_{48}$ with $71022$ crossings}

\begin{figure}[htb!]
\begin{center}
\includegraphics[width=5in]{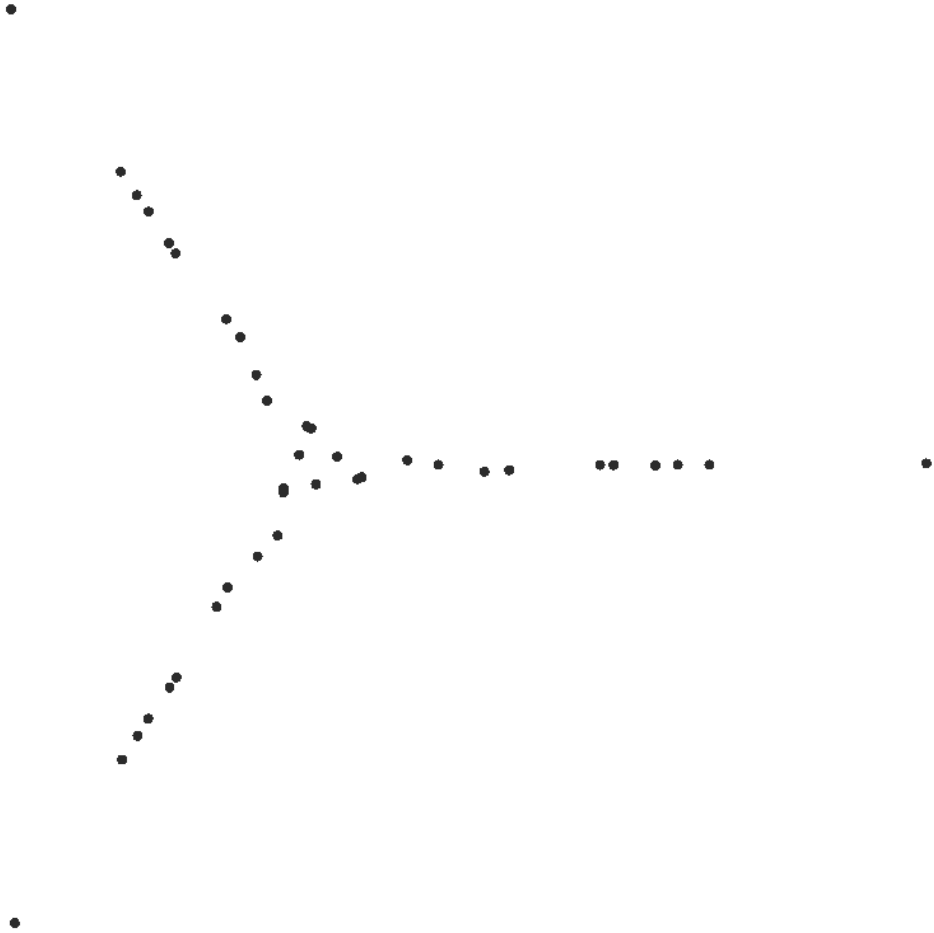}
\end{center}
\caption{The underlying vertex set of a $3$--symmetric
geometric drawing of $K_{48}$ with $71022$ crossings.} 
\label{fig:k48}
\end{figure}

Consider the $48$--point set obtained from the points
$p_{	1	}=(	-57807.48847	,	99345.28317	)$, 
$p_{	2	}=(	-57806.65857	,	99343.86617	)$, 
$p_{	3	}=(	-34105.90293	,	58848.08466	)$,
$p_{	4	}=(	-37110.08631	,	64005.82257	)$, 
$p_{	5	}=(	-31864.30787	,	55277.26387	)$, 
$p_{	6	}=(	-27997.58687	,	48376.53697	)$, 
$p_{	7	}=(	-26732.18287	,	46163.98867	)$, 
$p_{	8	}=(	-14558.27587	,	27959.08197	)$, 
$p_{	9	}=(	-17179.16207	,	31883.97347	)$, 
$p_{	10	}=(	-11528.14000	,	19697.46500	)$, 
$p_{	11	}=(	-9487.09731	,	14127.03628	)$, 
$p_{	12	}=(	-3461.52707	,	2301.65997	)$, 
$p_{	13	}=(	-3460.33257	,	2299.31657	)$, 
$p_{	14	}=(	-1969.55837	,	8536.56197	)$, 
$p_{	15	}=(	-1305.99477	,	8113.10777	)$, and  
$p_{	16	}=(	-1153.06188	,	8052.81507	)$,
plus the points obtained by rotating each of
these points $120$ and $240$ degrees around the origin.  See
Figure~\ref{fig:k48}. The induced geometric drawing of $K_{48}$ has
$71022$ crossings.

\clearpage

\subsection{A $3$--symmetric drawing of $K_{51}$ with $91452$ crossings}

\begin{figure}[htb!]
\begin{center}
\includegraphics[width=5in]{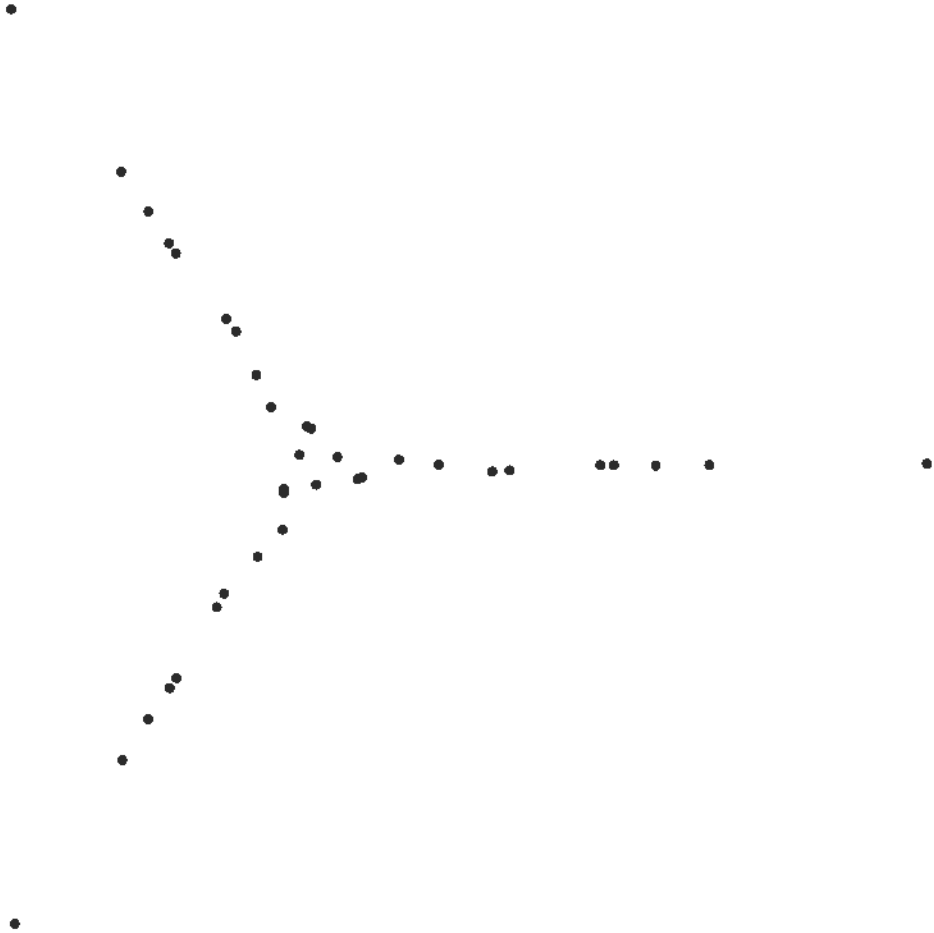}
\end{center}
\caption{The underlying vertex set of a $3$--symmetric
geometric drawing of $K_{51}$ with $91452$ crossings.} 
\label{fig:k51}
\end{figure}

Consider the $51$--point set obtained from the points $p_{ 1 }=(
3716.08787 , 1847.16703 )$, $p_{ 2 }=( 3723.66827 , 1846.89633 )$,
$p_{ 3 }=( 7559.84917 , -3018.73497 )$, $p_{ 4 }=( 7681.27767 ,
-2924.32337 )$, $p_{ 5 }=( 8372.80747 , -2555.43267 )$, $p_{ 6 }=(
15380.80127 , 1242.65413 )$, $p_{ 7 }=( 22830.08397 , 149.29793 )$,
$p_{ 8 }=( 22833.62767 , 150.01693 )$, $p_{ 9 }=( 32961.31257 ,
-1302.20837 )$, $p_{ 10 }=( 36202.07107 , -1066.09417 )$, $p_{ 11 }=(
53346.71877 , 75.35363 )$, $p_{ 12 }=( 55888.52997 , 69.24083 )$, $p_{
  13 }=( 63804.95917 , -36.22667 )$, $p_{ 14 }=( 63807.51607 ,
-36.12177 )$, $p_{ 15 }=( 73923.83417 , 125.04913 )$, $p_{ 16 }=(
73924.52987 , 125.05093 )$, and $p_{ 17 }=( 114944.97357 , 395.74573
)$, plus the points obtained by rotating each of these points $120$
and $240$ degrees around the origin.  See Figure~\ref{fig:k57}. The
induced geometric drawing of $K_{51}$ has $91452$ crossings.

\clearpage

\subsection{A $3$--symmetric drawing of $K_{54}$ with $115977$ crossings}

\begin{figure}[htb!]
\begin{center}
\includegraphics[width=4.5in]{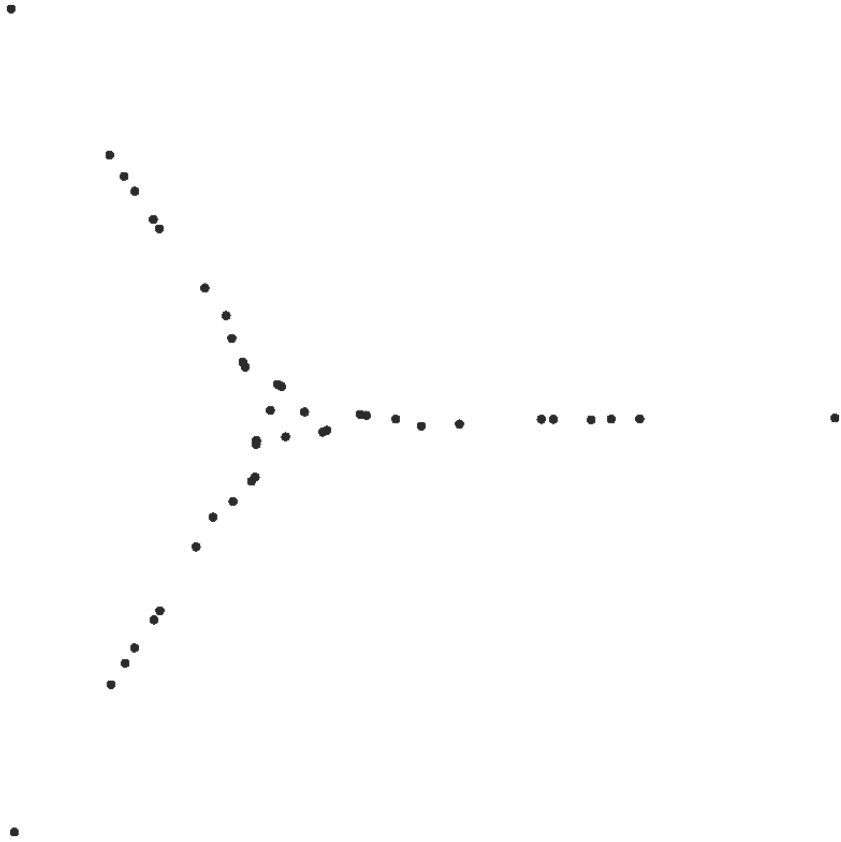}
\end{center}
\caption{The underlying vertex set of a $3$--symmetric
geometric drawing of $K_{54}$ with $115977$ crossings.} 
\label{fig:k54}
\end{figure}

Consider the $54$--point set obtained from the points
$p_{	1	}=(	-57807.48847	,	99345.28317	)$, 
$p_{	2	}=(	-57806.65857	,	99343.86617	)$, 
$p_{	3	}=(	-34105.90293	,	58848.08466	)$, 
$p_{	4	}=(	-37110.08631	,	64005.82257	)$, 
$p_{	5	}=(	-31864.30787	,	55277.26387	)$, 
$p_{	6	}=(	-27997.58687	,	48376.53697	)$, 
$p_{	7	}=(	-26732.18287	,	46163.98867	)$, 
$p_{	8	}=(	-17179.16207	,	31883.97347	)$, 
$p_{	9	}=(	-17177.09877	,	31880.90437	)$, 
$p_{	10	}=(	-12710.94699	,	25192.60584	)$, 
$p_{	11	}=(	-11528.14000	,	19697.46500	)$, 
$p_{	12	}=(	-9224.14377	,	13900.95197	)$, 
$p_{	13	}=(	-8764.40677	,	12704.76127	)$, 
$p_{	14	}=(	-3461.52707	,	2301.65997	)$, 
$p_{	15	}=(	-3460.33257	,	2299.31657	)$, 
$p_{	16	}=(	-1969.55837	,	8536.56197	)$, 
$p_{	17	}=(	-1305.99477	,	8113.10777	)$, and 
$p_{	18	}=(	-1153.06188	,	8052.81507	)$,
plus the points obtained by rotating each of these points $120$ and $240$ degrees around the origin.  See
Figure~\ref{fig:k54}. The induced geometric drawing of $K_{54}$ has $115977$ crossings.

\clearpage

\subsection{A $3$--symmetric drawing of $K_{57}$ with $145176$ crossings}

\begin{figure}[htb!]
\begin{center}
\includegraphics[width=4.5in]{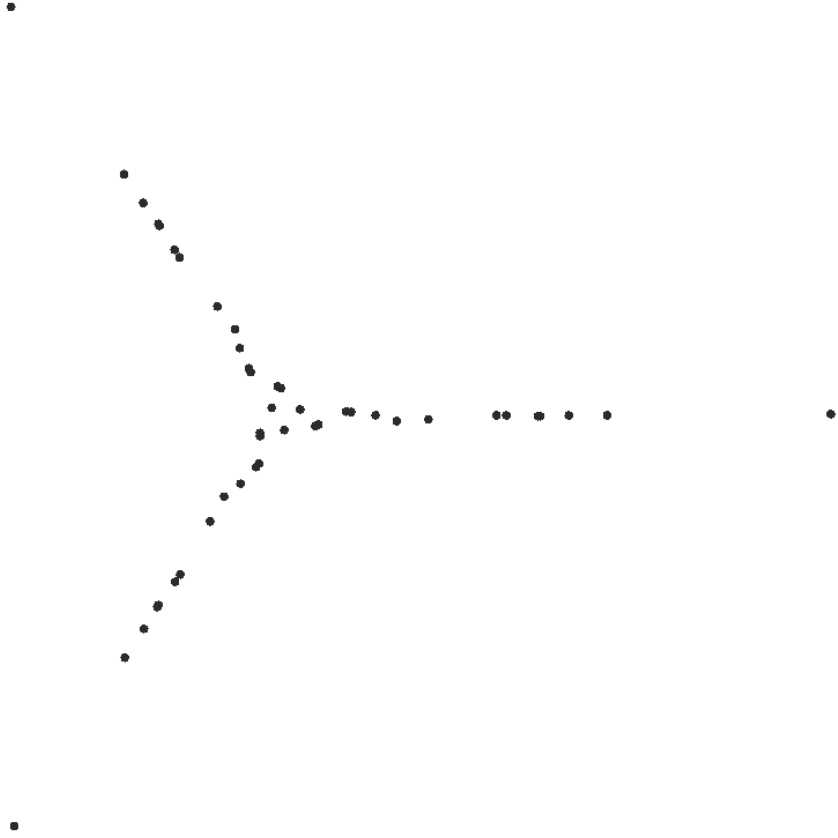}
\end{center}
\caption{The underlying vertex set of a $3$-symmetric
geometric drawing of $K_{57}$ with $145176$ crossings.} 
\label{fig:k57}
\end{figure}

Consider the $57$--point set obtained from the points
$p_{	1	}=(	-31817.67721	,	55426.14425	)$, 
$p_{	2	}=(	-69368.98616	,	119214.33860	)$, 
$p_{	3	}=(	-69367.99028	,	119212.63940	)$, 
$p_{	4	}=(	-40804.41177	,	70433.67069	)$, 
$p_{	5	}=(	-35943.52523	,	62061.44313	)$, 
$p_{	6	}=(	-32126,	55922	)$, 
$p_{	7	}=(	-28013.28687	,	48376.53697	)$, 
$p_{	8	}=(	-26778.48287	,	46163.98867	)$, 
$p_{	9	}=(	-17179.16207	,	31883.97347	)$, 
$p_{	10	}=(	-17177.09877	,	31880.90437	)$, 
$p_{	11	}=(	-12710.94699	,	25192.60584	)$, 
$p_{	12	}=(	-11528.14	,	19697.465	)$, 
$p_{	13	}=(	-9224.14377,	13900.95197	)$, 
$p_{	14	}=(	-8764.40677,	12704.76127	)$, 
$p_{	15	}=(	-3461.52707,	2301.65997	)$, 
$p_{	16	}=(	-3460.33257,	2299.31657	)$, 
$p_{	17	}=(	-1969.55837,	8536.56197	)$, 
$p_{	18	}=(	-1305.99477,	8113.10777	)$, and  
$p_{	19	}=(	-1153.06188,	8052.81507	)$,
plus the points obtained by rotating each of these points $120$ and $240$ degrees around the origin.  See
Figure~\ref{fig:k57}. The induced geometric drawing of $K_{57}$ has $145176$ crossings.

\clearpage

\subsection{(How to construct) A drawing 
of $K_{315}$ with $152210640$ crossings}

We describe how to obtain a drawing of $K_{315}$ with $152210640$
crossings using the construction technique in Section~\ref{gencon}.
As explained at the end of Section~\ref{symdra},  
it suffices to give a base drawing $D_{m}$ for some
suitable $m<n$ (equivalently, the underlying point set $P_m$), 
the cluster models $S_{i}$, $i=1,\ldots,m$, and a pre--halving set of
lines $\{\beta_i\}_{i\in I}$ for those points in $P_{m}$ that get
transformed into a cluster.  In this case, we work with
a base set with $51$ points, that is, $m=51$.

These ingredients are given below. The result is a drawing of
$K_{315}$ with $152210640$ crossings.

\subsubsection{The base point configuration}\label{t315}

We use as base configuration a $51$--point set
$P=\{p_1,p_2,\ldots,p_{51}\}$.  We give explicitly the coordinates of
$17$ of the $51$ points, and obtain the remaining $34$ points by
rotating each of these points $120$ and $240$ degrees around the origin.

Thus, we let:
$p_1=  (114935.3031, 381.37451)$,
$p_2=  (73931.7862, 127.25511)$, 
$p_3=  (67347.3942, 75.62961)$,
$p_4=  (63815.8559, -37.63049)$, 
$p_5=  (55899.7316, 58.88221)$,
$p_6=  (53352.4837, 69.45451)$, 
$p_7=  (36214.634, -1062.97569)$,
$p_8=  (31509.8338, -1373.94309)$, 
$p_9=  (22847.349, 151.00411)$,
$p_{10}=  (17043.162, 1175.66911)$, 
$p_{11}=  (16655.0717, 1034.97731)$,
$p_{12}=(15393.4257, 1230.20761)$, 
$p_{13}=(8387.4352, -2549.11369)$,
$p_{14}=  (7690.1479, -2921.61509)$, 
$p_{15}=  (7573.2312, -3011.73969)$,
$p_{16}=  (3717.1198, 1845.13511)$, and 
$p_{17}=  (3714.3655, 1845.37901)$.

We also let:
$p_{18}=\theta^2(p_{17})$,
$p_{19}=\theta^2(p_{16})$, 
$p_{20}=\theta(p_{15})$,
$p_{21}=\theta(p_{14})$,
$p_{22}=\theta(p_{13})$,
$p_{23}=\theta(p_{17})$,
$p_{24}=\theta(p_{16})$,
$p_{25}=\theta^2(p_{14})$,
$p_{26}=\theta^2(p_{15})$,
$p_{27}=\theta^2(p_{13})$,
$p_{28}=\theta^2(p_{12})$,
$p_{29}=\theta^2(p_{11})$,
$p_{30}=\theta^2(p_{10})$,
$p_{31}=\theta(p_{12})$,
$p_{32}=\theta(p_{11})$,
$p_{33}=\theta(p_{10})$,
$p_{34}=\theta^2(p_{9})$,
$p_{35}=\theta(p_{9})$,
$p_{36}=\theta(p_{8})$,
$p_{37}=\theta^2(p_{8})$,
$p_{38}=\theta(p_{7})$,
$p_{39}=\theta^2(p_{7})$,
$p_{40}=\theta^2(p_{6})$,
$p_{41}=\theta(p_{6})$,
$p_{42}=\theta^2(p_{5})$,
$p_{43}=\theta(p_{5})$,
$p_{44}=\theta(p_{4})$,
$p_{45}=\theta^2(p_{4})$,
$p_{46}=\theta^2(p_{3})$,
$p_{47}=\theta(p_{3})$,
$p_{48}=\theta^2(p_{2})$,
$p_{49}=\theta(p_{2})$,
$p_{50}=\theta^2(p_{1})$, and
$p_{51}=\theta(p_{1})$.

\subsubsection{The cluster models}\label{whereais}

The cluster models for those points that do not get augmented or get
augmented into a cluster of size $2$ or $3$ are trivial (any point
sets in general position work). 
For all other cases (clusters
of size $n$, $4\le n \le 12$), we have used as cluster models the
underlying point sets $A_n$ of the drawings of $K_n$ 
given in~\cite{aweb2}.
 We remark that by using other cluster models,
slightly better results can be obtained.

\bigskip

For $i=3, 46$, and $47$, we let $S_i$ have one point (there is no need
to specify its coordinates, as we mentioned above).

For $i = 10, 30$, and $33$, we let $S_i$ have two points  (there is no need
to specify its coordinates, as we mentioned above).

For $i = 11, 16, 19, 24, 29$, and $32$, we let $S_i=A_4:= 
\{
     (0      ,16865),
     (41470      ,13435),
     ( 2213        ,  0),$ $
     (24229      ,14674)\}$. 

For $i = 6, 17, 18, 23, 40$, and $41$, we let $S_i=A_5:=\{
(    56337,      50707)
  (       0,      38814), $ $
   (  42575,          0),  $ $
    ( 51990,      40716), $ $
    ( 30815,      21467)\}$. 

For $i = 8, 14, 15, 20, 21, 25, 26, 36$, and $37$, we let $S_i=A_6:=\{
(     31913,      61624), $ $
(         0     , 39366), $ $
(    13197,      35824), $ $
(   49018 ,         0), $ $
(  27438 ,     48183), $ $
( 34377 ,     27824), $ $
\}$.   

For $i = 5, 7, 12, 28, 31, 38, 39, 42$, and $43$, we let $S_i=A_7:= \{
(     10881,      31696),  $ $
(     36061,       6218), $ $
(      5214 ,     39717), $ $
(         0,      59285), $ $
(      8359,      24119), $ $
(        59,      26990), $ $
(     44957,          0)
\}$. 

For $i = 2, 9, 13, 22, 27, 34, 35, 48$, and $49$, we let $S_i=A_8:=\{
(     55255 ,     59712), $ $
(     16631 ,     25552), $ $
(     23666 ,     43408), $ $
(     26741 ,     44334), $ $
(     15615 ,         0), $ $
(      3227  ,    56082), $ $
(         0  ,    62548), $ $
(     12393,      15412)
\}$. 

For $i = 4, 44$, and $45$, we let $S_i:=A_{9}:=\{
(      15928,      20352), $ $
(     22642,      16618), $ $
(      3049 ,         0), $ $
(     18325,      13804),$ $
(     32948,      11155), $ $
(     15236,      11815), $ $
 (        0     , 29904), $ $
 (    30218,      12585), $ $
 (     3815,      27123)
\}$. 

For $i = 1, 50$, and $51$, we let $S_i=A_{12}:=\{
(     13290     , 30827), $ $
(     45233    ,  24125), $ $
(     10217   ,   11859), $ $
(      6294   ,       0), $ $
 (        0     , 49579), $ $
  (   13699      ,33996), $ $
(      2314      ,46508),$ $
 (    16411    ,  17184), $ $
  (   29175   ,   22801), $ $
   (  52500  ,    24275), $ $
    ( 24447 ,     26182), $ $
     ( 8784 ,      6906)\}$.

Thus, the set $I$ of those subscripts $i$ such that $s_i:=|S_i|>1$ is 
$
I= \{1,2,\ldots,51\} \setminus \{3,46,47\}$. 

\subsubsection{A pre--halving set of lines}

We finally define a pre--halving set of lines 
$\{\beta_i\}_{i\in I}$.

\begin{enumerate}

\item Let $\beta_1$ be the line that goes through $p_1$ and $p_2$,
  directed from $p_1$ towards $p_2$. 
Thus $\beta_1$ is splitting.
\item Let $\beta_2$ be the line that goes through $p_2$ and $p_4$,
  directed from $p_2$ towards $p_4$. 
Thus $\beta_2$ is splitting.
\item Let $\beta_4$ be the line that goes through $p_4$ and $p_7$,
  directed from $p_4$ towards $p_7$. 
Thus $\beta_4$ is splitting.
\item Let $\beta_5$ be the line that goes through $p_5$ and $p_7$,
  directed from $p_4$ towards $p_7$. 
Thus $\beta_5$ is splitting.
\item Let $\beta_{6}$ be the line that goes through $p_{6}$ with slope $-0.0072$.
Thus $\beta_6$ is simple.
\item Let $\beta_{7}$ be the line that goes through $p_{7}$ with slope $ 0.0656 $.
Thus $\beta_7$ is simple.
\item Let $\beta_{8}$ be the line that goes through $p_{8}$ with slope $ -0.17 $.
Thus $\beta_8$ is simple.
\item Let $\beta_{9}$ be the line that goes through $p_{9}$ with slope $ -0.1763 $.
Thus $\beta_9$ is simple.
\item Let $\beta_{10}$ be the line that goes through $p_{10}$ with slope $ -0.04 $.
Thus $\beta_{10}$ is simple.
\item Let $\beta_{11}$ be the line that goes through $p_{11}$ with slope $ -0.3942 $.
Thus $\beta_{11}$ is simple.
\item Let $\beta_{12}$ be the line that goes through $p_{12}$ with slope $ -0.052668 $.
Thus $\beta_{12}$ is simple.
\item Let $\beta_{13}$ be the line that goes through $p_{13}$ with slope $  0.052 $.
Thus $\beta_{13}$ is simple.
\item Let $\beta_{14}$ be the line that goes through $p_{14}$ with slope $ -1.1994 $.
Thus $\beta_{14}$ is simple.
\item Let $\beta_{15}$ be the line that goes through $p_{15}$ with slope $ -1.2591 $.
Thus $\beta_{15}$ is simple.
\item Let $\beta_{16}$ be the line that goes through $p_{16}$ with slope $-0.07$.
Thus $\beta_{16}$ is simple.
\item Let $\beta_{17}$ be the line that goes through $p_{17}$ with slope $-1.35028010$.
Thus $\beta_{17}$ is simple.
\item Let $\beta_{18}$ be the line $\theta^2(\beta_{17}) $.
Thus $\beta_{18}$ is simple and goes through $p_{18}$.  
\item Let $\beta_{19}$ be the line $\theta^2(\beta_{16}) $.
Thus $\beta_{19}$ is simple and goes through $p_{19}$.  
\item Let $\beta_{20}$ be the line $\theta(\beta_{15}) $.
Thus $\beta_{20}$ is simple and goes through $p_{20}$.  
\item Let $\beta_{21}$ be the line $\theta(\beta_{14}) $.
Thus $\beta_{21}$ is simple and goes through $p_{21}$.  
\item Let $\beta_{22}$ be the line $\theta(\beta_{13}) $.
Thus $\beta_{22}$ is simple and goes through $p_{22}$.
\item Let $\beta_{23}$ be the line $\theta(\beta_{17}) $.
Thus $\beta_{23}$ is simple and goes through $p_{23}$.
\item Let $\beta_{	24	}$ be the line $\theta(\beta_{		16}) $.
Thus $\beta_{	24	}$ is simple and goes through $p_{	24	} $.
\item Let $\beta_{	25	}$ be the line $\theta^2(\beta_{	14	}) $.
Thus $\beta_{	25	}$ is simple and goes through $p_{	25	} $.
\item Let $\beta_{	26	}$ be the line $\theta^2(\beta_{	15	}) $.
Thus $\beta_{	26	}$ is simple and goes through $p_{	26	} $.
\item Let $\beta_{	27	}$ be the line $\theta^2(\beta_{ 13	}) $.
Thus $\beta_{	27	}$ is simple and goes through $p_{	27	} $.
\item Let $\beta_{	28	}$ be the line $\theta^2(\beta_{	12	}) $.
Thus $\beta_{	28	}$ is simple and goes through $p_{	28	} $.
\item Let $\beta_{	29	}$ be the line $\theta^2(\beta_{	11}) $.
Thus $\beta_{	29	}$ is simple and goes through $p_{	29	} $.
\item Let $\beta_{	30	}$ be the line $\theta^2(\beta_{	10	}) $.
Thus $\beta_{	30	}$ is simple and goes through $p_{	30	} $.
\item Let $\beta_{	31	}$ be the line $\theta(\beta_{	12	}) $.
Thus $\beta_{	31	}$ is simple and goes through $p_{	31	} $.
\item Let $\beta_{	32	}$ be the line $\theta(\beta_{	11	}) $.
Thus $\beta_{	32	}$ is simple and goes through $p_{	32	} $.
\item Let $\beta_{	33	}$ be the line $\theta(\beta_{	10	}) $.
Thus $\beta_{	33	}$ is simple and goes through $p_{	33	} $.
\item Let $\beta_{	34	}$ be the line $\theta^2(\beta_{	9	}) $.
Thus $\beta_{	34	}$ is simple and goes through $p_{	34	} $.
\item Let $\beta_{	35	}$ be the line $\theta(\beta_{	9	}) $.
Thus $\beta_{	35	}$ is simple and goes through $p_{	35	} $.
\item Let $\beta_{	36	}$ be the line $\theta(\beta_{	8	}) $.
Thus $\beta_{	36	}$ is simple and goes through $p_{	36	} $.
\item Let $\beta_{	37	}$ be the line $\theta^2(\beta_{	8	}) $.
Thus $\beta_{	37	}$ is simple and goes through $p_{	37	} $.
\item Let $\beta_{	38	}$ be the line $\theta(\beta_{	7	}) $.
Thus $\beta_{	38	}$ is simple and goes through $p_{	38	} $.
\item Let $\beta_{	39	}$ be the line $\theta^2(\beta_{	7	}) $.
Thus $\beta_{	39	}$ is simple and goes through $p_{	39	} $.
\item Let $\beta_{	40	}$ be the line $\theta^2(\beta_{	6	}) $.
Thus $\beta_{	40	}$ is simple and goes through $p_{	40	} $.
\item Let $\beta_{	41	}$ be the line $\theta(\beta_{	6	}) $.
Thus $\beta_{	41	}$ is simple and goes through $p_{	41	} $.
\item Let $\beta_{	42	}$ be the line $\theta^2(\beta_{	5	}) $.
Thus $\beta_{	42	}$ is splitting and goes through $p_{	42	} $ and $p_{39} $, directed from $p_{42}$ towards $p_{39}$.
\item Let $\beta_{	43	}$ be the line $\theta(\beta_{	5	}) $.
Thus $\beta_{	43	}$ is splitting and goes through $p_{	43
} $ and $p_{	38} $, directed from $p_{43}$ towards $p_{38}$.
\item Let $\beta_{	44	}$ be the line $\theta(\beta_{	4	}) $.
Thus $\beta_{	44	}$ is splitting and goes through $p_{	44	} $ and $p_{38} $, directed from $p_{44}$ towards $p_{38}$.
\item Let $\beta_{	45	}$ be the line $\theta^2(\beta_{	4	}) $.
Thus $\beta_{	45	}$ is splitting and goes through $p_{	45	} $ and $p_{39} $, directed from $p_{45}$ towards $p_{39}$.
\item Let $\beta_{	48	}$ be the line $\theta^2(\beta_{	2	}) $.
Thus $\beta_{	48	}$ is splitting and goes through $p_{	48	} $ and $p_{45	} $, directed from $p_{48}$ towards $p_{45}$.
\item Let $\beta_{	49	}$ be the line $\theta(\beta_{	2	}) $.
Thus $\beta_{	49	}$ is splitting and goes through $p_{	49	} $ and $p_{44} $, directed from $p_{49}$ towards $p_{44}$.
\item Let $\beta_{	50	}$ be the line $\theta^2(\beta_{	1	}) $.
Thus $\beta_{	50	}$ is splitting and goes through $p_{	50	} $ and $p_{48} $, directed from $p_{50}$ towards $p_{48}$.
\item Let $\beta_{	51	}$ be the line $\theta(\beta_{	1	}) $.
Thus $\beta_{	51	}$ is splitting and goes through $p_{	51	}$ and $p_{	49} $, directed from $p_{51}$ towards $p_{49}$. 
\end{enumerate}

\newpage

\subsection{(How to construct) A drawing 
of $K_{33}$ with $14634$ crossings}

We describe how to obtain a drawing of $K_{33}$ with $14634$
crossings using the construction technique in Section~\ref{gencon}.
As explained at the end of Section~\ref{symdra},  
it suffices to give a base drawing $D_{m}$ for some
suitable $m<n$ (equivalently, the underlying point set $P_m$), 
the cluster models $S_{i}$, $i=1,\ldots,m$, and a pre--halving set of
lines $\{\beta_i\}_{i\in I}$ for those points in $P_{m}$ that get
transformed into a cluster.  In this case, we work with
a base set with $30$ points, that is, $m=30$.

These ingredients are given below. The result is a drawing of
$K_{33}$ with $14634$ crossings.

\subsubsection{The base point configuration}\label{base30}


We use as base configuration a $30$--point set
$P=\{p_1,p_2,\ldots,p_{30}\}$.  We give explicitly the coordinates of
$10$ of the $30$ points, and obtain the remaining $20$ points by
rotating each of these points $120$ and $240$ degrees around the origin.

Thus, we let:
$p_{	1	}= (	-500218.885	,	793018.474	)$,
$p_{	2	}= (	-451723.944	,	711948.989	)$,
$p_{	5	}= (	-200125.330	,	285855.310	)$,
$p_{	6	}= (	-158721.037	,	223132.241	)$,
$p_{	9	}= (	-103183.924	,	120586.624	)$,
$p_{	10	}= (	-88519.236	,	109026.774	)$,
$p_{11	}= (	-70502.886	,	100103.259	)$,
$p_{	12	}= (	-66221.918	,	53889.958	)$,
$p_{	13}= (	-65940.116	,	50836.878	)$, and 
$p_{	18	}= (	-13567.216	,	45695.226	)$.

We also let:
$p_{3}=\theta(p_{1})$,
$p_{4}=\theta(p_{2})$,
$p_{7}=\theta(p_{5})$,
$p_{8}=\theta(p_{6})$,
$p_{14}=\theta(p_{9})$,
$p_{16}=\theta(p_{10})$,
$p_{15}=\theta(p_{11})$,
$p_{19}=\theta(p_{12})$,
$p_{20}=\theta(p_{13})$,
$p_{17}=\theta(p_{18})$,
$p_{30}=\theta^2(p_{1})$,
$p_{29}=\theta^2(p_{2})$,
$p_{28}=\theta^2(p_{5})$,
$p_{27}=\theta^2(p_{6})$,
$p_{26}=\theta^2(p_{9})$,
$p_{25}=\theta^2(p_{10})$,
$p_{24}=\theta^2(p_{11})$,
$p_{23}=\theta^2(p_{12})$,
$p_{22}=\theta^2(p_{13})$,
$p_{21}=\theta^2(p_{18})$,

\subsubsection{The cluster models}

The cluster models for those points that do not get augmented or get
augmented into a cluster of size $2$ are trivial (any point
set in general position work). 

For $i = 1, 2, 3, 4, 5, 6, 7, 8, 10, 11, 12, 13, 15, 16, 17, 18, 19, 20, 21, 22, 23, 24, 25, 27, 28, 29$, an $30 $, we let $S_i$ have one point  (there is no need
to specify its coordinates, as we mentioned above).

For $i = 9, 14, 26$, we let $S_i$ have two points  (there is no need
to specify its coordinates, as we mentioned above).

Thus, the set $I$ of those subscripts $i$ such that $s_i:=|S_i|>1$ is 
$I=\{ 9, 14, 26 \}$. 

\subsubsection{A pre--halving set of lines}

We finally define a pre--halving set of lines 
$\{\beta_i\}_{i\in I}$.

\begin{enumerate}

\item Let $\beta_{9}$ be the line that goes through $p_{9}$ with slope
$  -2 $.
Thus $\beta_9$ is simple.
\item Let $\beta_{	14	}$ be the line $\theta(\beta_{	9	}) $.
Thus $\beta_{	14	}$ is simple and goes through $p_{	14	} $.
\item Let $\beta_{	26	}$ be the line $\theta^2(\beta_{	9	}) $.
Thus $\beta_{	26	}$ is simple and goes through $p_{	26	} $.

\end{enumerate}

\clearpage

\subsection{(How to construct) A drawing 
of $K_{60}$ with $179541$ crossings}

We describe how to obtain a drawing of $K_{60}$ with $179541$
crossings using the construction technique in Section~\ref{gencon}.
As explained at the end of Section~\ref{symdra},  
it suffices to give a base drawing $D_{m}$ for some
suitable $m<n$ (equivalently, the underlying point set $P_m$), 
the cluster models $S_{i}$, $i=1,\ldots,m$, and a pre--halving set of
lines $\{\beta_i\}_{i\in I}$ for those points in $P_{m}$ that get
transformed into a cluster.  In this case, we work with
a base set with $30$ points, that is, $m=30$.

These ingredients are given below. The result is a drawing of
$K_{60}$ with $179541$ crossings.

\subsubsection{The base point configuration}

We use as base configuration the $30$--point set
$P=\{p_1,p_2,\ldots,p_{30}\}$ from Section~\ref{base30}.  

\subsubsection{The cluster models}

The cluster models for those points that do not get augmented or get
augmented into a cluster of size $2$ or $3$ are trivial (any point
sets in general position work). Since all clusters in this case are of
size $1$, $2$, or $3$, the description is greatly simplified in this
case:

For $i = 1, 2, \ldots, 30$, we let $S_i$ have two points  (there is no need
to specify its coordinates, as we mentioned above).

Thus, the set $I$ of those subscripts $i$ such that $s_i:=|S_i|>1$ is 
$I=\{1, 2, \ldots, 30\}$. 

\subsubsection{A pre--halving set of lines}

We finally define a pre--halving set of lines 
$\{\beta_i\}_{i\in I}$.

\begin{enumerate}

\item Let $\beta_{ 1 }$ be the line that goes through $p_{ 1 }$ and $p_{
2 }$,
 directed from $p_{ 1 }$ towards $p_{ 2 }$.
\item Let $\beta_{ 2 }$ be the line that goes through $p_{ 2 }$ and $p_{
5 }$,
 directed from $p_{ 2 }$ towards $p_{ 5 }$.
\item Let $\beta_{    3    }$ be the line $\theta(\beta_{    1    }) $.
Thus $\beta_{    3    }$ is splitting and goes through $p_{    3    }$
and $p_{     4 } $, directed from $p_{ 3 }$ towards $p_{ 4 }$.
\item Let $\beta_{    4    }$ be the line $\theta(\beta_{    2    }) $.
Thus $\beta_{    4    }$ is splitting and goes through $p_{    4    }$
and $p_{    7  } $, directed from $p_{ 4 }$ towards $p_{ 7 }$.
\item Let $\beta_{ 5 }$ be the line that goes through $p_{ 5 }$ and $p_{
9 }$,
 directed from $p_{ 5 }$ towards $p_{ 9 }$.
\item Let $\beta_{ 6 }$ be the line that goes through $p_{ 6 }$ and $p_{
10 }$,
 directed from $p_{ 6 }$ towards $p_{ 10 }$.

\item Let $\beta_{    7    }$ be the line $\theta(\beta_{    5    }) $.
Thus $\beta_{    7    }$ is splitting and goes through $p_{    7    }$
and $p_{     14 } $, directed from $p_{ 7 }$ towards $p_{ 14 }$.

\item Let $\beta_{    8    }$ be the line $\theta(\beta_{    6    }) $.
Thus $\beta_{    8    }$ is splitting and goes through $p_{    8    }$
and $p_{    16  } $, directed from $p_{ 8 }$ towards $p_{ 16 }$.

\item Let $\beta_{ 9 }$ be the line that goes through $p_{ 9 }$ and $p_{
10 }$,
 directed from $p_{ 9 }$ towards $p_{ 10 }$.

\item Let $\beta_{ 10 }$ be the line that goes through $p_{ 10 }$ and
$p_{ 11 }$,
 directed from $p_{ 10 }$ towards $p_{ 11 }$.

\item Let $\beta_{ 11 }$ be the line that goes through $p_{ 11 }$ and
$p_{ 12 }$,
 directed from $p_{ 11 }$ towards $p_{ 12 }$.

\item Let $\beta_{ 12 }$ be the line that goes through $p_{ 12 }$ and
$p_{ 13 }$,
 directed from $p_{ 12 }$ towards $p_{ 13 }$.

\item Let $\beta_{ 13 }$ be the line that goes through $p_{ 10 }$ and
$p_{ 13 }$,
 directed from $p_{ 13 }$ towards $p_{ 10 }$.

\item Let $\beta_{    14    }$ be the line $\theta(\beta_{    9    }) $.
Thus $\beta_{    14    }$ is splitting and goes through $p_{    14    }$
and $p_{    16  } $, directed from $p_{ 14 }$ towards $p_{ 16 }$.

\item Let $\beta_{    15    }$ be the line $\theta(\beta_{    11    }) $.
Thus $\beta_{    15    }$ is splitting and goes through $p_{    15    }$
and $p_{     19 } $, directed from $p_{ 15 }$ towards $p_{ 19 }$.

\item Let $\beta_{    16    }$ be the line $\theta(\beta_{    10    }) $.
Thus $\beta_{    16    }$ is splitting and goes through $p_{    15    }$
and $p_{    16  } $, directed from $p_{ 16 }$ towards $p_{ 15 }$.

\item Let $\beta_{    17    }$ be the line $\theta(\beta_{    18    }) $.
Thus $\beta_{    17    }$ is splitting and goes through $p_{    17    }$
and $p_{    19  } $, directed from $p_{ 17 }$ towards $p_{ 19 }$.

\item Let $\beta_{ 18 }$ be the line that goes through $p_{ 12 }$ and
$p_{ 18 }$,
 directed from $p_{ 18 }$ towards $p_{ 12 }$.

\item Let $\beta_{    19    }$ be the line $\theta(\beta_{    12    }) $.
Thus $\beta_{    19    }$ is splitting and goes through $p_{    19    }$
and $p_{    20  } $, directed from $p_{ 19 }$ towards $p_{ 20 }$.

\item Let $\beta_{    20    }$ be the line $\theta(\beta_{    13    }) $.
Thus $\beta_{    20    }$ is splitting and goes through $p_{    16    }$
and $p_{    20  } $, directed from $p_{ 20 }$ towards $p_{ 16 }$.

\item Let $\beta_{    21    }$ be the line $\theta^2(\beta_{    18    }) $.
Thus $\beta_{    21    }$ is splitting and goes through $p_{    21    }
$ and $p_{ 23 } $, directed from $p_{ 21 }$ towards $p_{ 23 }$.

\item Let $\beta_{    22    }$ be the line $\theta^2(\beta_{    13    }) $.
Thus $\beta_{    22    }$ is splitting and goes through $p_{    22    }
$ and $p_{ 25 } $, directed from $p_{ 22 }$ towards $p_{ 25 }$.

\item Let $\beta_{    23    }$ be the line $\theta^2(\beta_{    12    }) $.
Thus $\beta_{    23    }$ is splitting and goes through $p_{    22    }
$ and $p_{ 23 } $, directed from $p_{ 23 }$ towards $p_{ 22 }$.

\item Let $\beta_{    24    }$ be the line $\theta^2(\beta_{    11    }) $.
Thus $\beta_{    24    }$ is splitting and goes through $p_{    23    }
$ and $p_{ 24 } $, directed from $p_{ 24 }$ towards $p_{ 23 }$.

\item Let $\beta_{    25    }$ be the line $\theta^2(\beta_{    10    }) $.
Thus $\beta_{    25    }$ is splitting and goes through $p_{    24    }
$ and $p_{ 25 } $, directed from $p_{ 25 }$ towards $p_{ 24 }$.

\item Let $\beta_{    26    }$ be the line $\theta^2(\beta_{    9    }) $.
Thus $\beta_{    26    }$ is splitting and goes through $p_{    25    }
$ and $p_{ 26 } $, directed from $p_{ 26 }$ towards $p_{ 25 }$.

\item Let $\beta_{    27    }$ be the line $\theta^2(\beta_{    6    }) $.
Thus $\beta_{    27    }$ is splitting and goes through $p_{    25    }
$ and $p_{ 27 } $, directed from $p_{ 27 }$ towards $p_{ 25 }$.

\item Let $\beta_{    28    }$ be the line $\theta^2(\beta_{    5    }) $.
Thus $\beta_{    28    }$ is splitting and goes through $p_{    26    }
$ and $p_{ 28 } $, directed from $p_{ 28 }$ towards $p_{ 26 }$.

\item Let $\beta_{    29    }$ be the line $\theta^2(\beta_{    2    }) $.
Thus $\beta_{    29    }$ is splitting and goes through $p_{    28    }
$ and $p_{ 29 } $, directed from $p_{ 29 }$ towards $p_{ 28 }$.

\item Let $\beta_{    30    }$ be the line $\theta^2(\beta_{    1    }) $.
Thus $\beta_{    30    }$ is splitting and goes through $p_{    29    }
$ and $p_{ 30 } $, directed from $p_{ 30 }$ towards $p_{ 29 }$.

\end{enumerate}


\clearpage

\subsection{(How to construct) A drawing 
of $K_{63}$ with $219681$ crossings}

We describe how to obtain a drawing of $K_{63}$ with $219681$
crossings using the construction technique in Section~\ref{gencon}.
As explained at the end of Section~\ref{symdra},  
it suffices to give a base drawing $D_{m}$ for some
suitable $m<n$ (equivalently, the underlying point set $P_m$), 
the cluster models $S_{i}$, $i=1,\ldots,m$, and a pre--halving set of
lines $\{\beta_i\}_{i\in I}$ for those points in $P_{m}$ that get
transformed into a cluster.  In this case, we work with
a base set with $51$ points, that is, $m=51$.

These ingredients are given below. The result is a drawing of
$K_{63}$ with $219681$ crossings.

\subsubsection{The base point configuration}

We use as base configuration the $51$--point set
$P=\{p_1,p_2,\ldots,p_{51}\}$ from Section~\ref{t315}.  

\subsubsection{The cluster models}

The cluster models for those points that do not get augmented or get
augmented into a cluster of size $2$ or $3$ are trivial (any point
sets in general position work). Since all clusters in this case are of
size $1$, $2$, or $3$, the description is greatly simplified in this
case:

For $i= 1 , 2 , 5 , 6 , 8 , 9 , 10 , 11 , 12 , 14 , 15 , 16 , 17 , 18 , 19 , 20 , 21 , 23 , 24 , 25 , 26 , 28 , 29 , 30 , 31 , 32, 33 , 34 $,
$35 ,36 , 37 , 40 , 41 , 42 , 43 , 48 , 49 , 50$ and $51$ we let $S_i$ have one point (there is no need
to specify its coordinates, as we mentioned above).

For $i = 3, 4, 7, 13, 22, 27, 38, 39, 44, 45, 46$ and $47$, we let $S_i$ have two points  (there is no need
to specify its coordinates, as we mentioned above).

Thus, the set $I$ of those subscripts $i$ such that $s_i:=|S_i|>1$ is $I=\{ 3, 4, 7, 13$, 
$22, 27, 38, 39, 44, 45, 46, 47 \}$. 

\subsubsection{A pre--halving set of lines}

We finally define a pre--halving set of lines 
$\{\beta_i\}_{i\in I}$.

\begin{enumerate}

\item Let $\beta_{3}$ be the line that goes through $p_{3}$ with slope $0.001$.
Thus $\beta_3$ is simple.
\item Let $\beta_{4}$ be the line that goes through $p_{4}$ with slope $-0.02$.
Thus $\beta_4$ is simple.
\item Let $\beta_{7}$ be the line that goes through $p_{7}$ with slope $-0.1$.
Thus $\beta_7$ is simple.
\item Let $\beta_{13}$ be the line that goes through $p_{13}$ with slope $-1$.
Thus $\beta_{13}$ is simple.
\item Let $\beta_{22}$ be the line $\theta(\beta_{13}) $.
Thus $\beta_{22}$ is simple and goes through $p_{22}$.  
\item Let $\beta_{27}$ be the line $\theta^2(\beta_{13}) $.
Thus $\beta_{27}$ is simple and goes through $p_{27}$.  
\item Let $\beta_{38}$ be the line $\theta(\beta_{7}) $.
Thus $\beta_{38}$ is simple and goes through $p_{38}$.  
\item Let $\beta_{39}$ be the line $\theta^2(\beta_{7}) $.
Thus $\beta_{39}$ is simple and goes through $p_{39}$.  
\item Let $\beta_{44}$ be the line $\theta(\beta_{4}) $.
Thus $\beta_{44}$ is simple and goes through $p_{44}$.  
\item Let $\beta_{45}$ be the line $\theta^2(\beta_{4}) $.
Thus $\beta_{45}$ is simple and goes through $p_{45}$.  
\item Let $\beta_{46}$ be the line $\theta^2(\beta_{3}) $.
Thus $\beta_{46}$ is simple and goes through $p_{46}$.  
\item Let $\beta_{47}$ be the line $\theta(\beta_{3}) $.
Thus $\beta_{47}$ is simple and goes through $p_{47}$.

\end{enumerate}

\clearpage

\subsection{(How to construct) A drawing 
of $K_{66}$ with $266181$ crossings}

We describe how to obtain a drawing of $K_{63}$ with $266181$
crossings using the construction technique in Section~\ref{gencon}.
As explained at the end of Section~\ref{symdra},  
it suffices to give a base drawing $D_{m}$ for some
suitable $m<n$ (equivalently, the underlying point set $P_m$), 
the cluster models $S_{i}$, $i=1,\ldots,m$, and a pre--halving set of
lines $\{\beta_i\}_{i\in I}$ for those points in $P_{m}$ that get
transformed into a cluster.  In this case, we work with
a base set with $51$ points, that is, $m=51$.

These ingredients are given below. The result is a drawing of
$K_{63}$ with $266181$ crossings.

\subsubsection{The base point configuration}

We use as base configuration the $51$--point set
$P=\{p_1,p_2,\ldots,p_{51}\}$ from Section~\ref{t315}.  

\subsubsection{The cluster models}

The cluster models for those points that do not get augmented or get
augmented into a cluster of size $2$ or $3$ are trivial (any point
sets in general position work). Since all clusters in this case are of
size $1$, $2$, or $3$, the description is greatly simplified in this
case:

For $i=1 , 2 , 5 , 7 , 9 , 10 , 11 , 12 , 13 , 15 , 16 , 17 , 18 , 19 , 20 , 22 , 23 , 24 , 26 , 27 , 28 , 29 , 30 , 31 , 32 , 33 , 34$, 
$35 , 38 , 39 , 42 , 43 , 48 , 49 , 50$, and $51$ we let $S_i$ have one point (there is no need
to specify its coordinates, as we mentioned above).

For $i =3, 4, 6, 8, 14, 21, 25, 36, 37, 40, 41, 44, 45, 46$, and $47 $, we let $S_i$ have two points  (there is no need
to specify its coordinates, as we mentioned above).

Thus, the set $I$ of those subscripts $i$ such that $s_i:=|S_i|>1$ is 
$I=\{ 3, 4, 6, 8, 14, 21, 25, 36, 37, 40, 41, 44, 45, 46, 47\}$. 

\subsubsection{A pre--halving set of lines}

We finally define a pre--halving set of lines 
$\{\beta_i\}_{i\in I}$.

\begin{enumerate}

\item Let $\beta_{3}$ be the line that goes through $p_{3}$ with slope $0.001$.
Thus $\beta_3$ is simple.
\item Let $\beta_{4}$ be the line that goes through $p_{4}$ with slope $-0.011$.
Thus $\beta_4$ is simple.
\item Let $\beta_{6}$ be the line that goes through $p_{6}$ with slope $-0.0305$.
Thus $\beta_6$ is simple.
\item Let $\beta_{8}$ be the line that goes through $p_{8}$ with slope $-0.17$.
Thus $\beta_8$ is simple.
\item Let $\beta_{14}$ be the line that goes through $p_{14}$ with slope $-1.21$.
Thus $\beta_{14}$ is simple.
\item Let $\beta_{21}$ be the line $\theta(\beta_{14}) $.
Thus $\beta_{21}$ is simple and goes through $p_{21}$.  
\item Let $\beta_{25}$ be the line $\theta^2(\beta_{14}) $.
Thus $\beta_{25}$ is simple and goes through $p_{25}$.  
\item Let $\beta_{36}$ be the line $\theta(\beta_{8}) $.
Thus $\beta_{36}$ is simple and goes through $p_{37}$.  
\item Let $\beta_{37}$ be the line $\theta^2(\beta_{8}) $.
Thus $\beta_{37}$ is simple and goes through $p_{37}$.  
\item Let $\beta_{40}$ be the line $\theta^2(\beta_{6}) $.
Thus $\beta_{40}$ is simple and goes through $p_{40}$.  
\item Let $\beta_{41}$ be the line $\theta(\beta_{6}) $.
Thus $\beta_{41}$ is simple and goes through $p_{41}$.  
\item Let $\beta_{44}$ be the line $\theta(\beta_{4}) $.
Thus $\beta_{44}$ is simple and goes through $p_{44}$.  
\item Let $\beta_{45}$ be the line $\theta^2(\beta_{4}) $.
Thus $\beta_{45}$ is simple and goes through $p_{45}$.  
\item Let $\beta_{46}$ be the line $\theta^2(\beta_{3}) $.
Thus $\beta_{46}$ is simple and goes through $p_{46}$.  
\item Let $\beta_{47}$ be the line $\theta(\beta_{3}) $.
Thus $\beta_{47}$ is simple and goes through $p_{47}$.  

\end{enumerate}

\clearpage

\subsection{(How to construct) A drawing 
of $K_{69}$ with $319731$ crossings}

We describe how to obtain a drawing of $K_{69}$ with $319731$
crossings using the construction technique in Section~\ref{gencon}.
As explained at the end of Section~\ref{symdra},  
it suffices to give a base drawing $D_{m}$ for some
suitable $m<n$ (equivalently, the underlying point set $P_m$), 
the cluster models $S_{i}$, $i=1,\ldots,m$, and a pre--halving set of
lines $\{\beta_i\}_{i\in I}$ for those points in $P_{m}$ that get
transformed into a cluster.  In this case, we work with
a base set with $51$ points, that is, $m=51$.

These ingredients are given below. The result is a drawing of
$K_{69}$ with $319731$ crossings.

\subsubsection{The base point configuration}

We use as base configuration the $51$--point set
$P=\{p_1,p_2,\ldots,p_{51}\}$ from Section~\ref{t315}.  

\subsubsection{The cluster models}

The cluster models for those points that do not get augmented or get
augmented into a cluster of size $2$ or $3$ are trivial (any point
sets in general position work). Since all clusters in this case are of
size $1$, $2$, or $3$, the description is greatly simplified in this
case:

For $i=  1 , 5 , 6 , 7 , 10 , 11 , 12 , 13 , 15 , 16 , 17 , 18 , 19 , 20 , 22 , 23 , 24 , 26 , 27 , 28 , 29 , 30 , 31 , 32 , 33 , 38 , 39$, 
$40 , 41 , 42 , 43 , 50 , 51$ we let $S_i$ have one point (there is no need
to specify its coordinates, as we mentioned above).

For $i = 2, 3, 4, 8, 9, 14, 21, 25, 34, 35, 36, 37, 44, 45, 46, 47, 48$, and $49$, we let $S_i$ have two points  (there is no need
to specify its coordinates, as we mentioned above).

Thus, the set $I$ of those subscripts $i$ such that $s_i:=|S_i|>1$ is 
$I=\{ 2, 3, 4, 8, 9, 14, 21, 25, 34, 35, 36, 37, 44, 45, 46, 47, 48, 49\}$. 

\subsubsection{A pre--halving set of lines}

We finally define a pre--halving set of lines 
$\{\beta_i\}_{i\in I}$.

\begin{enumerate}

\item Let $\beta_{2}$ be the line that goes through $p_{2}$ with slope $0.005$.
Thus $\beta_2$ is simple.
\item Let $\beta_{3}$ be the line that goes through $p_{3}$ with slope $-0.0006$.
Thus $\beta_3$ is simple.
\item Let $\beta_{4}$ be the line that goes through $p_{4}$ with slope $-0.017$.
Thus $\beta_4$ is simple.
\item Let $\beta_{8}$ be the line that goes through $p_{8}$ with slope $ -0.17$.
Thus $\beta_8$ is simple.
\item Let $\beta_{9}$ be the line that goes through $p_{9}$ with slope $ -0.25 $.
Thus $\beta_9$ is simple.
\item Let $\beta_{14}$ be the line that goes through $p_{14}$ with slope $ -1.21 $.
Thus $\beta_{14}$ is simple.
\item Let $\beta_{21}$ be the line $\theta(\beta_{14}) $.
Thus $\beta_{21}$ is simple and goes through $p_{21}$.  
\item Let $\beta_{	25	}$ be the line $\theta^2(\beta_{	14	}) $.
Thus $\beta_{	25	}$ is simple and goes through $p_{	25	} $.
\item Let $\beta_{	34	}$ be the line $\theta^2(\beta_{	9	}) $.
Thus $\beta_{	34	}$ is simple and goes through $p_{	34	} $.
\item Let $\beta_{	35	}$ be the line $\theta(\beta_{	9	}) $.
Thus $\beta_{	35	}$ is simple and goes through $p_{	35	} $.
\item Let $\beta_{	36	}$ be the line $\theta(\beta_{	8	}) $.
Thus $\beta_{	36	}$ is simple and goes through $p_{	36	} $.
\item Let $\beta_{	37	}$ be the line $\theta^2(\beta_{	8	}) $.
Thus $\beta_{	37	}$ is simple and goes through $p_{	37	} $.
\item Let $\beta_{44}$ be the line $\theta(\beta_{4}) $.
Thus $\beta_{44}$ is simple and goes through $p_{44}$.  
\item Let $\beta_{45}$ be the line $\theta^2(\beta_{4}) $.
Thus $\beta_{45}$ is simple and goes through $p_{45}$.  
\item Let $\beta_{46}$ be the line $\theta^2(\beta_{3}) $.
Thus $\beta_{46}$ is simple and goes through $p_{46}$.  
\item Let $\beta_{47}$ be the line $\theta(\beta_{3}) $.
Thus $\beta_{47}$ is simple and goes through $p_{47}$.  
\item Let $\beta_{	48	}$ be the line $\theta^2(\beta_{	2	}) $.
Thus $\beta_{48}$ is simple and goes through $p_{48}$.  
\item Let $\beta_{	49	}$ be the line $\theta(\beta_{	2	}) $.
Thus $\beta_{49}$ is simple and goes through $p_{49}$.  

\end{enumerate}

\clearpage

\subsection{(How to construct) A drawing 
of $K_{72}$ with $380964$ crossings}

We describe how to obtain a drawing of $K_{72}$ with $380964$
crossings using the construction technique in Section~\ref{gencon}.
As explained at the end of Section~\ref{symdra},  
it suffices to give a base drawing $D_{m}$ for some
suitable $m<n$ (equivalently, the underlying point set $P_m$), 
the cluster models $S_{i}$, $i=1,\ldots,m$, and a pre--halving set of
lines $\{\beta_i\}_{i\in I}$ for those points in $P_{m}$ that get
transformed into a cluster.  In this case, we work with
a base set with $51$ points, that is, $m=51$.

These ingredients are given below. The result is a drawing of
$K_{72}$ with $380964$ crossings.

\subsubsection{The base point configuration}

We use as base configuration the $51$--point set
$P=\{p_1,p_2,\ldots,p_{51}\}$ from Section~\ref{t315}.  

\subsubsection{The cluster models}

The cluster models for those points that do not get augmented or get
augmented into a cluster of size $2$ or $3$ are trivial (any point
sets in general position work). Since all clusters in this case are of
size $1$, $2$, or $3$, the description is greatly simplified in this
case:

For $i = 1 , 6 , 7 , 8 , 10 , 11 , 12 , 15 , 16 , 17 , 18 , 19 , 20 , 23 , 24 , 26 , 28 , 29 , 30 , 31 , 32 , 33 , 36 , 37 , 38 , 39 , 40$, $41 , 50$, and $51$ we let $S_i$ have one point (there is no need
to specify its coordinates, as we mentioned above).

For $i = 2 ,3, 4, 5, 9, 13, 14, 21, 22, 25, 27, 34, 35, 42, 43, 44, 45, 46, 47, 48$, and $49$ we let $S_i$ have two points  (there is no need
to specify its coordinates, as we mentioned above).

Thus, the set $I$ of those subscripts $i$ such that $s_i:=|S_i|>1$ is 
$I=\{ 2 ,3, 4, 5, 9, 13, 14, 21, 22, 25, 27, 34, 35, 42, 43, 44, 45, 46, 47, 48, 49\}$. 

\subsubsection{A pre--halving set of lines}

We finally define a pre--halving set of lines 
$\{\beta_i\}_{i\in I}$.

\begin{enumerate}

\item Let $\beta_{2}$ be the line that goes through $p_{2}$ with slope $0.005$.
Thus $\beta_2$ is simple.
\item Let $\beta_{3}$ be the line that goes through $p_{3}$ with slope $0.001$.
Thus $\beta_3$ is simple.
\item Let $\beta_{4}$ be the line that goes through $p_{4}$ with slope $0.04$.
Thus $\beta_4$ is simple.
\item Let $\beta_{5}$ be the line that goes through $p_{5}$ with slope $-0.027$.
Thus $\beta_5$ is simple.
\item Let $\beta_{9}$ be the line that goes through $p_{9}$ with slope $ -0.1763 $.
Thus $\beta_9$ is simple.
\item Let $\beta_{13}$ be the line that goes through $p_{13}$ with slope $  0.052 $.
Thus $\beta_{13}$ is simple.
\item Let $\beta_{14}$ be the line that goes through $p_{14}$ with slope $ -1.1994 $.
Thus $\beta_{14}$ is simple.
\item Let $\beta_{21}$ be the line $\theta(\beta_{14}) $.
Thus $\beta_{21}$ is simple and goes through $p_{21}$.  
\item Let $\beta_{22}$ be the line $\theta(\beta_{13}) $.
Thus $\beta_{22}$ is simple and goes through $p_{22}$.
\item Let $\beta_{	25	}$ be the line $\theta^2(\beta_{	14	}) $.
Thus $\beta_{	25	}$ is simple and goes through $p_{	25	} $.
\item Let $\beta_{	27	}$ be the line $\theta^2(\beta_{ 13	}) $.
Thus $\beta_{	27	}$ is simple and goes through $p_{	27	} $.
\item Let $\beta_{	34	}$ be the line $\theta^2(\beta_{	9	}) $.
Thus $\beta_{	34	}$ is simple and goes through $p_{	34	} $.
\item Let $\beta_{	35	}$ be the line $\theta(\beta_{	9	}) $.
Thus $\beta_{	35	}$ is simple and goes through $p_{	35	} $.
\item Let $\beta_{	42	}$ be the line $\theta^2(\beta_{	5	}) $.
Thus $\beta_{	42	}$ is simple and goes through $p_{	42	} $.
\item Let $\beta_{	43	}$ be the line $\theta(\beta_{	5	}) $.
Thus $\beta_{	43	}$ is simple and goes through $p_{	43	} $.
\item Let $\beta_{44}$ be the line $\theta(\beta_{4}) $.
Thus $\beta_{44}$ is simple and goes through $p_{44}$.  
\item Let $\beta_{45}$ be the line $\theta^2(\beta_{4}) $.
Thus $\beta_{45}$ is simple and goes through $p_{45}$.  
\item Let $\beta_{46}$ be the line $\theta^2(\beta_{3}) $.
Thus $\beta_{46}$ is simple and goes through $p_{46}$.  
\item Let $\beta_{47}$ be the line $\theta(\beta_{3}) $.
Thus $\beta_{47}$ is simple and goes through $p_{47}$.  
\item Let $\beta_{	48	}$ be the line $\theta^2(\beta_{	2	}) $.
Thus $\beta_{48}$ is simple and goes through $p_{48}$.  
\item Let $\beta_{	49	}$ be the line $\theta(\beta_{	2	}) $.
Thus $\beta_{49}$ is simple and goes through $p_{49}$.  

\end{enumerate}

\clearpage

\subsection{(How to construct) A drawing 
of $K_{75}$ with $450540$ crossings}

We describe how to obtain a drawing of $K_{75}$ with $450540$
crossings using the construction technique in Section~\ref{gencon}.
As explained at the end of Section~\ref{symdra},  
it suffices to give a base drawing $D_{m}$ for some
suitable $m<n$ (equivalently, the underlying point set $P_m$), 
the cluster models $S_{i}$, $i=1,\ldots,m$, and a pre--halving set of
lines $\{\beta_i\}_{i\in I}$ for those points in $P_{m}$ that get
transformed into a cluster.  In this case, we work with
a base set with $51$ points, that is, $m=51$.

These ingredients are given below. The result is a drawing of
$K_{75}$ with $450540$ crossings.

\subsubsection{The base point configuration}

We use as base configuration the $51$--point set
$P=\{p_1,p_2,\ldots,p_{51}\}$ from Section~\ref{t315}.  

\subsubsection{The cluster models}

The cluster models for those points that do not get augmented or get
augmented into a cluster of size $2$ or $3$ are trivial (any point
sets in general position work). Since all clusters in this case are of
size $1$, $2$, or $3$, the description is greatly simplified in this
case:

For $i=3 , 6 , 7 , 10 , 11 , 12 , 15 , 16 , 17 , 18 , 19 , 20 , 23 , 24 , 26 , 28 , 29 , 30 , 31 , 32 , 33 , 38 , 39 , 40 , 41 , 46$, and $47$ we let $S_i$ have one point (there is no need
to specify its coordinates, as we mentioned above).

For $i = 1, 2, 4, 5, 8, 9, 13, 14, 21, 22, 25, 27, 34, 35, 36, 37, 42, 43, 44, 45, 48, 49, 50$, and $51$ we let $S_i$ have two points  (there is no need
to specify its coordinates, as we mentioned above).

Thus, the set $I$ of those subscripts $i$ such that $s_i:=|S_i|>1$ is 
$I=\{ 1, 2, 4, 5, 8, 9, 13, 14, 21, 22, 25, 27, 34, 35, 36, 37, 42, 43, 44, 45, 48, 49, 50, 51 \}$. 

\subsubsection{A pre--halving set of lines}

We finally define a pre--halving set of lines 
$\{\beta_i\}_{i\in I}$.

\begin{enumerate}

\item Let $\beta_{1}$ be the line that goes through $p_{1}$ with slope $0.0063$.
Thus $\beta_1$ is simple.
\item Let $\beta_{2}$ be the line that goes through $p_{2}$ with slope $0.005$.
Thus $\beta_2$ is simple.
\item Let $\beta_{4}$ be the line that goes through $p_{4}$ with slope $-0.011$.
Thus $\beta_4$ is simple.
\item Let $\beta_{5}$ be the line that goes through $p_{5}$ with slope $-0.027$.
Thus $\beta_5$ is simple.
\item Let $\beta_{8}$ be the line that goes through $p_{8}$ with slope $ -0.17 $.
Thus $\beta_8$ is simple.
\item Let $\beta_{9}$ be the line that goes through $p_{9}$ with slope $ -0.25 $.
Thus $\beta_9$ is simple.
\item Let $\beta_{13}$ be the line that goes through $p_{13}$ with slope $ -1.08 $.
Thus $\beta_{13}$ is simple.
\item Let $\beta_{14}$ be the line that goes through $p_{14}$ with slope $ -1.2 $.
Thus $\beta_{14}$ is simple.
\item Let $\beta_{21}$ be the line $\theta(\beta_{14}) $.
Thus $\beta_{21}$ is simple and goes through $p_{21}$.  
\item Let $\beta_{22}$ be the line $\theta(\beta_{13}) $.
Thus $\beta_{22}$ is simple and goes through $p_{22}$.
\item Let $\beta_{	25	}$ be the line $\theta^2(\beta_{	14	}) $.
Thus $\beta_{	25	}$ is simple and goes through $p_{	25	} $.
\item Let $\beta_{	27	}$ be the line $\theta^2(\beta_{ 13	}) $.
Thus $\beta_{	27	}$ is simple and goes through $p_{	27	} $.
\item Let $\beta_{	34	}$ be the line $\theta^2(\beta_{	9	}) $.
Thus $\beta_{	34	}$ is simple and goes through $p_{	34	} $.
\item Let $\beta_{	35	}$ be the line $\theta(\beta_{	9	}) $.
Thus $\beta_{	35	}$ is simple and goes through $p_{	35	} $.
\item Let $\beta_{	36	}$ be the line $\theta(\beta_{	8	}) $.
Thus $\beta_{	36	}$ is simple and goes through $p_{	36	} $.
\item Let $\beta_{	37	}$ be the line $\theta^2(\beta_{	8	}) $.
Thus $\beta_{	37	}$ is simple and goes through $p_{	37	} $.
\item Let $\beta_{	42	}$ be the line $\theta^2(\beta_{	5	}) $.
Thus $\beta_{	42	}$ is simple and goes through $p_{	42	} $.
\item Let $\beta_{	43	}$ be the line $\theta(\beta_{	5	}) $.
Thus $\beta_{	43	}$ is simple and goes through $p_{	43	} $.
\item Let $\beta_{44}$ be the line $\theta(\beta_{4}) $.
Thus $\beta_{44}$ is simple and goes through $p_{44}$.  
\item Let $\beta_{45}$ be the line $\theta^2(\beta_{4}) $.
Thus $\beta_{45}$ is simple and goes through $p_{45}$.  
\item Let $\beta_{	48	}$ be the line $\theta^2(\beta_{	2	}) $.
Thus $\beta_{48}$ is simple and goes through $p_{48}$.  
\item Let $\beta_{	49	}$ be the line $\theta(\beta_{	2	}) $.
Thus $\beta_{49}$ is simple and goes through $p_{49}$.  
\item Let $\beta_{	50	}$ be the line $\theta^2(\beta_{	1	}) $.
Thus $\beta_{50}$ is simple and goes through $p_{50}$.  
\item Let $\beta_{	51	}$ be the line $\theta(\beta_{	1	}) $.
Thus $\beta_{51}$ is simple and goes through $p_{51}$.  

\end{enumerate}

\clearpage

\subsection{(How to construct) A drawing 
of $K_{78}$ with $529332$ crossings}

We describe how to obtain a drawing of $K_{78}$ with $529332$
crossings using the construction technique in Section~\ref{gencon}.
As explained at the end of Section~\ref{symdra},  
it suffices to give a base drawing $D_{m}$ for some
suitable $m<n$ (equivalently, the underlying point set $P_m$), 
the cluster models $S_{i}$, $i=1,\ldots,m$, and a pre--halving set of
lines $\{\beta_i\}_{i\in I}$ for those points in $P_{m}$ that get
transformed into a cluster.  In this case, we work with
a base set with $51$ points, that is, $m=51$.

These ingredients are given below. The result is a drawing of
$K_{78}$ with $529332$ crossings.

\subsubsection{The base point configuration}

We use as base configuration the $51$--point set
$P=\{p_1,p_2,\ldots,p_{51}\}$ from Section~\ref{t315}.  

\subsubsection{The cluster models}

The cluster models for those points that do not get augmented or get
augmented into a cluster of size $2$ or $3$ are trivial (any point
sets in general position work). Since all clusters in this case are of
size $1$, $2$, or $3$, the description is greatly simplified in this
case:

For $i= 3 , 6 , 7 , 10 , 11 , 15 , 16 , 17 , 18 , 19 , 20 , 23 , 24 , 26 , 29 , 30 , 32 , 33 , 38 , 39 , 40 , 41 , 46 , 47 $ we let $S_i$ have one point (there is no need
to specify its coordinates, as we mentioned above).

For $i = 1, 2, 4, 5, 8, 9, 12, 13, 14, 21, 22, 25, 27, 28, 31, 34, 35, 36, 37, 42, 43, 44, 45, 48, 49, 50$, and $51$ we let $S_i$ have two points  (there is no need
to specify its coordinates, as we mentioned above).

Thus, the set $I$ of those subscripts $i$ such that $s_i:=|S_i|>1$ is 
$I=\{  1, 2, 4, 5, 8, 9, 12, 13, 14, 21, 22, 25, 27, 28, 31, 34, 35, 36, 37, 42, 43, 44, 45, 48, 49, 50, 51 \}$. 

\subsubsection{A pre--halving set of lines}

We finally define a pre--halving set of lines 
$\{\beta_i\}_{i\in I}$.

\begin{enumerate}

\item Let $\beta_1$ be the line that goes through $p_1$ and $p_2$,
  directed from $p_1$ towards $p_2$. 
Thus $\beta_1$ is splitting.
\item Let $\beta_{2}$ be the line that goes through $p_{2}$ with slope $ 0.01 $.
Thus $\beta_2$ is simple.
\item Let $\beta_{4}$ be the line that goes through $p_{4}$ with slope $ 0.02 $.
Thus $\beta_4$ is simple.
\item Let $\beta_{5}$ be the line that goes through $p_{5}$ with slope $ -0.027 $.
Thus $\beta_5$ is simple.
\item Let $\beta_{8}$ be the line that goes through $p_{8}$ with slope $ -0.17 $.
Thus $\beta_8$ is simple.
\item Let $\beta_{9}$ be the line that goes through $p_{9}$ with slope $ -0.1763 $.
Thus $\beta_9$ is simple.
\item Let $\beta_{12}$ be the line that goes through $p_{12}$ with slope $ -0.416 $.
Thus $\beta_{12}$ is simple.
\item Let $\beta_{13}$ be the line that goes through $p_{13}$ with slope $  0.052 $.
Thus $\beta_{13}$ is simple.
\item Let $\beta_{14}$ be the line that goes through $p_{14}$ with slope $  -1.1994 $.
Thus $\beta_{14}$ is simple.
\item Let $\beta_{21}$ be the line $\theta(\beta_{14}) $.
Thus $\beta_{21}$ is simple and goes through $p_{21}$.  
\item Let $\beta_{22}$ be the line $\theta(\beta_{13}) $.
Thus $\beta_{22}$ is simple and goes through $p_{22}$.
\item Let $\beta_{	25	}$ be the line $\theta^2(\beta_{	14	}) $.
Thus $\beta_{	25	}$ is simple and goes through $p_{	25	} $.
\item Let $\beta_{	27	}$ be the line $\theta^2(\beta_{ 13	}) $.
Thus $\beta_{	27	}$ is simple and goes through $p_{	27	} $.
\item Let $\beta_{	28	}$ be the line $\theta^2(\beta_{	12	}) $.
Thus $\beta_{	28	}$ is simple and goes through $p_{	28	} $.
\item Let $\beta_{	31	}$ be the line $\theta(\beta_{	12	}) $.
Thus $\beta_{	31	}$ is simple and goes through $p_{	31	} $.
\item Let $\beta_{	34	}$ be the line $\theta^2(\beta_{	9	}) $.
Thus $\beta_{	34	}$ is simple and goes through $p_{	34	} $.
\item Let $\beta_{	35	}$ be the line $\theta(\beta_{	9	}) $.
Thus $\beta_{	35	}$ is simple and goes through $p_{	35	} $.
\item Let $\beta_{	36	}$ be the line $\theta(\beta_{	8	}) $.
Thus $\beta_{	36	}$ is simple and goes through $p_{	36	} $.
\item Let $\beta_{	37	}$ be the line $\theta^2(\beta_{	8	}) $.
Thus $\beta_{	37	}$ is simple and goes through $p_{	37	} $.
\item Let $\beta_{	42	}$ be the line $\theta^2(\beta_{	5	}) $.
Thus $\beta_{	42	}$ is simple and goes through $p_{	42	} $.
\item Let $\beta_{	43	}$ be the line $\theta(\beta_{	5	}) $.
Thus $\beta_{	43	}$ is simple and goes through $p_{	43	} $.
\item Let $\beta_{44}$ be the line $\theta(\beta_{4}) $.
Thus $\beta_{44}$ is simple and goes through $p_{44}$.  
\item Let $\beta_{45}$ be the line $\theta^2(\beta_{4}) $.
Thus $\beta_{45}$ is simple and goes through $p_{45}$.  
\item Let $\beta_{	48	}$ be the line $\theta^2(\beta_{	2	}) $.
Thus $\beta_{48}$ is simple and goes through $p_{48}$.  
\item Let $\beta_{	49	}$ be the line $\theta(\beta_{	2	}) $.
Thus $\beta_{49}$ is simple and goes through $p_{49}$.  
\item Let $\beta_{	50	}$ be the line $\theta^2(\beta_{	1	}) $.
Thus $\beta_{	50	}$ is splitting and goes through $p_{	50	} $ and $p_{48} $, directed from $p_{50}$ towards $p_{48}$.
\item Let $\beta_{	51	}$ be the line $\theta(\beta_{	1	}) $.
Thus $\beta_{	51	}$ is splitting and goes through $p_{	51	}$ and $p_{	49} $, directed from $p_{51}$ towards $p_{49}$. 

\end{enumerate}

\clearpage

\subsection{(How to construct) A drawing 
of $K_{81}$ with $618018$ crossings}

We describe how to obtain a drawing of $K_{81}$ with $618018$
crossings using the construction technique in Section~\ref{gencon}.
As explained at the end of Section~\ref{symdra},  
it suffices to give a base drawing $D_{m}$ for some
suitable $m<n$ (equivalently, the underlying point set $P_m$), 
the cluster models $S_{i}$, $i=1,\ldots,m$, and a pre--halving set of
lines $\{\beta_i\}_{i\in I}$ for those points in $P_{m}$ that get
transformed into a cluster.  In this case, we work with
a base set with $51$ points, that is, $m=51$.

These ingredients are given below. The result is a drawing of
$K_{81}$ with $618018$ crossings.

\subsubsection{The base point configuration}

We use as base configuration the $51$--point set
$P=\{p_1,p_2,\ldots,p_{51}\}$ from Section~\ref{t315}.  

\subsubsection{The cluster models}

The cluster models for those points that do not get augmented or get
augmented into a cluster of size $2$ or $3$ are trivial (any point
sets in general position work). Since all clusters in this case are of
size $1$, $2$, or $3$, the description is greatly simplified in this
case:

For $i=3 , 6 , 7 , 10 , 11 , 15 , 16 , 17 , 18 , 19 , 20 , 23 , 24 , 26 , 29 , 30 , 32 , 33 , 38 , 39 , 40 , 41 , 46$, and $47$ we let $S_i$ have one point (there is no need
to specify its coordinates, as we mentioned above).

For $i = 1, 2, 5, 8, 9, 12, 13, 14, 21, 22, 25, 27, 28, 31, 34, 35, 36, 37, 42, 43, 48, 49, 50, 51$, we let $S_i$ have two points  (there is no need
to specify its coordinates, as we mentioned above).

For $i = 4, 44, 45$, we let $S_i$ have three points  (there is no need
to specify its coordinates, as we mentioned above).

Thus, the set $I$ of those subscripts $i$ such that $s_i:=|S_i|>1$ is 
$I=\{ 1, 2, 4, 5, 8, 9, 12, 13, 14, 21, 22, 25, 27, 28, 31, 34, 35, 36, 37, 42, 43, 44, 45, 48, 49, 50, 51 \}$. 

\subsubsection{A pre--halving set of lines}

We finally define a pre--halving set of lines 
$\{\beta_i\}_{i\in I}$.

\begin{enumerate}

\item Let $\beta_{1}$ be the line that goes through $p_{1}$ with slope $ 0.0063 $.
Thus $\beta_1$ is simple.
\item Let $\beta_{2}$ be the line that goes through $p_{2}$ with slope $ 0.005 $.
Thus $\beta_2$ is simple.
\item Let $\beta_{4}$ be the line that goes through $p_{4}$ with slope $ 0.02 $.
Thus $\beta_4$ is simple.
\item Let $\beta_{5}$ be the line that goes through $p_{5}$ with slope $ -0.0288 $.
Thus $\beta_5$ is simple.
\item Let $\beta_{8}$ be the line that goes through $p_{8}$ with slope $ -0.169 $.
Thus $\beta_8$ is simple.
\item Let $\beta_{9}$ be the line that goes through $p_{9}$ with slope $ -0.25 $.
Thus $\beta_9$ is simple.
\item Let $\beta_{12}$ be the line that goes through $p_{12}$ with slope $ -0.416 $.
Thus $\beta_{12}$ is simple.
\item Let $\beta_{13}$ be the line that goes through $p_{13}$ with slope $ -1.08 $.
Thus $\beta_{13}$ is simple.
\item Let $\beta_{14}$ be the line that goes through $p_{14}$ with slope $ -1.21 $.
Thus $\beta_{14}$ is simple.
\item Let $\beta_{21}$ be the line $\theta(\beta_{14}) $.
Thus $\beta_{21}$ is simple and goes through $p_{21}$.  
\item Let $\beta_{22}$ be the line $\theta(\beta_{13}) $.
Thus $\beta_{22}$ is simple and goes through $p_{22}$.
\item Let $\beta_{	25	}$ be the line $\theta^2(\beta_{	14	}) $.
Thus $\beta_{	25	}$ is simple and goes through $p_{	25	} $.
\item Let $\beta_{	27	}$ be the line $\theta^2(\beta_{ 13	}) $.
Thus $\beta_{	27	}$ is simple and goes through $p_{	27	} $.
\item Let $\beta_{	28	}$ be the line $\theta^2(\beta_{	12	}) $.
Thus $\beta_{	28	}$ is simple and goes through $p_{	28	} $.
\item Let $\beta_{	31	}$ be the line $\theta(\beta_{	12	}) $.
Thus $\beta_{	31	}$ is simple and goes through $p_{	31	} $.
\item Let $\beta_{	34	}$ be the line $\theta^2(\beta_{	9	}) $.
Thus $\beta_{	34	}$ is simple and goes through $p_{	34	} $.
\item Let $\beta_{	35	}$ be the line $\theta(\beta_{	9	}) $.
Thus $\beta_{	35	}$ is simple and goes through $p_{	35	} $.
\item Let $\beta_{	36	}$ be the line $\theta(\beta_{	8	}) $.
Thus $\beta_{	36	}$ is simple and goes through $p_{	36	} $.
\item Let $\beta_{	37	}$ be the line $\theta^2(\beta_{	8	}) $.
Thus $\beta_{	37	}$ is simple and goes through $p_{	37	} $.
\item Let $\beta_{	42	}$ be the line $\theta^2(\beta_{	5	}) $.
Thus $\beta_{	42	}$ is simple and goes through $p_{	42	} $.
\item Let $\beta_{	43	}$ be the line $\theta(\beta_{	5	}) $.
Thus $\beta_{	43	}$ is simple and goes through $p_{	43	} $.
\item Let $\beta_{44}$ be the line $\theta(\beta_{4}) $.
Thus $\beta_{44}$ is simple and goes through $p_{44}$.  
\item Let $\beta_{45}$ be the line $\theta^2(\beta_{4}) $.
Thus $\beta_{45}$ is simple and goes through $p_{45}$.  
\item Let $\beta_{	48	}$ be the line $\theta^2(\beta_{	2	}) $.
Thus $\beta_{48}$ is simple and goes through $p_{48}$.  
\item Let $\beta_{	49	}$ be the line $\theta(\beta_{	2	}) $.
Thus $\beta_{49}$ is simple and goes through $p_{49}$.  
\item Let $\beta_{	50	}$ be the line $\theta^2(\beta_{	1	}) $.
Thus $\beta_{50}$ is simple and goes through $p_{50}$.  
\item Let $\beta_{	51	}$ be the line $\theta(\beta_{	1	}) $.
Thus $\beta_{51}$ is simple and goes through $p_{51}$.  

\end{enumerate}

\clearpage

\subsection{(How to construct) A drawing 
of $K_{84}$ with $717360$ crossings}

We describe how to obtain a drawing of $K_{84}$ with $717360$
crossings using the construction technique in Section~\ref{gencon}.
As explained at the end of Section~\ref{symdra},  
it suffices to give a base drawing $D_{m}$ for some
suitable $m<n$ (equivalently, the underlying point set $P_m$), 
the cluster models $S_{i}$, $i=1,\ldots,m$, and a pre--halving set of
lines $\{\beta_i\}_{i\in I}$ for those points in $P_{m}$ that get
transformed into a cluster.  In this case, we work with
a base set with $51$ points, that is, $m=51$.

These ingredients are given below. The result is a drawing of
$K_{84}$ with $717360$ crossings.

\subsubsection{The base point configuration}

We use as base configuration the $51$--point set
$P=\{p_1,p_2,\ldots,p_{51}\}$ from Section~\ref{t315}.  

\subsubsection{The cluster models}

The cluster models for those points that do not get augmented or get
augmented into a cluster of size $2$ or $3$ are trivial (any point
sets in general position work). Since all clusters in this case are of
size $1$, $2$, or $3$, the description is greatly simplified in this
case:

For $i=3 , 6 , 10 , 11 , 12 , 15 , 17 , 18 , 20 , 23 , 26 , 28 , 29 , 30 , 31 , 32 , 33 , 40 , 41 , 46$, and $47 $ we let $S_i$ have one point (there is no need
to specify its coordinates, as we mentioned above).

For $i = 1, 4, 5, 7, 8, 9, 13, 14, 16, 19, 21, 22, 24, 25, 27, 34, 35, 36, 37, 38, 39, 42, 43, 44, 45, 50$, and $51$, 
we let $S_i$ have two points  (there is no need
to specify its coordinates, as we mentioned above).

For $i = 2, 48, 49$, we let $S_i$ have three points  (there is no need
to specify its coordinates, as we mentioned above).

Thus, the set $I$ of those subscripts $i$ such that $s_i:=|S_i|>1$ is 
$I=\{ 1, 2, 4, 5, 7, 8, 9, 13, 14, 16, 19, 21, 22, 24, 25, 27, 34, 35, 36, 37, 38, 39, 42, 43, 44, 45, 48, 49, 50, 51\}$. 

\subsubsection{A pre--halving set of lines}

We finally define a pre--halving set of lines 
$\{\beta_i\}_{i\in I}$.

\begin{enumerate}

\item Let $\beta_1$ be the line that goes through $p_1$ and $p_2$,
  directed from $p_1$ towards $p_2$. 
\item Let $\beta_{2}$ be the line that goes through $p_{2}$ with slope $ 0.005 $.
Thus $\beta_2$ is simple.
\item Let $\beta_{4}$ be the line that goes through $p_{4}$ with slope $ 0.04 $.
Thus $\beta_4$ is simple.
\item Let $\beta_{5}$ be the line that goes through $p_{5}$ with slope $ -0.027 $.
Thus $\beta_5$ is simple.
\item Let $\beta_{7}$ be the line that goes through $p_{7}$ with slope $ -0.09 $.
Thus $\beta_7$ is simple.
\item Let $\beta_{8}$ be the line that goes through $p_{8}$ with slope $ -0.17 $.
Thus $\beta_8$ is simple.
\item Let $\beta_{9}$ be the line that goes through $p_{9}$ with slope $ -0.1763 $.
Thus $\beta_9$ is simple.
\item Let $\beta_{13}$ be the line that goes through $p_{13}$ with slope $  0.052 $.
Thus $\beta_{13}$ is simple.
\item Let $\beta_{14}$ be the line that goes through $p_{14}$ with slope $ 0.065 $.
Thus $\beta_{14}$ is simple.
\item Let $\beta_{16}$ be the line that goes through $p_{16}$ with slope $ -1.265 $.
Thus $\beta_{16}$ is simple.
\item Let $\beta_{19}$ be the line $\theta^2(\beta_{16}) $.
Thus $\beta_{19}$ is simple and goes through $p_{19}$.  
\item Let $\beta_{21}$ be the line $\theta(\beta_{14}) $.
Thus $\beta_{21}$ is simple and goes through $p_{21}$.  
\item Let $\beta_{22}$ be the line $\theta(\beta_{13}) $.
Thus $\beta_{22}$ is simple and goes through $p_{22}$.
\item Let $\beta_{	24	}$ be the line $\theta(\beta_{		16}) $.
Thus $\beta_{	24	}$ is simple and goes through $p_{	24	} $.
\item Let $\beta_{	25	}$ be the line $\theta^2(\beta_{	14	}) $.
Thus $\beta_{	25	}$ is simple and goes through $p_{	25	} $.
\item Let $\beta_{	27	}$ be the line $\theta^2(\beta_{ 13	}) $.
Thus $\beta_{	27	}$ is simple and goes through $p_{	27	} $.
\item Let $\beta_{	34	}$ be the line $\theta^2(\beta_{	9	}) $.
Thus $\beta_{	34	}$ is simple and goes through $p_{	34	} $.
\item Let $\beta_{	35	}$ be the line $\theta(\beta_{	9	}) $.
Thus $\beta_{	35	}$ is simple and goes through $p_{	35	} $.
\item Let $\beta_{	36	}$ be the line $\theta(\beta_{	8	}) $.
Thus $\beta_{	36	}$ is simple and goes through $p_{	36	} $.
\item Let $\beta_{	37	}$ be the line $\theta^2(\beta_{	8	}) $.
Thus $\beta_{	37	}$ is simple and goes through $p_{	37	} $.
\item Let $\beta_{	38	}$ be the line $\theta(\beta_{	7	}) $.
Thus $\beta_{	38	}$ is simple and goes through $p_{	38	} $.
\item Let $\beta_{	39	}$ be the line $\theta^2(\beta_{	7	}) $.
Thus $\beta_{	39	}$ is simple and goes through $p_{	39	} $.
\item Let $\beta_{	42	}$ be the line $\theta^2(\beta_{	5	}) $.
Thus $\beta_{	42	}$ is simple and goes through $p_{	42	} $.
\item Let $\beta_{	43	}$ be the line $\theta(\beta_{	5	}) $.
Thus $\beta_{	43	}$ is simple and goes through $p_{	43	} $.
\item Let $\beta_{44}$ be the line $\theta(\beta_{4}) $.
Thus $\beta_{44}$ is simple and goes through $p_{44}$.  
\item Let $\beta_{45}$ be the line $\theta^2(\beta_{4}) $.
Thus $\beta_{45}$ is simple and goes through $p_{45}$.  
\item Let $\beta_{	48	}$ be the line $\theta^2(\beta_{	2	}) $.
Thus $\beta_{48}$ is simple and goes through $p_{48}$.  
\item Let $\beta_{	49	}$ be the line $\theta(\beta_{	2	}) $.
Thus $\beta_{49}$ is simple and goes through $p_{49}$.  
\item Let $\beta_{	50	}$ be the line $\theta^2(\beta_{	1	}) $.
Thus $\beta_{	50	}$ is splitting and goes through $p_{	50	} $ and $p_{48} $, directed from $p_{50}$ towards $p_{48}$.
\item Let $\beta_{	51	}$ be the line $\theta(\beta_{	1	}) $.
Thus $\beta_{	51	}$ is splitting and goes through $p_{	51	}$ and $p_{	49} $, directed from $p_{51}$ towards $p_{49}$. 

\end{enumerate}

\clearpage

\subsection{(How to construct) A drawing 
of $K_{87}$ with $828225$ crossings}

We describe how to obtain a drawing of $K_{87}$ with $828225$
crossings using the construction technique in Section~\ref{gencon}.
As explained at the end of Section~\ref{symdra},  
it suffices to give a base drawing $D_{m}$ for some
suitable $m<n$ (equivalently, the underlying point set $P_m$), 
the cluster models $S_{i}$, $i=1,\ldots,m$, and a pre--halving set of
lines $\{\beta_i\}_{i\in I}$ for those points in $P_{m}$ that get
transformed into a cluster.  In this case, we work with
a base set with $51$ points, that is, $m=51$.

These ingredients are given below. The result is a drawing of
$K_{87}$ with $828225$ crossings.

\subsubsection{The base point configuration}

We use as base configuration the $51$--point set
$P=\{p_1,p_2,\ldots,p_{51}\}$ from Section~\ref{t315}.  

\subsubsection{The cluster models}

The cluster models for those points that do not get augmented or get
augmented into a cluster of size $2$ or $3$ are trivial (any point
sets in general position work). Since all clusters in this case are of
size $1$, $2$, or $3$, the description is greatly simplified in this
case:

For $i=7 , 10 , 11 , 16 , 17 , 18 , 19 , 23 , 24 , 29 , 30 , 32 , 33 , 38$, and $39$ we let $S_i$ have one point (there is no need
to specify its coordinates, as we mentioned above).

For $i = 1, 2, 3, 4, 5, 6, 8, 9, 12, 13, 14, 15, 20, 21, 22, 25, 26, 27, 28, 31, 34, 35, 36, 37, 40, 41, 42, 43$, 
$44, 45, 46, 47, 48, 49, 50$, and $51$, we let $S_i$ have two points  (there is no need
to specify its coordinates, as we mentioned above).

Thus, the set $I$ of those subscripts $i$ such that $s_i:=|S_i|>1$ is 
$I=\{  1, 2, 3, 4, 5, 6, 8, 9, 12, 13, 14, 15, 20, 21, 22, 25, 26, 27, 28, 31, 34, 35, 36, 37, 40, 41, 42, 43, 44, 45, 46, 47$, $48, 49, 50, 51\}$. 

\subsubsection{A pre--halving set of lines}

We finally define a pre--halving set of lines 
$\{\beta_i\}_{i\in I}$.

\begin{enumerate}

\item Let $\beta_{1}$ be the line that goes through $p_{1}$ with slope $ 0.0063 $.
Thus $\beta_1$ is simple.
\item Let $\beta_{2}$ be the line that goes through $p_{2}$ with slope $ 0.005 $.
Thus $\beta_2$ is simple.
\item Let $\beta_{3}$ be the line that goes through $p_{3}$ with slope $ 0.001 $.
Thus $\beta_3$ is simple.
\item Let $\beta_{4}$ be the line that goes through $p_{4}$ with slope $ -0.011 $.
Thus $\beta_4$ is simple.
\item Let $\beta_{5}$ be the line that goes through $p_{5}$ with slope $ -0.026 $.
Thus $\beta_5$ is simple.
\item Let $\beta_{6}$ be the line that goes through $p_{6}$ with slope $ -0.0305 $.
Thus $\beta_6$ is simple.
\item Let $\beta_{8}$ be the line that goes through $p_{8}$ with slope $ -0.17 $.
Thus $\beta_8$ is simple.
\item Let $\beta_{9}$ be the line that goes through $p_{9}$ with slope $ -0.25 $.
Thus $\beta_9$ is simple.
\item Let $\beta_{12}$ be the line that goes through $p_{12}$ with slope $ -0.416 $.
Thus $\beta_{12}$ is simple.
\item Let $\beta_{13}$ be the line that goes through $p_{13}$ with slope $ -1.08 $.
Thus $\beta_{13}$ is simple.
\item Let $\beta_{14}$ be the line that goes through $p_{14}$ with slope $ -1.21 $.
Thus $\beta_{14}$ is simple.
\item Let $\beta_{15}$ be the line that goes through $p_{15}$ with slope $ -1.262 $.
Thus $\beta_{15}$ is simple.
\item Let $\beta_{20}$ be the line $\theta(\beta_{15}) $.
Thus $\beta_{20}$ is simple and goes through $p_{20}$.  
\item Let $\beta_{21}$ be the line $\theta(\beta_{14}) $.
Thus $\beta_{21}$ is simple and goes through $p_{21}$.  
\item Let $\beta_{22}$ be the line $\theta(\beta_{13}) $.
Thus $\beta_{22}$ is simple and goes through $p_{22}$.
\item Let $\beta_{	25	}$ be the line $\theta^2(\beta_{	14	}) $.
Thus $\beta_{	25	}$ is simple and goes through $p_{	25	} $.
\item Let $\beta_{	26	}$ be the line $\theta^2(\beta_{	15	}) $.
Thus $\beta_{	26	}$ is simple and goes through $p_{	26	} $.
\item Let $\beta_{	27	}$ be the line $\theta^2(\beta_{ 13	}) $.
Thus $\beta_{	27	}$ is simple and goes through $p_{	27	} $.
\item Let $\beta_{	28	}$ be the line $\theta^2(\beta_{	12	}) $.
Thus $\beta_{	28	}$ is simple and goes through $p_{	28	} $.
\item Let $\beta_{	31	}$ be the line $\theta(\beta_{	12	}) $.
Thus $\beta_{	31	}$ is simple and goes through $p_{	31	} $.
\item Let $\beta_{	34	}$ be the line $\theta^2(\beta_{	9	}) $.
Thus $\beta_{	34	}$ is simple and goes through $p_{	34	} $.
\item Let $\beta_{	35	}$ be the line $\theta(\beta_{	9	}) $.
Thus $\beta_{	35	}$ is simple and goes through $p_{	35	} $.
\item Let $\beta_{	36	}$ be the line $\theta(\beta_{	8	}) $.
Thus $\beta_{	36	}$ is simple and goes through $p_{	36	} $.
\item Let $\beta_{	37	}$ be the line $\theta^2(\beta_{	8	}) $.
Thus $\beta_{	37	}$ is simple and goes through $p_{	37	} $.
\item Let $\beta_{	40	}$ be the line $\theta^2(\beta_{	6	}) $.
Thus $\beta_{	40	}$ is simple and goes through $p_{	40	} $.
\item Let $\beta_{	41	}$ be the line $\theta(\beta_{	6	}) $.
Thus $\beta_{	41	}$ is simple and goes through $p_{	41	} $.
\item Let $\beta_{	42	}$ be the line $\theta^2(\beta_{	5	}) $.
Thus $\beta_{	42	}$ is simple and goes through $p_{	42	} $.
\item Let $\beta_{	43	}$ be the line $\theta(\beta_{	5	}) $.
Thus $\beta_{	43	}$ is simple and goes through $p_{	43	} $.
\item Let $\beta_{44}$ be the line $\theta(\beta_{4}) $.
Thus $\beta_{44}$ is simple and goes through $p_{44}$.  
\item Let $\beta_{45}$ be the line $\theta^2(\beta_{4}) $.
Thus $\beta_{45}$ is simple and goes through $p_{45}$.  
\item Let $\beta_{46}$ be the line $\theta^2(\beta_{3}) $.
Thus $\beta_{46}$ is simple and goes through $p_{46}$.  
\item Let $\beta_{47}$ be the line $\theta(\beta_{3}) $.
Thus $\beta_{47}$ is simple and goes through $p_{47}$.  
\item Let $\beta_{	48	}$ be the line $\theta^2(\beta_{	2	}) $.
Thus $\beta_{48}$ is simple and goes through $p_{48}$.  
\item Let $\beta_{	49	}$ be the line $\theta(\beta_{	2	}) $.
Thus $\beta_{49}$ is simple and goes through $p_{49}$.  
\item Let $\beta_{	50	}$ be the line $\theta^2(\beta_{	1	}) $.
Thus $\beta_{50}$ is simple and goes through $p_{50}$.  
\item Let $\beta_{	51	}$ be the line $\theta(\beta_{	1	}) $.
Thus $\beta_{51}$ is simple and goes through $p_{51}$.  

\end{enumerate}

\clearpage

\subsection{(How to construct) A drawing 
of $K_{90}$ with $951459$ crossings}

We describe how to obtain a drawing of $K_{90}$ with $951459$
crossings using the construction technique in Section~\ref{gencon}.
As explained at the end of Section~\ref{symdra},  
it suffices to give a base drawing $D_{m}$ for some
suitable $m<n$ (equivalently, the underlying point set $P_m$), 
the cluster models $S_{i}$, $i=1,\ldots,m$, and a pre--halving set of
lines $\{\beta_i\}_{i\in I}$ for those points in $P_{m}$ that get
transformed into a cluster.  In this case, we work with
a base set with $51$ points, that is, $m=51$.

These ingredients are given below. The result is a drawing of
$K_{90}$ with $951459$ crossings.

\subsubsection{The base point configuration}

We use as base configuration the $51$--point set
$P=\{p_1,p_2,\ldots,p_{51}\}$ from Section~\ref{t315}.  

\subsubsection{The cluster models}

The cluster models for those points that do not get augmented or get
augmented into a cluster of size $2$ or $3$ are trivial (any point
sets in general position work). Since all clusters in this case are of
size $1$, $2$, or $3$, the description is greatly simplified in this
case:

For $i=3 , 6 , 10 , 11 , 15 , 16 , 19 , 20 , 24 , 26 , 29 , 30 , 32 , 33 , 40 , 41 , 46$, and $47$ we let $S_i$ have one point (there is no need
to specify its coordinates, as we mentioned above).

For $i = 1 , 5 , 7 , 8 , 9 , 12 , 13 , 14 , 17 , 18 , 21 , 22 , 23 , 25 , 27 , 28 , 31 , 34 , 35 , 36 , 37 , 38 , 39 , 42 , 43 , 50$, 
and $51$, we let $S_i$ have two points  (there is no need
to specify its coordinates, as we mentioned above).

For $i = 2, 4, 44, 45, 48, 49$, we let $S_i$ have three points  (there is no need
to specify its coordinates, as we mentioned above).

Thus, the set $I$ of those subscripts $i$ such that $s_i:=|S_i|>1$ is 
$I=\{ 1 , 2, 4, 5 , 7 , 8 , 9 , 12 , 13 , 14 , 17 , 18 , 21 , 22 , 23 , 25 , 27 , 28 , 31 , 34 , 35 , 36 , 37 , 38 , 39 , 42 , 43 , 44, 45, 48, 49$, $50 , 51 \}$. 

\subsubsection{A pre--halving set of lines}

We finally define a pre--halving set of lines 
$\{\beta_i\}_{i\in I}$.

\begin{enumerate}

\item Let $\beta_1$ be the line that goes through $p_1$ and $p_2$,
  directed from $p_1$ towards $p_2$. 
Thus $\beta_1$ is splitting.
\item Let $\beta_{2}$ be the line that goes through $p_{2}$ with slope $ 0.005 $.
Thus $\beta_2$ is simple.
\item Let $\beta_{4}$ be the line that goes through $p_{4}$ with slope $ -0.017 $.
Thus $\beta_4$ is simple.
\item Let $\beta_{5}$ be the line that goes through $p_{5}$ with slope $ -0.0288 $.
Thus $\beta_5$ is simple.
\item Let $\beta_{7}$ be the line that goes through $p_{7}$ with slope $ 0.055 $.
Thus $\beta_7$ is simple.
\item Let $\beta_{8}$ be the line that goes through $p_{8}$ with slope $ -0.17 $.
Thus $\beta_8$ is simple.
\item Let $\beta_{9}$ be the line that goes through $p_{9}$ with slope $ -0.1763 $.
Thus $\beta_9$ is simple.
\item Let $\beta_{12}$ be the line that goes through $p_{12}$ with slope $ -0.416 $.
Thus $\beta_{12}$ is simple.
\item Let $\beta_{13}$ be the line that goes through $p_{13}$ with slope $ 0.052 $.
Thus $\beta_{13}$ is simple.
\item Let $\beta_{14}$ be the line that goes through $p_{14}$ with slope $ -1.1994 $.
Thus $\beta_{14}$ is simple.
\item Let $\beta_{17}$ be the line that goes through $p_{17}$ with slope $ -1.35 $.
Thus $\beta_{17}$ is simple.
\item Let $\beta_{18}$ be the line $\theta^2(\beta_{17}) $.
Thus $\beta_{18}$ is simple and goes through $p_{18}$.  
\item Let $\beta_{21}$ be the line $\theta(\beta_{14}) $.
Thus $\beta_{21}$ is simple and goes through $p_{21}$.  
\item Let $\beta_{22}$ be the line $\theta(\beta_{13}) $.
Thus $\beta_{22}$ is simple and goes through $p_{22}$.
\item Let $\beta_{23}$ be the line $\theta(\beta_{17}) $.
Thus $\beta_{23}$ is simple and goes through $p_{23}$.
\item Let $\beta_{	25	}$ be the line $\theta^2(\beta_{	14	}) $.
Thus $\beta_{	25	}$ is simple and goes through $p_{	25	} $.
\item Let $\beta_{	27	}$ be the line $\theta^2(\beta_{ 13	}) $.
Thus $\beta_{	27	}$ is simple and goes through $p_{	27	} $.
\item Let $\beta_{	28	}$ be the line $\theta^2(\beta_{	12	}) $.
Thus $\beta_{	28	}$ is simple and goes through $p_{	28	} $.
\item Let $\beta_{	31	}$ be the line $\theta(\beta_{	12	}) $.
Thus $\beta_{	31	}$ is simple and goes through $p_{	31	} $.
\item Let $\beta_{	34	}$ be the line $\theta^2(\beta_{	9	}) $.
Thus $\beta_{	34	}$ is simple and goes through $p_{	34	} $.
\item Let $\beta_{	35	}$ be the line $\theta(\beta_{	9	}) $.
Thus $\beta_{	35	}$ is simple and goes through $p_{	35	} $.
\item Let $\beta_{	36	}$ be the line $\theta(\beta_{	8	}) $.
Thus $\beta_{	36	}$ is simple and goes through $p_{	36	} $.
\item Let $\beta_{	37	}$ be the line $\theta^2(\beta_{	8	}) $.
Thus $\beta_{	37	}$ is simple and goes through $p_{	37	} $.
\item Let $\beta_{	38	}$ be the line $\theta(\beta_{	7	}) $.
Thus $\beta_{	38	}$ is simple and goes through $p_{	38	} $.
\item Let $\beta_{	39	}$ be the line $\theta^2(\beta_{	7	}) $.
Thus $\beta_{	39	}$ is simple and goes through $p_{	39	} $.
\item Let $\beta_{	42	}$ be the line $\theta^2(\beta_{	5	}) $.
Thus $\beta_{	42	}$ is simple and goes through $p_{	42	} $.
\item Let $\beta_{	43	}$ be the line $\theta(\beta_{	5	}) $.
Thus $\beta_{	43	}$ is simple and goes through $p_{	43	} $.
\item Let $\beta_{44}$ be the line $\theta(\beta_{4}) $.
Thus $\beta_{44}$ is simple and goes through $p_{44}$.  
\item Let $\beta_{45}$ be the line $\theta^2(\beta_{4}) $.
Thus $\beta_{45}$ is simple and goes through $p_{45}$.  
\item Let $\beta_{	48	}$ be the line $\theta^2(\beta_{	2	}) $.
Thus $\beta_{48}$ is simple and goes through $p_{48}$.  
\item Let $\beta_{	49	}$ be the line $\theta(\beta_{	2	}) $.
Thus $\beta_{49}$ is simple and goes through $p_{49}$.  
\item Let $\beta_{	50	}$ be the line $\theta^2(\beta_{	1	}) $.
Thus $\beta_{	50	}$ is splitting and goes through $p_{	50	} $ and $p_{48} $, directed from $p_{50}$ towards $p_{48}$.
\item Let $\beta_{	51	}$ be the line $\theta(\beta_{	1	}) $.
Thus $\beta_{	51	}$ is splitting and goes through $p_{	51	}$ and $p_{	49} $, directed from $p_{51}$ towards $p_{49}$. 

\end{enumerate}

\clearpage

\subsection{(How to construct) A drawing 
of $K_{93}$ with $1088055 $ crossings}

We describe how to obtain a drawing of $K_{93}$ with $1088055$
crossings using the construction technique in Section~\ref{gencon}.
As explained at the end of Section~\ref{symdra},  
it suffices to give a base drawing $D_{m}$ for some
suitable $m<n$ (equivalently, the underlying point set $P_m$), 
the cluster models $S_{i}$, $i=1,\ldots,m$, and a pre--halving set of
lines $\{\beta_i\}_{i\in I}$ for those points in $P_{m}$ that get
transformed into a cluster.  In this case, we work with
a base set with $51$ points, that is, $m=51$.

These ingredients are given below. The result is a drawing of
$K_{93}$ with $1088055$ crossings.

\subsubsection{The base point configuration}

We use as base configuration the $51$--point set
$P=\{p_1,p_2,\ldots,p_{51}\}$ from Section~\ref{t315}.  

\subsubsection{The cluster models}

The cluster models for those points that do not get augmented or get
augmented into a cluster of size $2$ or $3$ are trivial (any point
sets in general position work). Since all clusters in this case are of
size $1$, $2$, or $3$, the description is greatly simplified in this
case:

For $i=3 , 7 , 10 , 11 , 12 , 16 , 19 , 24 , 28 , 29 , 30 , 31 , 32 , 33 , 38 , 39 , 46$, and $47$ we let $S_i$ have one point (there is no need
to specify its coordinates, as we mentioned above).

For $i = 2 , 5 , 6 , 8 , 13 , 14 , 15 , 17 , 18 , 20 , 21 , 22 , 23 , 25 , 26 , 27 , 36 , 37 , 40 , 41 , 42 , 43 , 48$ 
and $49$, we let $S_i$ have two points  (there is no need
to specify its coordinates, as we mentioned above).

For $i = 1, 4, 9, 34, 35, 44, 45, 50$ and $51$, we let $S_i$ have three points  (there is no need
to specify its coordinates, as we mentioned above).

Thus, the set $I$ of those subscripts $i$ such that $s_i:=|S_i|>1$ is 
$I=\{1, 2 ,4, 5 , 6 , 8 , 9, 13 , 14 , 15 , 17 , 18 , 20 , 21 , 22 , 23 , 25 , 26 , 27 , 34, 35, 36 , 37 , 40 , 41 , 42 , 43 , 44, 45, 48, 49$, 
$50, 51\}$. 

\subsubsection{A pre--halving set of lines}

We finally define a pre--halving set of lines 
$\{\beta_i\}_{i\in I}$.

\begin{enumerate}

\item Let $\beta_1$ be the line that goes through $p_1$ and $p_2$,
  directed from $p_1$ towards $p_2$. 
\item Let $\beta_{2}$ be the line that goes through $p_{2}$ with slope $ 0.0033 $.
Thus $\beta_2$ is simple.
\item Let $\beta_{4}$ be the line that goes through $p_{4}$ with slope $ -0.011 $.
Thus $\beta_4$ is simple.
\item Let $\beta_{5}$ be the line that goes through $p_{5}$ with slope $ -0.026 $.
Thus $\beta_5$ is simple.
\item Let $\beta_{6}$ be the line that goes through $p_{6}$ with slope $ -0.033 $.
Thus $\beta_6$ is simple.
\item Let $\beta_{8}$ be the line that goes through $p_{8}$ with slope $ -0.17 $.
Thus $\beta_8$ is simple.
\item Let $\beta_{9}$ be the line that goes through $p_{9}$ with slope $ -0.1763 $.
Thus $\beta_9$ is simple.
\item Let $\beta_{13}$ be the line that goes through $p_{13}$ with slope $  -1.08 $.
Thus $\beta_{13}$ is simple.
\item Let $\beta_{14}$ be the line that goes through $p_{14}$ with slope $ -1.1994 $.
Thus $\beta_{14}$ is simple.
\item Let $\beta_{15}$ be the line that goes through $p_{15}$ with slope $ -1.2591 $.
Thus $\beta_{15}$ is simple.
\item Let $\beta_{17}$ be the line that goes through $p_{17}$ with slope $-1.35 $.
Thus $\beta_{17}$ is simple.
\item Let $\beta_{18}$ be the line $\theta^2(\beta_{17}) $.
Thus $\beta_{18}$ is simple and goes through $p_{18}$.  
\item Let $\beta_{20}$ be the line $\theta(\beta_{15}) $.
Thus $\beta_{20}$ is simple and goes through $p_{20}$.  
\item Let $\beta_{21}$ be the line $\theta(\beta_{14}) $.
Thus $\beta_{21}$ is simple and goes through $p_{21}$.  
\item Let $\beta_{22}$ be the line $\theta(\beta_{13}) $.
Thus $\beta_{22}$ is simple and goes through $p_{22}$.
\item Let $\beta_{23}$ be the line $\theta(\beta_{17}) $.
Thus $\beta_{23}$ is simple and goes through $p_{23}$.
\item Let $\beta_{	25	}$ be the line $\theta^2(\beta_{	14	}) $.
Thus $\beta_{	25	}$ is simple and goes through $p_{	25	} $.
\item Let $\beta_{	26	}$ be the line $\theta^2(\beta_{	15	}) $.
Thus $\beta_{	26	}$ is simple and goes through $p_{	26	} $.
\item Let $\beta_{	27	}$ be the line $\theta^2(\beta_{ 13	}) $.
Thus $\beta_{	27	}$ is simple and goes through $p_{	27	} $.
\item Let $\beta_{	34	}$ be the line $\theta^2(\beta_{	9	}) $.
Thus $\beta_{	34	}$ is simple and goes through $p_{	34	} $.
\item Let $\beta_{	35	}$ be the line $\theta(\beta_{	9	}) $.
Thus $\beta_{	35	}$ is simple and goes through $p_{	35	} $.
\item Let $\beta_{	36	}$ be the line $\theta(\beta_{	8	}) $.
Thus $\beta_{	36	}$ is simple and goes through $p_{	36	} $.
\item Let $\beta_{	37	}$ be the line $\theta^2(\beta_{	8	}) $.
Thus $\beta_{	37	}$ is simple and goes through $p_{	37	} $.
\item Let $\beta_{	40	}$ be the line $\theta^2(\beta_{	6	}) $.
Thus $\beta_{	40	}$ is simple and goes through $p_{	40	} $.
\item Let $\beta_{	41	}$ be the line $\theta(\beta_{	6	}) $.
Thus $\beta_{	41	}$ is simple and goes through $p_{	41	} $.
\item Let $\beta_{	42	}$ be the line $\theta^2(\beta_{	5	}) $.
Thus $\beta_{	42	}$ is simple and goes through $p_{	42	} $.
\item Let $\beta_{	43	}$ be the line $\theta(\beta_{	5	}) $.
Thus $\beta_{	43	}$ is simple and goes through $p_{	43	} $.
\item Let $\beta_{44}$ be the line $\theta(\beta_{4}) $.
Thus $\beta_{44}$ is simple and goes through $p_{44}$.  
\item Let $\beta_{45}$ be the line $\theta^2(\beta_{4}) $.
Thus $\beta_{45}$ is simple and goes through $p_{45}$.  
\item Let $\beta_{	48	}$ be the line $\theta^2(\beta_{	2	}) $.
Thus $\beta_{48}$ is simple and goes through $p_{48}$.  
\item Let $\beta_{	49	}$ be the line $\theta(\beta_{	2	}) $.
Thus $\beta_{49}$ is simple and goes through $p_{49}$.  
\item Let $\beta_{	50	}$ be the line $\theta^2(\beta_{	1	}) $.
Thus $\beta_{	50	}$ is splitting and goes through $p_{	50	} $ and $p_{48} $, directed from $p_{50}$ towards $p_{48}$.
\item Let $\beta_{	51	}$ be the line $\theta(\beta_{	1	}) $.
Thus $\beta_{	51	}$ is splitting and goes through $p_{	51	}$ and $p_{	49} $, directed from $p_{51}$ towards $p_{49}$. 

\end{enumerate}

\clearpage

\subsection{(How to construct) A drawing 
of $K_{96}$ with $1238646$ crossings}

We describe how to obtain a drawing of $K_{96}$ with $1238646$
crossings using the construction technique in Section~\ref{gencon}.
As explained at the end of Section~\ref{symdra},  
it suffices to give a base drawing $D_{m}$ for some
suitable $m<n$ (equivalently, the underlying point set $P_m$), 
the cluster models $S_{i}$, $i=1,\ldots,m$, and a pre--halving set of
lines $\{\beta_i\}_{i\in I}$ for those points in $P_{m}$ that get
transformed into a cluster.  In this case, we work with
a base set with $51$ points, that is, $m=51$.

These ingredients are given below. The result is a drawing of
$K_{96}$ with $1238646$ crossings.

\subsubsection{The base point configuration}

We use as base configuration the $51$--point set
$P=\{p_1,p_2,\ldots,p_{51}\}$ from Section~\ref{t315}.  

\subsubsection{The cluster models}

The cluster models for those points that do not get augmented or get
augmented into a cluster of size $2$ or $3$ are trivial (any point
sets in general position work). Since all clusters in this case are of
size $1$, $2$, or $3$, the description is greatly simplified in this
case:

For $i= 3 , 10 , 11 , 16 , 19 , 24 , 29 , 30 , 32 , 33 , 46 , 47$ we let $S_i$ have one point (there is no need
to specify its coordinates, as we mentioned above).

For $i = 2 , 5 , 6 , 7 , 8 , 9 , 12 , 13 , 14 , 15 , 17 , 18 , 20 , 21 , 22 , 23 , 25 , 26 , 27 , 28 , 31 , 34 , 35 , 36 , 37 , 38 , 39 , 40$, 
$41 , 42 , 43 , 48 , 49$, we let $S_i$ have two points  (there is no need
to specify its coordinates, as we mentioned above).

For $i = 1, 4, 44, 45, 50, 51 $, we let $S_i$ have three points  (there is no need
to specify its coordinates, as we mentioned above).

Thus, the set $I$ of those subscripts $i$ such that $s_i:=|S_i|>1$ is 
$I=\{1, 2 , 4, 5 , 6 , 7 , 8 , 9 , 12 , 13 , 14 , 15 , 17 , 18 , 20 , 21 , 22 , 23 , 25 , 26 , 27 , 28 , 31 , 34 , 35 , 36 , 37 , 38 , 39 , 40 , 41 , 42$, 
$43 , 44, 45, 48 , 49 , 50, 51 \}$. 

\subsubsection{A pre--halving set of lines}

We finally define a pre--halving set of lines 
$\{\beta_i\}_{i\in I}$.

\begin{enumerate}

\item Let $\beta_{1}$ be the line that goes through $p_{1}$ with slope $ 0.0063 $.
Thus $\beta_1$ is simple.
\item Let $\beta_2$ be the line that goes through $p_2$ and $p_4$,
  directed from $p_2$ towards $p_4$. 
\item Let $\beta_{4}$ be the line that goes through $p_{4}$ with slope $-0.02$.
Thus $\beta_4$ is simple.
\item Let $\beta_{5}$ be the line that goes through $p_{5}$ with slope $-0.0072$.
Thus $\beta_5$ is simple.
\item Let $\beta_{6}$ be the line that goes through $p_{6}$ with slope $-0.0072$.
Thus $\beta_6$ is simple.
\item Let $\beta_{7}$ be the line that goes through $p_{7}$ with slope $ 0.0656 $.
Thus $\beta_7$ is simple.
\item Let $\beta_{8}$ be the line that goes through $p_{8}$ with slope $ -0.17 $.
Thus $\beta_8$ is simple.
\item Let $\beta_{9}$ be the line that goes through $p_{9}$ with slope $ -0.1763 $.
Thus $\beta_9$ is simple.
\item Let $\beta_{12}$ be the line that goes through $p_{12}$ with slope $ -0.052668 $.
Thus $\beta_{12}$ is simple.
\item Let $\beta_{13}$ be the line that goes through $p_{13}$ with slope $  0.052 $.
Thus $\beta_{13}$ is simple.
\item Let $\beta_{14}$ be the line that goes through $p_{14}$ with slope $ -1.1994 $.
Thus $\beta_{14}$ is simple.
\item Let $\beta_{15}$ be the line that goes through $p_{15}$ with slope $ -1.2591 $.
Thus $\beta_{15}$ is simple.
\item Let $\beta_{17}$ be the line that goes through $p_{17}$ with slope $-1.35028010$.
Thus $\beta_{17}$ is simple.
\item Let $\beta_{18}$ be the line $\theta^2(\beta_{17}) $.
Thus $\beta_{18}$ is simple and goes through $p_{18}$.  
\item Let $\beta_{20}$ be the line $\theta(\beta_{15}) $.
Thus $\beta_{20}$ is simple and goes through $p_{20}$.  
\item Let $\beta_{21}$ be the line $\theta(\beta_{14}) $.
Thus $\beta_{21}$ is simple and goes through $p_{21}$.  
\item Let $\beta_{22}$ be the line $\theta(\beta_{13}) $.
Thus $\beta_{22}$ is simple and goes through $p_{22}$.
\item Let $\beta_{23}$ be the line $\theta(\beta_{17}) $.
Thus $\beta_{23}$ is simple and goes through $p_{23}$.
\item Let $\beta_{	25	}$ be the line $\theta^2(\beta_{	14	}) $.
Thus $\beta_{	25	}$ is simple and goes through $p_{	25	} $.
\item Let $\beta_{	26	}$ be the line $\theta^2(\beta_{	15	}) $.
Thus $\beta_{	26	}$ is simple and goes through $p_{	26	} $.
\item Let $\beta_{	27	}$ be the line $\theta^2(\beta_{ 13	}) $.
Thus $\beta_{	27	}$ is simple and goes through $p_{	27	} $.
\item Let $\beta_{	28	}$ be the line $\theta^2(\beta_{	12	}) $.
Thus $\beta_{	28	}$ is simple and goes through $p_{	28	} $.
\item Let $\beta_{	31	}$ be the line $\theta(\beta_{	12	}) $.
Thus $\beta_{	31	}$ is simple and goes through $p_{	31	} $.
\item Let $\beta_{	34	}$ be the line $\theta^2(\beta_{	9	}) $.
Thus $\beta_{	34	}$ is simple and goes through $p_{	34	} $.
\item Let $\beta_{	35	}$ be the line $\theta(\beta_{	9	}) $.
Thus $\beta_{	35	}$ is simple and goes through $p_{	35	} $.
\item Let $\beta_{	36	}$ be the line $\theta(\beta_{	8	}) $.
Thus $\beta_{	36	}$ is simple and goes through $p_{	36	} $.
\item Let $\beta_{	37	}$ be the line $\theta^2(\beta_{	8	}) $.
Thus $\beta_{	37	}$ is simple and goes through $p_{	37	} $.
\item Let $\beta_{	38	}$ be the line $\theta(\beta_{	7	}) $.
Thus $\beta_{	38	}$ is simple and goes through $p_{	38	} $.
\item Let $\beta_{	39	}$ be the line $\theta^2(\beta_{	7	}) $.
Thus $\beta_{	39	}$ is simple and goes through $p_{	39	} $.
\item Let $\beta_{	40	}$ be the line $\theta^2(\beta_{	6	}) $.
Thus $\beta_{	40	}$ is simple and goes through $p_{	40	} $.
\item Let $\beta_{	41	}$ be the line $\theta(\beta_{	6	}) $.
Thus $\beta_{	41	}$ is simple and goes through $p_{	41	} $.
\item Let $\beta_{	42	}$ be the line $\theta^2(\beta_{	5	}) $.
Thus $\beta_{	42	}$ is simple and goes through $p_{	42	} $.
\item Let $\beta_{	43	}$ be the line $\theta(\beta_{	5	}) $.
Thus $\beta_{	43	}$ is simple and goes through $p_{	43	} $.
\item Let $\beta_{44}$ be the line $\theta(\beta_{4}) $.
Thus $\beta_{44}$ is simple and goes through $p_{44}$.  
\item Let $\beta_{45}$ be the line $\theta^2(\beta_{4}) $.
Thus $\beta_{45}$ is simple and goes through $p_{45}$.  
\item Let $\beta_{	48	}$ be the line $\theta^2(\beta_{	2	}) $.
Thus $\beta_{	48	}$ is splitting and goes through $p_{	45 } $ and $p_{48} $, directed from $p_{48}$ towards $p_{45}$.
\item Let $\beta_{	49	}$ be the line $\theta(\beta_{	2	}) $.
Thus $\beta_{	49	}$ is splitting and goes through $p_{	44	}$ and $p_{	49} $, directed from $p_{49}$ towards $p_{44}$. 
\item Let $\beta_{	50	}$ be the line $\theta^2(\beta_{	1	}) $.
Thus $\beta_{50}$ is simple and goes through $p_{50}$.  
\item Let $\beta_{	51	}$ be the line $\theta(\beta_{	1	}) $.
Thus $\beta_{51}$ is simple and goes through $p_{51}$.  

\end{enumerate}

\clearpage

\subsection{(How to construct) A drawing 
of $K_{99}$ with $1404552$ crossings}

We describe how to obtain a drawing of $K_{99}$ with $1404552$
crossings using the construction technique in Section~\ref{gencon}.
As explained at the end of Section~\ref{symdra},  
it suffices to give a base drawing $D_{m}$ for some
suitable $m<n$ (equivalently, the underlying point set $P_m$), 
the cluster models $S_{i}$, $i=1,\ldots,m$, and a pre--halving set of
lines $\{\beta_i\}_{i\in I}$ for those points in $P_{m}$ that get
transformed into a cluster.  In this case, we work with
a base set with $51$ points, that is, $m=51$.

These ingredients are given below. The result is a drawing of
$K_{99}$ with $1404552$ crossings.

\subsubsection{The base point configuration}

We use as base configuration the $51$--point set
$P=\{p_1,p_2,\ldots,p_{51}\}$ from Section~\ref{t315}.  

\subsubsection{The cluster models}

The cluster models for those points that do not get augmented or get
augmented into a cluster of size $2$ or $3$ are trivial (any point
sets in general position work). Since all clusters in this case are of
size $1$, $2$, or $3$, the description is greatly simplified in this
case:

For $i= 6 , 10 , 11 , 16 , 19 , 24 , 29 , 30 , 32 , 33 , 40 , 41$ we let $S_i$ have one point (there is no need
to specify its coordinates, as we mentioned above).

For $i = 1 , 2 , 3 , 7 , 8 , 12 , 13 , 14 , 15 , 17 , 18 , 20 , 21 , 22 , 23 , 25 , 26 , 27 , 28 , 31 , 36 , 37 , 38 , 39 , 46 , 47 , 48 , 49$, $50$ and $51$, we let $S_i$ have two points  (there is no need
to specify its coordinates, as we mentioned above).

For $i = 4, 5, 9, 34, 35, 42, 43, 44, 45$, we let $S_i$ have three points  (there is no need
to specify its coordinates, as we mentioned above).

Thus, the set $I$ of those subscripts $i$ such that $s_i:=|S_i|>1$ is 
$I=\{ 1 , 2 , 3 , 4, 5, 7 , 8 , 9, 12 , 13 , 14 , 15 , 17 , 18 , 20 , 21 , 22 , 23 , 25 , 26 , 27 , 28 , 31 , 34, 35, 36 , 37 , 38 , 39 , 42, 43, 44$, 
$45, 46 , 47 , 48 , 49 , 50 , 51\}$. 

\subsubsection{A pre--halving set of lines}

We finally define a pre--halving set of lines 
$\{\beta_i\}_{i\in I}$.

\begin{enumerate}

\item Let $\beta_{1}$ be the line that goes through $p_{1}$ with slope $ 0.0063 $.
Thus $\beta_1$ is simple.
\item Let $\beta_{2}$ be the line that goes through $p_{2}$ with slope $ 0.005 $.
Thus $\beta_2$ is simple.
\item Let $\beta_{3}$ be the line that goes through $p_{3}$ with slope $0.001 $.
Thus $\beta_3$ is simple.
\item Let $\beta_4$ be the line that goes through $p_4$ and $p_7$,
  directed from $p_4$ towards $p_7$. 
\item Let $\beta_{5}$ be the line that goes through $p_{5}$ with slope $ -0.026 $.
Thus $\beta_5$ is simple.
\item Let $\beta_{7}$ be the line that goes through $p_{7}$ with slope $ 0.055 $.
Thus $\beta_7$ is simple.
\item Let $\beta_{8}$ be the line that goes through $p_{8}$ with slope $ -0.17 $.
Thus $\beta_8$ is simple.
\item Let $\beta_{9}$ be the line that goes through $p_{9}$ with slope $ -0.1763 $.
Thus $\beta_9$ is simple.
\item Let $\beta_{12}$ be the line that goes through $p_{12}$ with slope $ -0.416 $.
Thus $\beta_{12}$ is simple.
\item Let $\beta_{13}$ be the line that goes through $p_{13}$ with slope $ -1.08 $.
Thus $\beta_{13}$ is simple.
\item Let $\beta_{14}$ be the line that goes through $p_{14}$ with slope $ -1.1994 $.
Thus $\beta_{14}$ is simple.
\item Let $\beta_{15}$ be the line that goes through $p_{15}$ with slope $ -1.2591 $.
Thus $\beta_{15}$ is simple.
\item Let $\beta_{17}$ be the line that goes through $p_{17}$ with slope $-1.35$.
Thus $\beta_{17}$ is simple.
\item Let $\beta_{18}$ be the line $\theta^2(\beta_{17}) $.
Thus $\beta_{18}$ is simple and goes through $p_{18}$.  
\item Let $\beta_{20}$ be the line $\theta(\beta_{15}) $.
Thus $\beta_{20}$ is simple and goes through $p_{20}$.  
\item Let $\beta_{21}$ be the line $\theta(\beta_{14}) $.
Thus $\beta_{21}$ is simple and goes through $p_{21}$.  
\item Let $\beta_{22}$ be the line $\theta(\beta_{13}) $.
Thus $\beta_{22}$ is simple and goes through $p_{22}$.
\item Let $\beta_{23}$ be the line $\theta(\beta_{17}) $.
Thus $\beta_{23}$ is simple and goes through $p_{23}$.
\item Let $\beta_{	25	}$ be the line $\theta^2(\beta_{	14	}) $.
Thus $\beta_{	25	}$ is simple and goes through $p_{	25	} $.
\item Let $\beta_{	26	}$ be the line $\theta^2(\beta_{	15	}) $.
Thus $\beta_{	26	}$ is simple and goes through $p_{	26	} $.
\item Let $\beta_{	27	}$ be the line $\theta^2(\beta_{ 13	}) $.
Thus $\beta_{	27	}$ is simple and goes through $p_{	27	} $.
\item Let $\beta_{	28	}$ be the line $\theta^2(\beta_{	12	}) $.
Thus $\beta_{	28	}$ is simple and goes through $p_{	28	} $.
\item Let $\beta_{	31	}$ be the line $\theta(\beta_{	12	}) $.
Thus $\beta_{	31	}$ is simple and goes through $p_{	31	} $.
\item Let $\beta_{	34	}$ be the line $\theta^2(\beta_{	9	}) $.
Thus $\beta_{	34	}$ is simple and goes through $p_{	34	} $.
\item Let $\beta_{	35	}$ be the line $\theta(\beta_{	9	}) $.
Thus $\beta_{	35	}$ is simple and goes through $p_{	35	} $.
\item Let $\beta_{	36	}$ be the line $\theta(\beta_{	8	}) $.
Thus $\beta_{	36	}$ is simple and goes through $p_{	36	} $.
\item Let $\beta_{	37	}$ be the line $\theta^2(\beta_{	8	}) $.
Thus $\beta_{	37	}$ is simple and goes through $p_{	37	} $.
\item Let $\beta_{	38	}$ be the line $\theta(\beta_{	7	}) $.
Thus $\beta_{	38	}$ is simple and goes through $p_{	38	} $.
\item Let $\beta_{	39	}$ be the line $\theta^2(\beta_{	7	}) $.
Thus $\beta_{	39	}$ is simple and goes through $p_{	39	} $.
\item Let $\beta_{	42	}$ be the line $\theta^2(\beta_{	5	}) $.
Thus $\beta_{	42	}$ is simple and goes through $p_{	42	} $.
\item Let $\beta_{	43	}$ be the line $\theta(\beta_{	5	}) $.
Thus $\beta_{	43	}$ is simple and goes through $p_{	43	} $.

\item Let $\beta_{	44	}$ be the line $\theta(\beta_{	4	}) $.
Thus $\beta_{	44	}$ is splitting and goes through $p_{	44	} $ and $p_{38} $, directed from $p_{44}$ towards $p_{38}$.
\item Let $\beta_{	45	}$ be the line $\theta^2(\beta_{	4	}) $.
Thus $\beta_{	45	}$ is splitting and goes through $p_{	45	} $ and $p_{39} $, directed from $p_{45}$ towards $p_{39}$.

\item Let $\beta_{46}$ be the line $\theta^2(\beta_{3}) $.
Thus $\beta_{46}$ is simple and goes through $p_{46}$.  
\item Let $\beta_{47}$ be the line $\theta(\beta_{3}) $.
Thus $\beta_{47}$ is simple and goes through $p_{47}$.  
\item Let $\beta_{	48	}$ be the line $\theta^2(\beta_{	2	}) $.
Thus $\beta_{48}$ is simple and goes through $p_{48}$.  
\item Let $\beta_{	49	}$ be the line $\theta(\beta_{	2	}) $.
Thus $\beta_{49}$ is simple and goes through $p_{49}$.  
\item Let $\beta_{	50	}$ be the line $\theta^2(\beta_{	1	}) $.
Thus $\beta_{50}$ is simple and goes through $p_{50}$.  
\item Let $\beta_{	51	}$ be the line $\theta(\beta_{	1	}) $.
Thus $\beta_{51}$ is simple and goes through $p_{51}$.  

\end{enumerate}

\end{document}